\documentclass[dvipsnames,11pt,reqno]{amsart}

\usepackage{amssymb,mathtools,amsthm}
\usepackage[a4paper,left=3cm,right=3cm,top=3cm,bottom=3cm]{geometry}
\usepackage[T1]{fontenc}
\usepackage[utf8]{inputenc}
\usepackage{microtype}
\usepackage{enumitem}
\usepackage{caption}
\usepackage{dsfont}

\setenumerate[1]{label={\arabic*.},ref={\arabic*}}
\mathtoolsset{showonlyrefs}

\renewcommand\dagger\diamond
\usepackage{xargs}					  % Use more than one optional parameter in a new commands
\usepackage[colorinlistoftodos,prependcaption,textsize=tiny]{todonotes}
\newcommandx\work[2][1=]{\todo[linecolor=RoyalBlue,backgroundcolor=RoyalBlue!25,bordercolor=RoyalBlue,#1]{{\scshape todo} #2}}
\newcommandx\comment[2][1=]{\todo[linecolor=OliveGreen,backgroundcolor=OliveGreen!25,bordercolor=OliveGreen,#1]{{\scshape comment} #2}}
\newcommandx\mistake[2][1=]{\todo[linecolor=red,backgroundcolor=red!25,bordercolor=red,#1]{{\scshape mistake} #2}}
\newcommandx\improve[2][1=]{\todo[linecolor=orange,backgroundcolor=orange!25,bordercolor=orange,#1]{{\scshape improve} #2}}
\newcommandx\change[2][1=]{\todo[linecolor=yellow,backgroundcolor=yellow!25,bordercolor=yellow,#1]{{\scshape changed} #2}}
\newcommandx\mem[2][1=]{\todo[linecolor=orange,backgroundcolor=orange!25,bordercolor=orange,#1]{{\scshape mem} #2}}

\newcounter{n}
\newcounter{bigresults}
\numberwithin{n}{section}
\theoremstyle{plain}
\newtheorem{claim}{Claim}
\newtheorem{lemma}[n]{Lemma}
\newtheorem{theorem}[bigresults]{Theorem}
\newtheorem*{theorem*}{Theorem}

\newtheorem{proposition}[n]{Proposition}
\newtheorem{corollary}[n]{Corollary}
\theoremstyle{definition}
\newtheorem{definition}[n]{Definition}
\newtheorem*{definition*}{Definition}
\newtheorem{remark}[n]{Remark}
\newtheorem*{remark*}{Remark}

\newtheorem*{claim*}{Claim}
\newtheorem*{assertion*}{Assertion}
\newtheorem*{proposition*}{Proposition}

\theoremstyle{remark}

\usepackage{tikz}
\usetikzlibrary{arrows.meta}
\usepackage{subcaption}

\usepackage{hyperref}
\definecolor{colorlinks}{RGB}{0, 24, 168}
\definecolor{colorcites}{RGB}{124, 10, 2}
\hypersetup{
	colorlinks=true,
	linkcolor=colorlinks,
	citecolor=colorcites,
	urlcolor=colorlinks,
	pdfborder={0 0 0}
}

\renewcommand\varphi\phi
\renewcommand\epsilon\varepsilon
\newcommand\eps\varepsilon
\renewcommand\setminus\smallsetminus

\renewcommand\O[1]{\(\mathrm{O}(#1)\)}

\newcommand\T{\mathbb T}

\newcommand\E{\mathbb E}

\newcommand\R{\mathbb R}

\renewcommand\P{\mathbb P}
\newcommand\Z{\mathbb Z}
\renewcommand\H{\mathbb H}

\newcommand\N{\mathbb N}

\newcommand\calC{\mathcal C}
\newcommand\calD{\mathcal D}

\newcommand\calH{\mathcal H}

\newcommand\calL{\mathcal L}
\newcommand\calP{\mathcal P}
\newcommand\calV{\mathcal V}

\newcommand\ind[1]{\mathds{1}_{\{#1\}}}

\newcommand\muLoop{{\sf Loop}}
\newcommand\muLip{{\sf Lip}}
\newcommand\muSpin{{\sf Spin}}
\newcommand\muHom{{\sf Hom}}
\newcommand\hexlattice{{\mathbb H}}
\newcommand\spaceLoop[1]{\mathfrak{S}_{\muLoop}(#1)}
\newcommand\spaceLip{\mathfrak{S}_{\muLip}}
\newcommand\spaceSpin{\mathfrak{S}_{\muSpin}}

\newcommand\Zising{Z_{\text{Ising}}}

\newcommand\squarelattice{{\mathbb L}}

\newcommand\face{F}
\newcommand\faces[1]{\face(#1)}
\newcommand\domain\Omega

\newcommand\vertices[1]{{V(#1)}}

\newcommand\edges[1]{{E(#1)}}
\newcommand\edgesA{E_a} %a-diaognals
\newcommand\edgesB{E_b} %b-diagonals
\newcommand\spaceHom{\mathfrak{S}_{\muHom}}

\newcommand\Circ{\mathsf{Circ}}
\newcommand\Rect{\mathsf{Rect}}
\newcommand\Cyl{\mathsf{Cyl}}
\newcommand\Strip{\mathsf{Strip}}
\renewcommand\varnothing\emptyset
\newcommand\intvert{V_{\mathrm{int}}}

\newcommand{\xlra}{\xleftrightarrow} %for path connections
\newcommand\concel[2]{\ooalign{$\hfil#1\mkern0mu/\hfil$\crcr$#1#2$}}  % used below
\newcommand\nxlra[1]{\mathrel{\mathpalette\concel{\xlra{#1}}}} %for path anti-connections

\newcommand\upvert{\mathrm{Y}} %set of vertices corresponding to upward-oriented triangles
\newcommand\triedges{\triangle}	%set of edges in (upward-oriented) triangles around vertices in a given set

\usepackage{tikz}
\usepackage{adjustbox}

\newcommand{\black}{{\bullet}}
\newcommand{\white}{{\circ}}
\newcommand{\mblack}{{\bullet}}
\newcommand{\blackp}{%
	{
	\begin{adjustbox}{trim= 0 0 0 0}
	\begin{tikzpicture}[scale=0.1]
		\draw [fill=black, black, line width=0.02em]
		(-0.75,-0.25)  rectangle (0.75,0.25);
		\draw [fill=black, black, line width=0.02em]
		(-0.25,-0.75)  rectangle (0.25,0.75);
	\end{tikzpicture}
	\end{adjustbox}
	}
}
\newcommand{\blackm}{%
	{
	\begin{adjustbox}{trim= 0 0 0 0}
	\begin{tikzpicture}[scale=0.1]
		\draw [fill=black, black, line width=0.02em]
		(-0.75,-0.25)  rectangle (0.75,0.25);
		\draw [fill=white, white, line width=0.02em]
		(-0.75,-0.75)  rectangle (0.75,-0.65);
	\end{tikzpicture}
	\end{adjustbox}
	}
}
\newcommand{\whitep}{%
	{
	\begin{adjustbox}{trim= 0 0 0 0}
	\begin{tikzpicture}[scale=0.1]
		\draw[black,line width=0.02em]
		(-0.25,-0.75) -- (-0.25,-0.25)
		(0.25,-0.75) -- (0.25,-0.25)
		(-0.25,-0.75)  -- (0.25,-0.75)
		(-0.25,0.75) -- (-0.25,0.25)
		(0.25,0.75) -- (0.25,0.25)
		(-0.25,0.75)  -- (0.25,0.75)
		(-0.75,-0.25) -- (-0.25,-0.25)
		(-0.75,0.25) -- (-0.25,0.25)
		(-0.75,-0.25)  -- (-0.75,0.25)
		(0.75,-0.25) -- (0.25,-0.25)
		(0.75,0.25) -- (0.25,0.25)
		(0.75,-0.25)  -- (0.75,0.25);
	\end{tikzpicture}
	\end{adjustbox}
	}
}
\newcommand{\whitem}{%
	{
	\begin{adjustbox}{trim= 0 0 0 0}
	\begin{tikzpicture}[scale=0.1]
		\draw[black,line width=0.02em]
		(-0.75,-0.25) -- (0.75,-0.25)
		(-0.75,0.25) -- (0.75,0.25)
		(-0.75,-0.25)  -- (-0.75,0.25)
		(0.75,-0.25)  -- (0.75,0.25);
		\draw [fill=white, white, line width=0.02em]
		(-0.75,-0.75)  rectangle (0.75,-0.65);
	\end{tikzpicture}
	\end{adjustbox}
	}
}

\newcommand\Zfree{Z^{\operatorname{free}}}
\newcommand\Zwired{Z^{\operatorname{wired}}}

\newcommand\phifree{\phi^{\operatorname{free}}}
\newcommand\phiwired{\phi^{\operatorname{wired}}}
\newcommand\etawired{\eta^{\operatorname{wired}}}
\newcommand\pselfdual{p_{\operatorname{sd}}}

\newcommand\blank{\,\cdot\,}

\newcommand\Var{\operatorname{Var}}

\renewcommand\subset\subseteq

\author{Alexander Glazman}
\address{Universität Innsbruck, Innsbruck, Austria}
\email{alexander.glazman@uibk.ac.at}
\author{Piet Lammers}
\address{Institut des Hautes \'Etudes Scientifiques, Bures-sur-Yvette, France}
\email{lammers@ihes.fr}

\title[Delocalisation and continuity in 2D]{Delocalisation and continuity in 2D: loop~\O{2}, six-vertex, and random-cluster models}

\makeatletter
\@namedef{subjclassname@2020}{\textup{2020} Mathematics Subject Classification}
\makeatother
\begin{document}

\maketitle

\begin{abstract}
We prove the existence of macroscopic loops in the loop~\O{2} model with
$\frac12\leq x^2\leq 1$ or, equivalently, delocalisation of the associated
integer-valued Lipschitz function on the triangular lattice. This settles one
side of the conjecture of Fan, Domany, and Nienhuis (1970s--80s) that~$x^2 =
\frac12$ is the critical point.

We also prove delocalisation in the six-vertex model with $0<a,\,b\leq c\leq
a+b$.  This yields a new proof of continuity of the phase transition in the
random-cluster and Potts models in two dimensions for $1\leq q\leq 4$ relying
neither on integrability tools (parafermionic observables, Bethe Ansatz), nor on
the Russo--Seymour--Welsh theory.

Our approach goes through a novel FKG property required for the non-coexistence
theorem of Zhang and Sheffield, which is used to prove delocalisation all the
way up to the critical point.  We also use the $\T$-circuit argument in the case
of the six-vertex model.

Finally, we extend an existing renormalisation inequality in order to quantify
the delocalisation as being logarithmic, in the regimes~$\frac12\leq x^2\leq 1$
and~$a=b\leq c\leq a+b$.  This is consistent with the conjecture that the
scaling limit is the Gaussian free field.

\end{abstract}

\setcounter{tocdepth}{1}
\tableofcontents

\section{Introduction}

\subsection{Preface}

\emph{Phase transitions} are natural phenomena in which a small change in an external
parameter causes a dramatic change in the qualitative structure of the object.
Mathematical models for the analysis of phase transitions in statistical
mechanics have been studied for more than 100 years, with the purpose of
unifying physical intuition with mathematical formalism. The mathematical
methods used for this can be divided into \emph{integrable} and \emph{probabilistic}.
Integrable methods (e.g., Bethe Ansatz, parafermionic observable) give a precise quantitative description of the system but are
not robust as they rely on identities that hold only in a small number
of special models at specific parameters.
In this work, we use only
probabilistic methods which rely on fundamental correlation inequalities, percolation theory, and the planar duality, and which
are more robust.

Our main goal is to establish the \emph{localisation-delocalisation phase transition} in
two models of integer-valued height functions. We focus on dimension two where
this transition is linked to \emph{the continuity-discontinuity transition}~\cite{DumGagHar16} in the
random-cluster model, the Berezinskii--Kosterlitz--Thouless transition~\cite{Ber71,KosTho73,FroSpe81} in the XY
model, and others. Planar height functions are thus at the crossroads of these
interesting phenomena. Height functions are expected to be
localised in all higher dimensions.

The seminal Peierls argument~\cite{Pei36} often implies the existence of a localised phase, and therefore
proving phase transition is more or less equivalent to establishing the
existence of a delocalised phase.
Several delocalisation results have been
derived in the last three decades~\cite{Ken97,She05,DumGlaPel21,GlaMan21,ChaPelSheTas21,Lam21deloc,LamOtt21,Lis21,DumKarManOul20} (we discuss these works later in further detail).

The current work develops a unified probabilistic argument for delocalisation in
the loop~\O{2} and six-vertex models. Our proof applies in the entire
delocalisation regime, under a convexity assumption on the potential. To the
best of our knowledge, this is the first time that the (conjectured) transition
point is reached through a probabilistic percolation-planarity argument. In
fact, the loop~\O{2} model is not believed to be integrable away from its
critical point. Our argument uses a percolation structure that reveals the
hidden planarity of the interaction potential at and above the critical point.

Finally, we use the Baxter--Kelland--Wu (BKW) coupling~\cite{BaxKelWu76} to derive continuity of the
phase transition in the random-cluster model from the delocalisation of the
six-vertex model. Our proof of continuity is conceptually different from the
original argument~\cite{DumSidTas17} as we do not rely on the parafermionic observable, nor on the
Bethe Ansatz.

\subsection{Informal statement of the main results}

This article concerns three models on two-dimensional lattices:
the \emph{loop~\O{n}} model, the \emph{six-vertex} model,
 and the \emph{FK percolation} or \emph{random-cluster} model.

\subsubsection{Delocalisation in the loop~\O{2} model}

The loop~\O{n} model is a model supported on loop configurations on the
hexagonal lattice (Fig.~\ref{fig:smallsamples}) and has two real parameters:
a {\em loop weight} $n>0$ and an {\em edge weight} $x>0$.  At~$n=2$, these loops
may be viewed as level lines of an integer-valued Lipschitz function on the dual
triangular lattice.  We prove that this height function delocalises whenever
$\frac12\leq x^2\leq 1$: the variance of the
height at a given vertex tends to infinity as boundary conditions are taken
further and further away; on the loop side, this means that the number of loops
surrounding a fixed hexagon tends to infinity (Theorem~\ref{thm:loop_soft_deloc}).
When $x^2=\frac12$, this was
proved in~\cite{DumGlaPel21} and, when $x^2=1$, this was proved
in~\cite{GlaMan21} (Fig.~\ref{fig:phase-diagram}).  Fan, Domany, and Nienhuis conjectured that the point
$x^2=\frac12$ is critical~\cite{Fan72,DomMukNie81,Nie82}.
Thus, we settle one half of
this conjecture (in the ferromagnetic regime $x\leq 1$).

\subsubsection{Delocalisation in the six-vertex model}

The six-vertex model is supported on edge orientations of the square lattice
that satisfy the {\em ice rule}: every vertex has two incoming and two outgoing
edges (Fig.~\ref{fig:smallsamples}).  Such an edge orientation may be interpreted as the gradient
of a certain {\em graph homomorphism} from $\Z^2$ to~$\Z$: an integer-valued
height function on the faces of the square lattice that differs by one at any
two adjacent faces.  There are three real parameters $a,\, b,\, c>0$ which describe
the weight of each vertex depending on the orientation of the edges incident to
it (Fig.~\ref{fig:six_vertex}).  We prove that the height function delocalises whenever $0<a,\,b\leq c\leq
a+b$ (Theorem~\ref{thm:hard_six_deloc}).  This was known previously under
certain additional assumptions on the parameters: the symmetric case~$a=b\leq
c\leq a+b$ was treated in~\cite{DumKarManOul20} and the case~$\sqrt{a^2+b^2 + ab}
\leq  c \leq a+b$ is treated in~\cite{DumKarKraManOul20} (the latter relies on the BKW correspondence to the random-cluster model with~$1\leq q\leq 4$).
Both works rely on integrability methods as well as on the Russo--Seymour--Welsh theory~--- we do not
require either of those for the qualitative delocalisation.

\begin{figure}
	\begin{subfigure}{0.3\textwidth}
		\centering
		\includegraphics{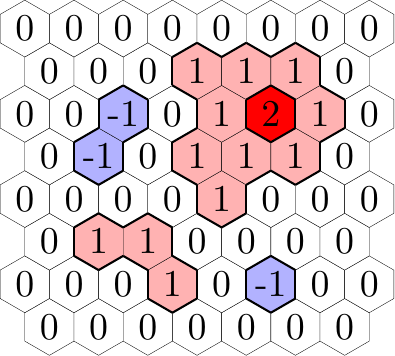}
		\subcaption*{The loop~\O{2} model}
	\end{subfigure}
	\begin{subfigure}{0.3\textwidth}
		\centering
		\includegraphics{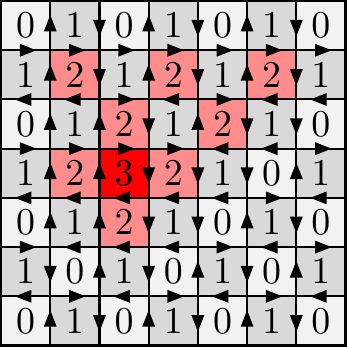}
		\subcaption*{The six-vertex model}
	\end{subfigure}
	\qquad
	\begin{subfigure}{0.3\textwidth}
		\centering
		\includegraphics{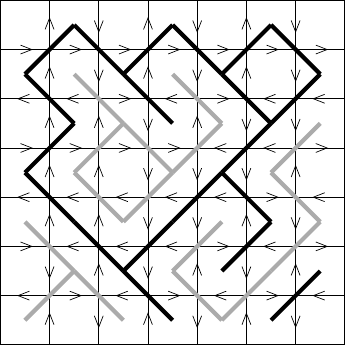}
		\subcaption*{The random-cluster model}
	\end{subfigure}
	\caption{\textsc{Left}: A sample from the loop~\O{2} model and its associated Lipschitz function.
	\textsc{Middle}: A sample from the six-vertex model and the associated graph homomorphism.
	\textsc{Right}: A sample from the random-cluster model. Primal clusters are black,
	dual clusters are grey. This model is coupled to the six-vertex model (through a complex-valued measure).}
	\label{fig:smallsamples}
\end{figure}

\subsubsection{Continuity of the phase transition in the planar random-cluster model}

The random-cluster model generalises independent (Bernoulli) bond percolation.
It is supported on spanning subgraphs of a given graph and has two real parameters: the \emph{edge weight} $p\in[0,1]$,
and the \emph{cluster weight} $q>0$.
We consider the random-cluster model on the two-dimensional square lattice
(Fig.~\ref{fig:smallsamples}).
Using our delocalisation result for the
six-vertex model and the Baxter--Kelland--Wu (BKW) coupling (Fig.~\ref{fig:BKWcomplete}), we provide a an elementary proof of continuity of the phase transition when $1\leq q\leq 4$ (Theorem~\ref{thm:continuity_asymetric}).
This was first derived in~\cite{DumSidTas17}. That approach goes through establishing a dichotomy;
the side of the dichotomy is decided on by appealing to the parafermionic observable.
We rely on neither of these tools, nor on sharpness of the phase transition or on the known value of the critical edge weight~$p$,
both of which are rigorously established in~\cite{BefDum12}.

Our proof does not rely on invariance under rotations by~$\pi/2$ and hence applies also to the anisotropic case with different edge weights on vertical and horizontal edges.
This statement was originally shown in~\cite{DumLiMan18} using the Yang--Baxter transformation.

For~$q>4$, the transition was proven to be discontinuous in~\cite{DumGagHar16} via the Bethe Ansatz.
An elementary proof has also recently emerged~\cite{RaySpi20}.
When~$q<1$, the random-cluster model does not satisfy the Fortuin--Kasteleyn--Ginibre (FKG) inequality and its phase diagram remains almost entirely open.

\subsubsection{Logarithmic delocalisation}

\label{subsubsec:log_early_intro}

The methods used for proving delocalisation are elementary in spirit and allow
to extend the RSW theory developed in~\cite{GlaMan21} for the uniform random
Lipschitz function to the non-uniform case and to the symmetric ($a=b$)
six-vertex model.  Thus, the renormalisation inequality first developed
for the random-cluster model in~\cite{DumSidTas17} (see also~\cite{DumTas19} which does not rely on self-duality) yields \emph{logarithmic
delocalisation} (Theorems~\ref{thm:soft_Lip_deloc}
and~\ref{thm:hard_six_deloc}).  This means that the variance of the height at a
given face in the loop \O{2} and six-vertex models grows logarithmically in the
distance from this face to the boundary of the domain.  This is in agreement
with the conjecture that the height functions of the loop~\O{2} model
and the six-vertex model converge to the conformally invariant Gaussian free field (GFF) (see for example~\cite{DumKarManOul20}).

\begin{figure}
	\begin{subfigure}{0.50\textwidth}
	\includegraphics[width=\textwidth]{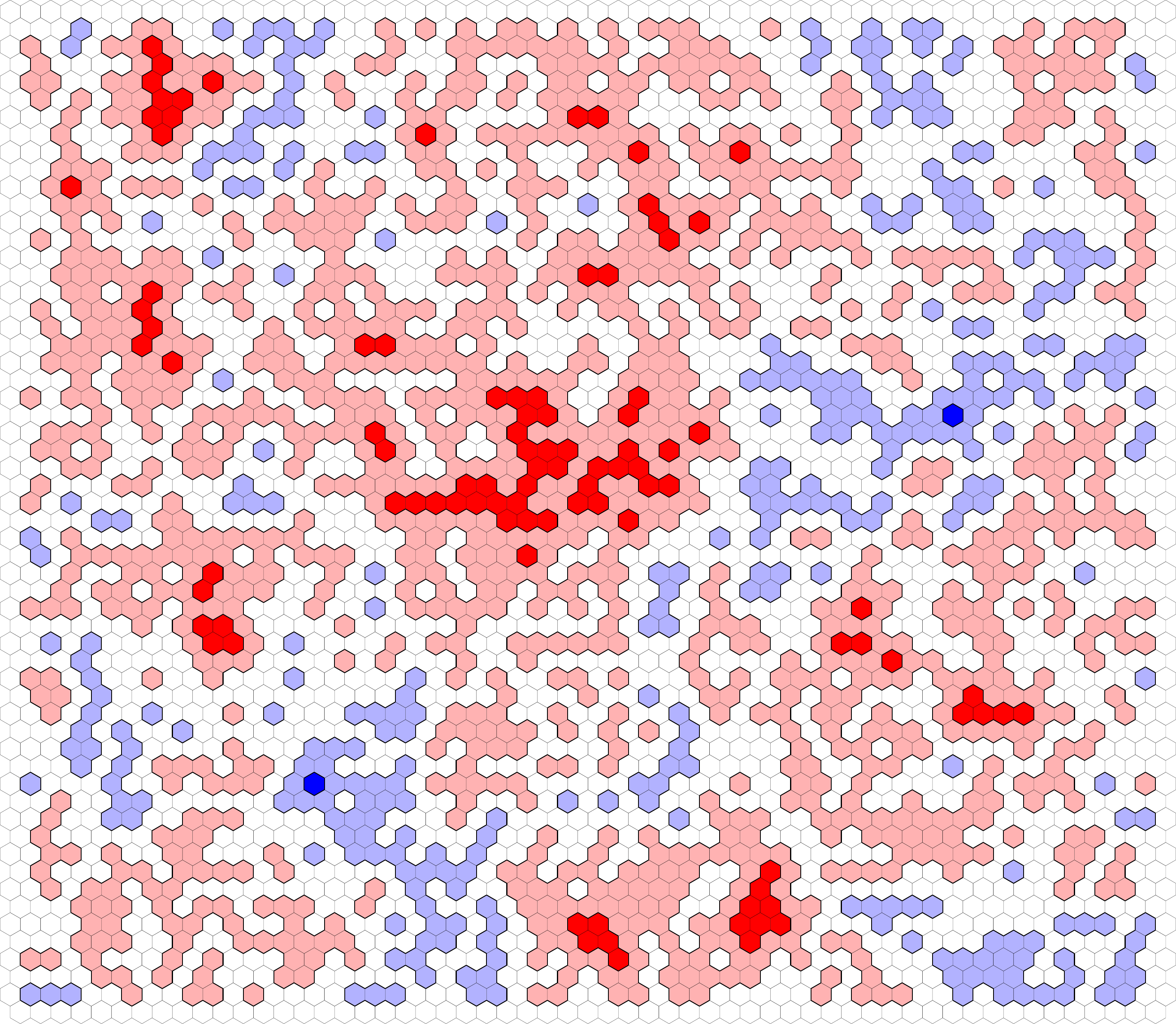}
	\subcaption*{The loop~\O{2} model with $x=1$}
	\end{subfigure}
	\quad
	\begin{subfigure}{0.43\textwidth}
	\includegraphics[width=\textwidth]{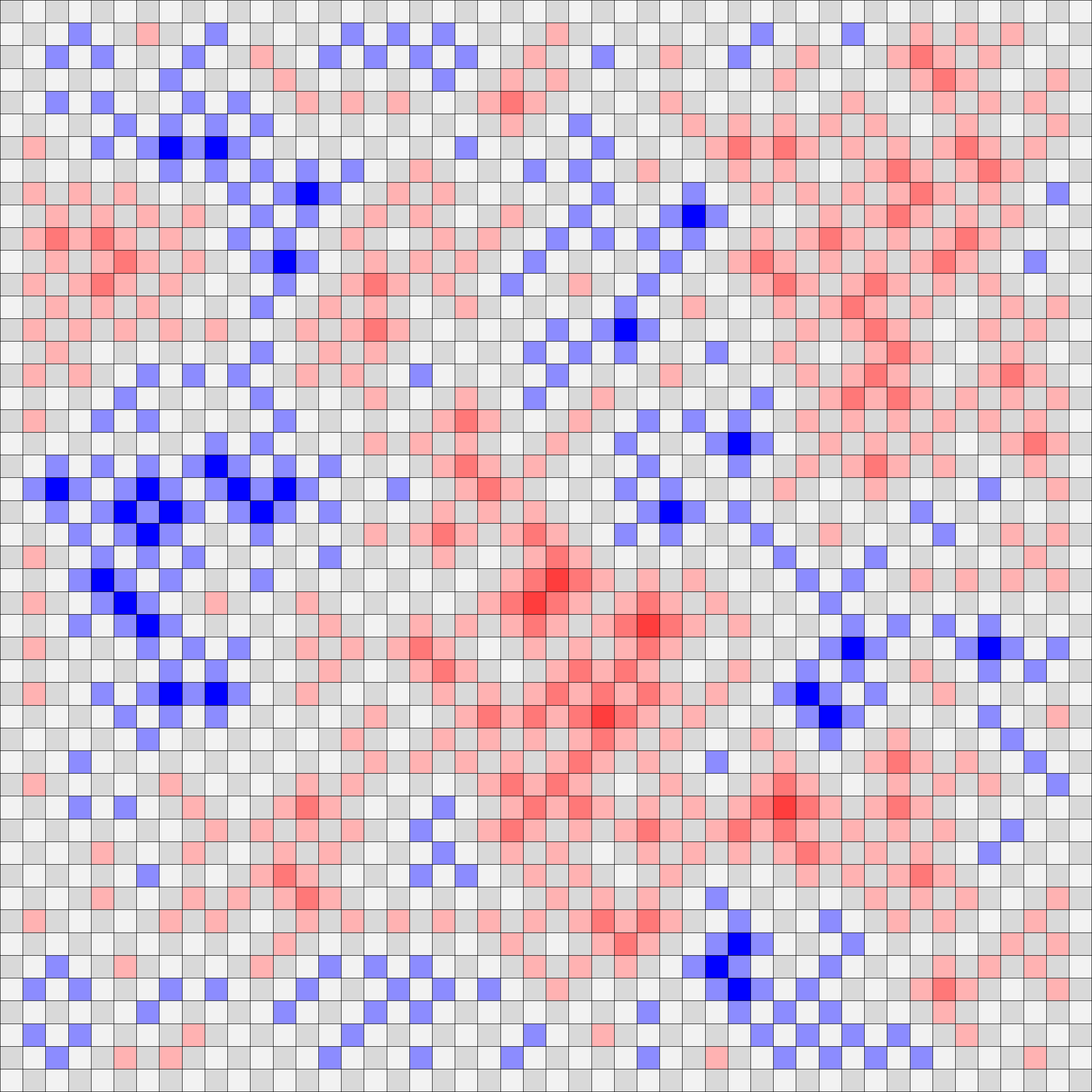}
	\subcaption*{The uniform six-vertex model}
	\end{subfigure}
	\caption{
		Two samples.
	In both cases, the radius of the largest red cluster is
	of the same order of magnitude as the size of the domain.
	}
	\label{fig:largesamples}
\end{figure}

\subsection{Proof strategy}

Our proofs of qualitative delocalisation and of continuity do not rely on
integrable methods nor on Russo--Seymour--Welsh (RSW) theory (the latter only
starts playing a role in the logarithmic quantification of the delocalisation).
This subsection sketches the main ingredients for our proofs.

\subsubsection{A two-spin representation}

Each height function is studied through an appropriate pair of~$\pm1$-valued
black and white spin configurations~$(\sigma^\black,\sigma^\white)$ that describes the heights modulo~$4$ (Figs.~\ref{fig:loop_percolations} and~\ref{fig:sixvertex_percolations}).
In the loop~\O{2} model, each face of the hexagonal lattice is assigned one black
and one white spin.
In the six-vertex model, the faces (\emph{squares}) are
partitioned into black squares and white squares (in a checkerboard
pattern);
black squares contain black spins and white squares contain white spins.
These representations are classical~\cite{Rys63,DomMukNie81} and have recently received significant attention~\cite{GlaMan21,GlaPel19,Lis22,Lis21}.

\subsubsection{The domain Markov property}

Given~$\sigma^\black$, the distribution of~$\sigma^\white$ is that of an Ising model on a modified graph.
Applying an Edwards--Sokal-type expansion, one arrives at the {\em black percolation}~$\xi^\black$ that satisfies the following domain Markov property (Figs.~\ref{fig:loop_percolations} and~\ref{fig:sixvertex_percolations}):
a circuit of~$\xi^\black$ determines the measure in its interior.
This eventually allows one to use the methods developed for the uniform cases in joint works of
the first author with Manolescu~\cite{GlaMan21} and
Peled~\cite[Section~$9$]{GlaPel19}.

We draw attention to an unusual definition of~$\xi^\black$ in the loop~\O{2} model: this is a {\em site percolation} on the triangular lattice formed by vertices of the hexagonal lattice belonging to {\em one chosen partite class} (out of the two,
see Fig.~\ref{fig:loop_percolations}).
The first author learnt about it from Harel and Spinka, and we are unaware of any previous mentioning of this representation in the literature.
In the six-vertex model, $\xi^\black$ is defined in a standard way: this is a {\em bond percolation} on the square lattice formed by black squares.
These edges were implicitly present already in the duality relation for the Ashkin--Teller model~\cite{AshTel43,MitSte71} (see also~\cite{HuaDenJacSal13}) and were described explicitly in~\cite{Lis22,RaySpi22,GlaPel19}.

\subsubsection{Planar duality}

The next step is to define the white percolation~$\xi^\white$ in a similar way by swapping the role of black and white spins:
in the loop~\O{2} model, both~$\xi^\black$ and $\xi^\white$ are site percolations on the same triangular lattice (Fig.~\ref{fig:loop_percolations});
in the six-vertex model, $\xi^\black$ and $\xi^\white$ are bond percolations on the square lattices dual to each other (Fig.~\ref{fig:sixvertex_percolations}).
In the
delocalisation regimes~$\frac12 \leq x^2\leq 1$ and~$a,\,b\leq c\leq a+b$,
one can couple~$\xi^\black$ and~$\xi^\white$ to have the following {\em
super-duality} relation: in the loop~\O{2} model, the {\em site} percolation
dual (that is, the complement) of~$\xi^\black$ is contained
inside~$\xi^\white$; in the six-vertex model, the {\em
bond} percolation dual of~$\xi^\black$ is contained
inside~$\xi^\white$.
For each model, the super-duality turns into exact duality precisely at the critical point; this indicates a deep structural connection between the loop~\O{2} model on the hexagonal lattice and
the six-vertex model on the square lattice.
In the case of the loop~\O{2} model, the coupling is new.
In the case of the six-vertex model, this coupling was introduced by Lis~\cite{Lis22}.

\subsubsection{Joint FKG property for spins and edges}

In the six-vertex model, the edges~$\xi^\black$ satisfy the positive correlation (or Fortuin--Kasteleyn--Ginibre (FKG)) inequality only when~$a+b\leq c$~\cite{Lis22,RaySpi22,GlaPel19}, that is, in the {\em localised} regime and at the transition line.
Our main innovation is the FKG inequality in the {\em delocalised} regime that we discover by representing~$\xi^\black$ as the disjoint union~$\xi^\blackp \cup \xi^\blackm$.
Indeed, by the definition of~$\xi^\black$, the following holds: in the loop~\O{2} model, $\sigma^\black$ is constant at the three faces surrounding a vertex in~$\xi^\black$; in the six-vertex model, $\sigma^\black$ takes the same value at the endpoints of an edge of~$\xi^\black$.
We define~$\xi^\blackp$ and~$\xi^\blackm$ as the subsets of~$\xi^\black$ having black spins plus and minus respectively.
We show that the {\em joint distribution} of the triple~$(\sigma^\black, \xi^\blackp, -\xi^\blackm)$ satisfies the FKG inequality.
Notice that open sites or edges for $\xi^\blackm$ are considered \emph{negative} information
in this setup.
This gives existence of the infinite-volume limit taken under certain maximal boundary conditions for~$\sigma^\black$; moreover, the marginal of this limit on~$(\sigma^\black,\xi^\black)$ is ergodic and extremal.
Using the domain Markov property and some height flipping operation (reminiscent of the cluster swap~\cite{She05}, see also~\cite{CohPel20}), the delocalisation will follow once we show that each of~$\xi^\black$ and~$\xi^\white$ contains infinitely many circuits around the origin.

This joint FKG inequality is inspired by a joint work of the second author with
Ott~\cite[Section~7]{LamOtt21} (see also~\cite[Lemma~3.2]{Lam22}).
The FKG inequalities for the distribution of only the spins were derived in~\cite{GlaMan21} (loop~\O{2} model at~$x=1$) and~\cite{Lis22,RaySpi22,GlaPel19} (in the six-vertex model with~$a,\,b\leq c$).

\subsubsection{The non-coexistence theorem}

In recent years, a significant number of delocalisation
results~\cite{She05,GlaMan21,ChaPelSheTas21,Lam21deloc,LamOtt21} have been derived
from a deep general non-coexistence result in percolation theory.

\begin{theorem*}[Non-coexistence theorem]
	Let $\mu$ denote a translation-invariant (site or bond) percolation measure
	on a planar locally finite doubly periodic graph $\mathbb G$ which satisfies the FKG inequality.
	Then, it is impossible that both a percolation configuration and its dual
	contain a unique infinite cluster $\mu$-almost surely.
\end{theorem*}
A version of this statement under additional symmetry assumptions on $\mu$ is
known as {\em Zhang's argument} (see~\cite[Lemma~11.12]{Gri99a}); it has a
short elementary proof and suffices for our delocalisation result in the
loop~\O{2} model.
The general
non-coexistence theorem was first proved by Sheffield~\cite{She05} (see
also~\cite{DumRaoTas19} for a simpler proof).
We require this general statement to show delocalisation in the six-vertex model.

\subsubsection{$\T$-circuits}

From the non-coexistence theorem, we obtain that, in the infinite-volume limit taken under maximal boundary conditions for~$\sigma^\black$, the configuration~$\sigma^\white$ is disordered.
In the loop~\O{2} model, an additional symmetry between~$\sigma^\black$ and~$\sigma^\white$ readily implies that~$\sigma^\black$ is also disordered, which implies that the height function delocalises.

In the case of the six-vertex model, we need an additional argument to rule out the scenario in which the even heights remain ordered, while the odd heights are disordered.
This is done via so-called~$\T$-circuits and the related coupling of height functions (Fig.~\ref{fig:T-circuits}) introduced
in~\cite{GlaPel19}.
These circuits are constructed on a triangulation obtained
from a square lattice by adding the diagonals parallel to the horizontal axis
(the square lattice itself is rotated by an angle of $\pi/4$ with respect to the standard orientation and therefore such diagonals indeed exist).

\subsubsection{The Baxter--Kelland--Wu coupling}

The seminal BKW coupling between the six-vertex model and the random-cluster model is
real-valued when~$q\geq 4$ and complex-valued when~$q<4$.  Even in the latter case,
the coupling yields nontrivial identities between observables in the two models.
To the best of our
knowledge, such expressions have first appeared in the work of
Dubédat~\cite[p.~398]{Dub11}.
They were used to prove the following results for the six-vertex model:
\begin{itemize}
	\item Delocalisation when~$\sqrt{2+\sqrt{2}}\cdot a=\sqrt{2+\sqrt{2}}\cdot b\leq c\leq a+b$ \cite{Lis21};
	\item Asymptotic rotational invariance when~$\sqrt{a^2+b^2+ab}\leq c\leq a+b$ \cite{DumKarKraManOul20}.
\end{itemize}
It is nontrivial to handle boundary conditions in the BKW coupling.
We follow~\cite{Lis21} in this regard:
we first work on the torus, then send the size of the torus to infinity
in order to obtain full-plane limits.

\subsection{Background}

\subsubsection{The loop~\O{n} model}

Introduced in~1981~\cite{DomMukNie81}, the loop~\O{n} model has
attracted significant attention due to its numerous connections with other models
of statistical mechanics, a rich phase diagram, and conjectured conformally
invariant behaviour.
Recall that the samples consist of a family of non-intersecting loops
on the hexagonal lattice~$\H$, see Fig.~\ref{fig:smallsamples}.
The relative weight of each configuration is calculated by assigning a loop weight $n>0$
to each loop, as well as an edge weight $x>0$ to each edge in the configuration.
Particular cases of the loop~\O{n} model include the ferromagnetic Ising
model ($n=1$, $x\leq 1$), Bernoulli site percolation ($n=x=1$), the dimer model ($n=1$, $x=\infty$),
and the self-avoiding walk ($n=0$, under Dobrushin boundary conditions).
For integer values of~$n$, the model is
heuristically related to the spin~\O{n} model; see~\cite{PelSpi17} for a
survey.

Nienhuis~\cite{Nie82} related the loop~\O{n} model to the six-vertex model on
the Kagomé lattice and conjectured that, for $n\in[0,2]$, the point
\[
	x_c(n):=\tfrac{1}{\sqrt{2+\sqrt{2-n}}}
\]
is critical.
He furthermore conjectured that the model exhibits {\em conformal invariance} for all~$x\geq x_c(n)$.
At~$n=2$, this agrees with the earlier prediction~$x_c(2)=1/\sqrt2$~\cite{DomMukNie81,Fan72} based on a relation with the Ashkin--Teller model.
With the arrival of the \emph{Schramm--Löwner evolution} (SLE)~\cite{Sch00}, this conjecture has taken a precise form:
it is expected that the loops converge to the \emph{conformal loop ensemble} (CLE) of parameter
\[
	\kappa: =
	\begin{cases}
		\tfrac{4\pi}{2\pi -\arccos(-n/2)}\in (\tfrac83, 4], & \text{ if } x=x_c(n),\\
		\tfrac{4\pi}{\arccos(-n/2)}\in [4,8), & \text{ if } x> x_c(n);
	\end{cases}
\]
see~\cite[Section~5.6]{KagNie04}.
This has been proved in the groundbreaking works developing the method of Smirnov's parafermionic observables: for the independent site percolation~\cite{Smi01,CamNew06} (see also~\cite{KhrSmi21})
and for the critical Ising model~($n=1,\, x=1/{\sqrt{3}}$)~\cite{Smi10,CheSmi11,CheDumHonKemSmi14}.
For other parameters, where such a complete description is currently out of reach,
the focus is on studying the coarser properties of the model.
Criticality of the edge weight~$x=x_c(0)$ in the self-avoiding walk ($n=0$)
was rigorously established by Duminil-Copin and Smirnov~\cite{DumSmi12}.

For~$n\in (0,2]$, the loop~\O{n} model is expected to undergo a phase transition
in terms of loop lengths at~$x_c(n)$: for $x< x_c(n)$, the loops lengths are
expected to have \emph{exponential tails}, while, for  $x\geq x_c(n)$, one expects
to see \emph{macroscopic loops}.  The latter means that any annulus of a fixed
aspect ratio is crossed by a loop with a uniformly positive probability.  Note
that this would be an immediate corollary of conformal
invariance.

At $n=2$, we may interpret the loops in the loop~\O{n} model as the level lines
of an integer-valued Lipschitz function on the faces of~$\H$ with the
nearest-neighbour interaction: orienting every loop clockwise or counterclockwise
with probability $\frac12$ makes the loop weight vanish since~$\frac12 \cdot n =
1$ whenever $n=2$.  The macroscopic behaviour is equivalent to logarithmic delocalisation of
this height function: the latter means that the variance at a given face grows
logarithmically in the distance from this face to the boundary of the domain.
The logarithmic delocalisation has been established at the conjectured critical point $x=x_c(2)=1/\sqrt{2}$ in~\cite{DumGlaPel21} (that work establishes macroscopic behaviour at~$x_c(n)$ for all~$n\in [1,2]$, see below) and in the uniform case~$x=1$ in~\cite{GlaMan21}.
We point out that our Theorems~\ref{thm:loop_soft_deloc} and~\ref{thm:soft_Lip_deloc} extend the approach developed in~\cite{GlaMan21} to the non-uniform case.  In
particular, we obtain a new proof for $x=x_c(n)$: unlike in~\cite{DumGlaPel21},
we do not use the parafermionic observable or other integrability tools.  We
point out that the the Lipschitz function at~$x=x_c(n)$ is in fact related to
the uniform graph homomorphism from the \emph{vertices} of~$\H$ to~$\Z$ and that
delocalisation in this case can also be derived from the non-coexistence theorem
via simpler arguments~\cite{ChaPelSheTas21} (see also \cite{Lam21deloc}).
A concurrent work of Karrila proves dichotomy for integer-valued Lipschitz functions on periodic trivalent graphs~\cite{Kar23}.

The only other known regimes for macroscopic behaviour are (see Fig.~\ref{fig:phase-diagram}):
\begin{itemize}
	\item At~$x=x_c(n)$ when~$n\in [1,2]$~\cite{DumGlaPel21},
	\item The supercritical
	Ising model $n=1$, $\tfrac{1}{\sqrt{3}}<x\leq 1$~\cite{Tas16},
	\item An area~$n\in [1, 1+\varepsilon], x\in [1-\varepsilon,
	\tfrac{1}{\sqrt{n}}]$ containing the percolation point~\cite{CraGlaHarPel20}.
\end{itemize}
The loop lengths are known to have exponential
tails in several regimes:
\begin{itemize}
	\item For~$n$ large enough and
	any~$x>0$~\cite{DumPelSamSpi17},
	\item For any~$n>0$ and~$x\leq
	\frac1{\sqrt{2+\sqrt{2}}}$~\cite{Tag18},
	\item For any~$n\geq 1$ and~$x < \frac1{\sqrt{3}} +
	\varepsilon(n)$, where~$\varepsilon(n)$ is some strictly increasing function with $\epsilon(1)=0$~\cite{GlaMan21b}.
\end{itemize}
At~$n=2$, exponential decay remains
open for all~$x\in [1/\sqrt{3}+\varepsilon(2),\frac1{\sqrt{2}})$.

\begin{figure}
	\begin{center}
		\begin{subfigure}{0.6\textwidth}
		\hspace{-2em}
		\includegraphics[scale=1]{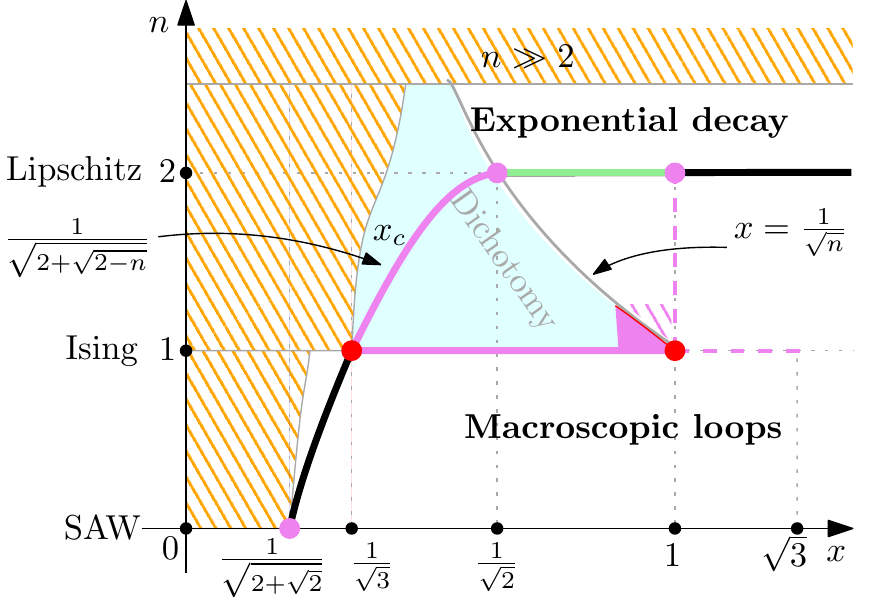}
		\subcaption*{\textbf{The loop~\O{n} model}}
		\end{subfigure}\hfill
		\begin{subfigure}{0.35\textwidth}
		\includegraphics[scale=1]{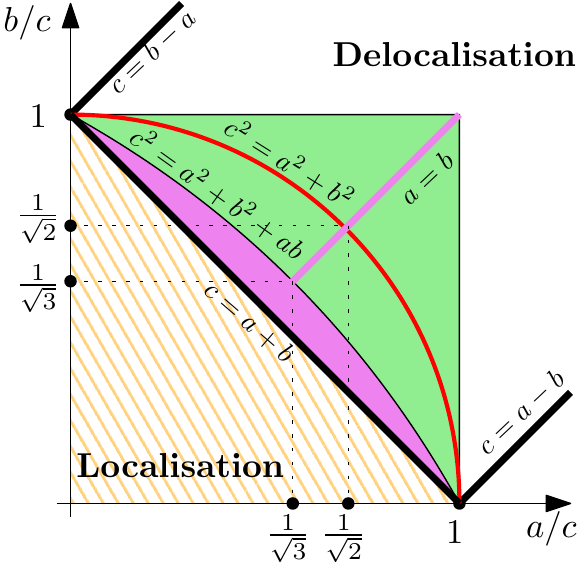}
		\subcaption*{\textbf{The six-vertex model}}
		\end{subfigure}
	\end{center}
	\caption{\textsc{Left}: Phase diagram of the loop~\O{n} model: conjectured transition curve (bold black), known conformal invariance (red), macroscopic behaviour (violet), exponential decay (orange), and dichotomy (blue); new macroscopic behaviour (green).
	 \textsc{Right}: Phase diagram of the six-vertex model: conjectured transition curve (bold black); known conformal invariance (red),  delocalisation (violet) and localisation (orange); new delocalisation (green).}
	\label{fig:phase-diagram}
\end{figure}

\subsubsection{The six-vertex model}

The model has three
parameters~$a,\,b,\,c>0$: each vertex receives the weight~$a$ or~$b$ if the heights
disagree along one of the
two pairs of diagonally adjacent faces, and weight~$c$ if the heights agree
along both diagonals (Fig.~\ref{fig:six_vertex}).  The model was originally
introduced by Pauling~\cite{Pau35} in three dimension as a simplified
representation of ice crystals; two-dimensional versions then appeared in the
works of Slater~\cite{Sla41} and Rys~\cite{Rys63}.  Based on the work of
Yang--Yang~\cite{YanYan66}, Lieb~\cite{Lie67KDP,Lie67b,Lie67c} and
Sutherland~\cite{Sut67} computed the free energy of the model via the
celebrated Bethe Ansatz~\cite{Bet31}.  The latter gives an explicit guess for
the eigenvectors and eigenvalues of a (very large) transfer matrix,
which in certain cases can be verified to be correct; see~\cite{DumGagHar16b}
for a review.
The free energy computations point to the following phase transition under flat boundary conditions (Fig.~\ref{fig:phase-diagram}):
\begin{itemize}
	\item Delocalisation when~$|a-b| \leq c\leq a+b$: the variance of the height function at a given face grows logarithmically in the distance of this face to the boundary; the height function is expected to converge to the GFF,
	\item Localisation when~$a+b < c$: the variance of the heights is uniformly bounded.
\end{itemize}
We also mention that in the uniform case~$a=b=c$, delocalisation has been established also under boundary conditions with a slope~\cite{She05}.

Convergence to the GFF was established at the {\em free fermion line}~$a^2+b^2=c^2$ (corresponding to the dimer model) in the celebrated works of Kenyon~\cite{Ken00,Ken01} and in a neighbourhood of this line by Giuliani, Mastropietro, and Toninelli~\cite{GiuMasTon17}.

The recent years have seen significant progress in the rigorous mathematical
analysis of the phase diagram of the model via a combination of probabilistic and integrable
arguments for~$a,\,b\leq c$.  Peled and the first author derived delocalisation at $a+b=c$ and
localisation at $a+b<c$ using the BKW coupling~\cite{GlaPel19} and known
continuity~\cite{DumSidTas17} and discontinuity~\cite{DumGagHar16} results for
the random-cluster model (obtained via the parafermionic observable and
the Bethe Ansatz respectively).  Ray and Spinka later found a simplified argument
for the localisation not relying on integrability~\cite{RaySpi20}.
Lis~\cite{Lis21} used the BKW coupling and the continuity of the phase
transition to derive the delocalisation of six-vertex model in the
range~$\sqrt{2+\sqrt{2}}\cdot a = \sqrt{2+\sqrt{2}}\cdot b\leq c\leq a+b$.
 Delocalisation for the uniform case
$a=b=c$ was derived from the non-coexistence theorem in~\cite{ChaPelSheTas21}
via a technique called \emph{cluster swapping} introduced by Sheffield~\cite{She05}
(see also~\cite{CohPel20,Lam19});
a shorter and more direct argument via~$\T$-circuits appeared later
in~\cite[Section~$9$]{GlaPel19}.
In a recent
work, Duminil-Copin, Karrila, Manolescu, and Oulamara
proved delocalisation when $a=b\leq c \leq a+b$ by combining information
coming from the Bethe Ansatz with probabilistic arguments~\cite{DumKarManOul20}.

Outside of the regime~$a,\,b\leq c$, the FKG inequality breaks down and most of the known probabilistic tools do not apply.
However, the tools of representation theory do apply when~$c \leq |a-b|$ which is called the stochastic six-vertex model; see~\cite{BorCorGor16} and references therein.
Similar tools also apply to the regime~$c > |a-b|$ and give a very precise information about the distribution of the height function close to the boundary~\cite{GorPan15,AggGor22,GorLie23}.

\subsubsection{The random-cluster model}

Fortuin and Kasteleyn~\cite{ForKas72} introduced the random-cluster model in 1972 as a graphical representation of the \emph{Ising} ($q=2$) and the \emph{Potts} ($q\in\mathbb Z_{\geq  2}$) models.
Configurations are spanning subgraphs of a given finite graph and the probability depends on the number of edges (via the edge weight~$p\in [0,1]$) and connected components (via the cluster weight~$q>0$).
The case~$q=1$ corresponds to the standard Bernoulli bond percolation model.

The understanding of the random-cluster model in two dimensions with~$q\geq 1$ has significantly developed over the past two decades.
The observation at the core of the probabilistic analysis of the model
is the FKG inequality~\cite{ForKasGin71}, which asserts that increasing events are positively correlated as soon as $q\geq 1$; see~\cite{Gri06} for an overview of the classical results.
In~\cite{BefDum12}, Beffara and Duminil-Copin proved that the self-dual point~$\pselfdual = \sqrt{q}/(1+\sqrt{q})$ is critical,
and that the phase transition is sharp: when~$p<\pselfdual$, the probability that two vertices are in the same connected component decays exponentially in the distance between them; when~$p> \pselfdual$, each vertex belongs to the unique infinite connected component with a positive probability.
The latter probability is called the {\em density} of the infinite cluster.

The type of the phase transition can be classified by looking at continuity properties of this density function: the transition is called {\em continuous} if the density is continuous in~$p$, and it is called {\em discontinuous} otherwise.
The free energy computations for the six-vertex models (alluded to above), together with the BKW coupling with the random-cluster model, allowed Baxter~\cite{Bax73} to non-rigorously derive that the transition is continuous when~$q \leq 4$ and discontinuous when~$q>4$.
For~$q\geq 1$, this was established rigorously:
Duminil-Copin, Sidoravicius, and Tassion~\cite{DumSidTas16} proved continuity for~$1\leq q\leq 4$ using the parafermionic observable;
Duminil-Copin, Gagnebin, Harel, Manolescu, and Tassion proved discontinuity for $q>4$~\cite{DumGagHar16} using the Bethe Ansatz.
For~$q>4$, Ray and Spinka~\cite{RaySpi20} later found an elementary argument for discontinuity
that does not rely on integrability.
The above results have been extended to the anisotropic case with different edge weights on vertical and horizontal edges
using the Yang--Baxter transformation~\cite{DumLiMan18}.

The spin correlations in the Ising and the Potts models can be expressed via connection probabilities in the random-cluster model.
Thus, continuity of the phase transition in the Ising and the Potts models
is essentially equivalent to continuity of the phase transition in the corresponding
random-cluster model.
We also point out the existence of an elementary proof for continuity for the planar Ising model due to Werner~\cite{Wer09a}.

Finally, we mention a beautiful work of K\"ohler-Schindler and Tassion that establishes the Russo--Seymour--Welsh estimates
in a very general setting relying only on symmetries and the FKG inequality~\cite{KohTas20}.

\subsubsection{The two-spin representation and percolations}

The two-spin representation of the loop~\O{2} model and the six-vertex model
is reminiscent of the Ashkin--Teller model that describes a pair of interacting Ising models~\cite{AshTel43}.
Such representations were introduced in the physics literature a long time ago~\cite{Rys63,DomMukNie81} and recently appeared in the analysis of the phase
diagram of several models~\cite{GlaPel19,GlaMan21,Lis21,AouDobGla23}.
Similarly to the Ising model, one can perform an Edwards--Sokal-type expansion and obtain a pair of coupled percolation configurations~$\xi^\black$ and~$\xi^\white$ as follows.
\begin{itemize}
	\item In the loop~\O{2} model, $\xi^\black$ and~$\xi^\white$ are site percolations on a triangular lattice.
	The definition of~$\xi^\black$ was communicated to the first author by Harel and Spinka during their joint stay at the Tel Aviv University in~2019.
	As far as we know, the coupling of the two percolations~$\xi^\black$ and~$\xi^\white$ is new.
	\item In the six-vertex model, $\xi^\black$ and~$\xi^\white$ are bond percolation on a square lattice.
	The coupling of $\xi^\black$ and~$\xi^\white$ was introduced by Lis who also showed that non-percolation in~$\xi^\black$ and~$\xi^\white$ implies delocalisation~\cite[Theorem~6.4]{Lis22}.
	Each of~$\xi^\black$ and~$\xi^\white$ separately implicitly appears already in the duality mappings for the Ashkin--Teller models~\cite{AshTel43,MitSte71}  (see also~\cite{HuaDenJacSal13} for a review) and are used in~\cite{GlaPel19,RaySpi22}.
\end{itemize}
The similarity in proofs for the loop~\O{2} and the six-vertex models in this
article suggests that the two models are linked (at least in spirit) and that other qualitative ideas
may be transported from one model to the other.
A prime candidate would be the proof of localisation in the
six-vertex model due to Ray and Spinka~\cite{RaySpi20}:
at the present, the problem of deriving
exponential decay of the loop length for $n=2$ and for $x$ just below $x_c(2)=1/\sqrt2$ remains open.
On the
other hand, our unifying approach to delocalisation may apply to other
integer-valued height functions, such as Lipschitz functions on the square lattice
and graph homomorphisms on the hexagonal lattice with a nontrivial (20-vertex-type) interaction.

\subsection{Organisation of the paper}
Section~\ref{sec:results} defines the loop~\O{2} model, the six-vertex
model, and the random-cluster model and states our main results for each model.
Sections~\ref{section:loop_on} and~\ref{sec:six-vertex} prove our qualitative
delocalisation results for the loop~\O{2} and the six-vertex models.
Most steps of the two proofs are essentially identical, so we structured the sections similarly.
Section~\ref{sec:continuity} derives the continuity of the phase transition in the random-cluster model
from the delocalisation result for the six-vertex model.
Appendix~\ref{app:dicho} describes how to apply the strategy of~\cite{GlaMan21}
to quantify the delocalisation result.

\subsection{Acknowledgements}
AG is grateful to Matan Harel and Yinon Spinka for sharing in~2019 their ideas
regarding vertex percolations in the loop~\O{2} model and to Marcin Lis for
discussions regarding connections between the six-vertex and the random-cluster
models.
PL thanks François
Jacopin, Alex Karrila, and Mendes Oulamara for a collaboration in
which he learned about the BKW coupling,
and Hugo Duminil-Copin for the same collaboration as well as for fruitful discussions regarding the current article.

The authors want to
thank Moritz Dober, Laurin Köhler-Schindler, Ioan Manolescu, and Maran
Mohanarangan for fruitful discussions, especially regarding the non-symmetric
six-vertex model,
and Romain Panis for commenting on earlier versions of the manuscript.
This project started at the University of Vienna and we want
to thank it for the hospitality.

AG is supported by the Austrian Science Fund grant P3471.
This project has received funding from the European Research
Council (ERC) under the European Union's Horizon 2020 research and innovation programme (grant
agreement No.~757296).

\section{Definitions and formal statements}
\label{sec:results}

\subsection{The loop \O{2} model}
\label{sec:results-lip}

Let $\hexlattice=(\vertices{\hexlattice},\edges{\hexlattice})$ denote the
hexagonal lattice whose faces $\faces{\hexlattice}$ are centred at $\{k+\ell
e^{i\pi/3}:k,\ell\in\Z\}\subset\mathbb C$.  A \emph{domain} is a finite subgraph
$\domain=(\vertices{\domain},\edges{\domain})\subset\hexlattice$ consisting precisely of the
sets of vertices and edges which are on or contained inside a cycle on
$\hexlattice$.  For a given domain $\domain$, this cycle
is denoted by $\partial\domain\subset\domain$.  A \emph{loop configuration} on~$\domain$ is a
spanning subgraph of~$\domain \setminus \partial \domain$ in which every vertex
has degree~$0$ or~$2$.  The term comes from the observation that each nontrivial connected
component of $\omega$ is a cycle that we call a \emph{loop}.  Denote the
set of all loop configurations on~$\domain$ by~$\spaceLoop{\domain}$.

\begin{definition}[The loop~\O{n} model]
	Let $n,x > 0$.  The \emph{loop~\O{n} model} on~$\domain$ with \emph{edge weight}~$x$ (and \emph{loop
	weight} $n$) is the probability measure
	$\muLoop_{\domain,n,x}$ on $\spaceLoop{\domain}$ defined by
	\begin{equation}
		\label{eq:loop_on_weight}
		\muLoop_{\domain,n,x}(\omega)
		:= \tfrac{1}{Z_{\domain,n,x}} \cdot  n^{\ell(\omega)} \cdot x^{|\omega|},
	\end{equation}
	where~$\ell(\omega)$ and~$|\omega|$ denote the numbers of loops and edges
	in~$\omega$ respectively, and~$Z_{\domain,n,x}$ is the normalising
	constant (called the {\em partition function}) that
	renders~$\muLoop_{\domain,n,x}$ a probability measure.
\end{definition}

Our results address the $\Omega\nearrow\H$ limit of this family of measures,
and for this reason we now introduce full-plane loop measures and an appropriate topology.
	A \emph{loop configuration} on~$\H$ is a
spanning subgraph of~$\H$ in which every vertex
has degree~$0$ or~$2$.
Denote the
set of all loop configurations on~$\H$ by~$\spaceLoop{\H}$.
Observe that the nontrivial connected components of a loop configuration
on $\H$ are loops or bi-infinite paths.
Write $\calP$ for the family of all probability measures on the sample space $\spaceLoop{\H}$ endowed with the $\sigma$-algebra generated by \emph{cylinder events}
(events depending on the state of finitely many edges).
We view each configuration $\omega\in\spaceLoop{\domain}$
	as a configuration in $\spaceLoop{\H}$ by identifying
	$\omega=(\vertices{\domain},\edges{\omega})$
	with $(\vertices{\H},\edges{\omega})$.
	This also allows us to view each measure $\muLoop_{\domain,n,x}$
	as a measure in $\calP$.

\begin{definition}[Full-plane Gibbs measures of the loop~\O{n} model]
	For any domain $\domain$ and~${\omega'}\in \spaceLoop{\H}$, we define
\[
\spaceLoop{\H;\domain;{\omega'}}
:=
\{
	\omega\in\spaceLoop{\H}
	:
	\edges{\omega}\setminus\edges{\domain}
	=
	\edges{{\omega'}}\setminus\edges{\domain}
\}.
\]
We define~$\muLoop_{\domain,n,x}^{\omega'}\in\calP$ as the following probability measure supported on $\spaceLoop{\H;\domain;{\omega'}}$:
\[
	\muLoop_{\domain,n,x}^{\omega'}(\omega) = \tfrac{1}{Z_{\domain,n,x}^{\omega'}} \cdot  n^{\ell(\omega;\domain)}\cdot x^{|\omega\cap\domain|},
\]
where~$\ell(\omega;\domain)$ is the number loops and bi-infinite paths in $\omega$ intersecting $\vertices{\domain}$,
$|\omega\cap\Omega|$ is the number of edges in $\omega\cap\Omega$, and~$Z_{\domain,n,x}^{\omega'}$ is the partition function.
A measure $\mu\in\calP$ is called a \emph{Gibbs measure}
	if, for any domain $\domain$
	and for $\mu$-almost every ${\omega'}$,
	the measure $\mu$ conditional on $\{\omega\in\spaceLoop{\H;\domain;{\omega'}}\}$
	equals $\muLoop_{\domain,n,x}^{\omega'}$.
	(This conditional measure is uniquely defined as a regular conditional probability distribution
	(r.c.p.d.)
	up to $\mu$-almost nowhere modifications.)
This definition of a Gibbs measure is equivalent to asking that for any domain $\domain$ and for any bounded measurable function $f:\spaceLoop{\H}\to\R$,
we have
\[
	\mu(f)=	\int\muLoop_{\domain,n,x}^{\omega'}(f) d\mu({\omega'})
	.
\]
Write $\mathcal G_{2,x}\subset\calP$ for the set of Gibbs measures.
The family $(\muLoop_{\domain,n,x}^{\omega'})_{\omega'}$ is a \emph{probability kernel},
and the family $(\muLoop_{\Omega,n,x}^{\omega'})_{\Omega,{\omega'}}$ is a \emph{specification}.
\end{definition}

Although not immediately apparent, the definition of the specification required us to make some arbitrary choices limiting in some
sense the universality of our main result;
more details may be found in Remark~\ref{rem:limitations} below.

\begin{definition}[The weak topology]
	The \emph{weak topology} on $\calP$ is the coarsest
	topology such that the map $\mu\mapsto \mu(A\subset \edges{\omega})$
	is continuous for any finite $A\subset\edges{\hexlattice}$.
\end{definition}

\begin{theorem}
	\label{thm:loop_soft_deloc}
	Let $n=2$ and
	$x\in[1/\sqrt2,1]$.
	Then, $\mathcal G_{2,x}$ consists of a unique measure, which we denote by $\muLoop_{2,x}$.
	This measure
	is extremal, shift-invariant, and ergodic with respect to the symmetries of~$\hexlattice$,
	exhibits no bi-infinite paths almost surely, and satisfies
	\begin{equation}\label{eq:inf-many-loops-lip}
		\muLoop_{2,x}(\text{each face in $\faces{\hexlattice}$ is surrounded by infinitely many loops}) = 1.
	\end{equation}
	Moreover, $\muLoop_{\domain_k,2,x}$ converges to~$\muLoop_{2,x}$ in the weak topology
	for any increasing sequence of domains $(\domain_k)_k\nearrow\hexlattice$.
\end{theorem}

\begin{remark}
	\label{rem:limitations}
	The probability kernel $(\muLoop_{\domain,n,x}^{\omega'})_{\omega'}:\spaceLoop{\hexlattice}\to\calP$
	is not continuous: the points of discontinuity are precisely the configurations
	which have at least two bi-infinite paths intersecting the domain.
	For this reason, Theorem~\ref{thm:loop_soft_deloc} does not classify
	all \emph{thermodynamical limits}, that is, weak limits of $\muLoop_{\domain_k,n,x}^{\omega'}$
	as $k\to\infty$ for some increasing sequence of domains $(\domain_k)_k\nearrow\hexlattice$.
	Such rogue limits cannot be invariant under lattice translations,
	since in that case the classical and robust Burton--Keane argument~\cite{BurKea89} would rule out the appearance
	of more than a single infinite interface~\cite{LamTas20}.

	The existence of Gibbs measures which are not translation-invariant has been ruled out for several
	statistical mechanics models in
	two dimensions: the Ising model~\cite{Aiz80,Hig81,georgii2000percolation}, the Potts and
	FK-percolation models, and the loop~\O{n} model for~$n\geq 1,\, x\leq
	1/\sqrt{n}$~\cite{CoqDumIofVel14,GlaMan21c},
	and finally the dimer model~\cite{Agg19}.
	The problem remains intricately open for the XY model,
	in which case only the uniqueness of the shift-invariance Gibbs measure has been established~\cite{Cha98}.
\end{remark}

The loops of the loop~\O{2} model appear naturally as the level lines of an integer-valued Lipschitz function on $\faces{\H}$.
Let $\face(\domain)\subset\faces{\hexlattice}$ denote the set of faces enclosed
by $\partial\domain$,
and let $\partial_\face\domain\subset\faces{\domain}$ denote the set of faces adjacent to $\partial\domain$.

\begin{definition}[Random Lipschitz function]
	A function~$h\colon \face(\domain) \to \Z$ is called a \emph{Lipschitz function} on~$\domain$ if,
	for any adjacent faces $u,v\in \face(\domain)$,
	\[
		|h(u) - h(v)| \in \{0,1\}.
	\]
	Let~$\spaceLip^0(\domain)$ denote the set of Lipschitz functions that satisfy~$\left.h\right|_{\partial_\face\domain} \equiv 0$.
	The \emph{domain wall} of $h$ is the spanning subgraph~$\omega[h]$ of~$\domain$ given by the set of edges in $\edges{\domain}\setminus\edges{\partial\domain}$ separating faces where~$h$ takes different values ({\em disagreement edges}).
	The \emph{random Lipschitz function} on domain~$\domain$ with parameter~$x>0$ under zero boundary conditions
	is the probability measure on $\spaceLip^0(\domain)$ defined by
	\begin{equation}
		\label{eq:homomorph_weight}
		\muLip^0_{\domain,x}(h) := \tfrac{1}{Z_{\domain,x}^0}\cdot x^{|\omega[h]|},
	\end{equation}
	where~$Z_{\domain,x}^0$ is the partition function.
\end{definition}

It is easy to see that the map
\[
	\spaceLip^0(\domain)\to \spaceLoop{\domain}
	,\,
	h\mapsto \omega[h]
\]
is well-defined and that the pushforward of $\muLip^0_{\domain,x}$
along this map is $\muLoop_{\domain,2,x}$.
Indeed, the gradient of a Lipschitz function consists of oriented loops, and so each loop configuration~$\omega$ has exactly~$2^{\ell(\omega)}$ Lipschitz functions corresponding to it;
see Proposition~\ref{prop:lip-spin-loop} below.
Theorem~\ref{thm:loop_soft_deloc} implies that the Lipschitz function is
delocalised for~$x\in[1/\sqrt2,1]$.
The following result also quantifies the delocalisation.

Let $\operatorname{dist}:\faces{\hexlattice}\times\faces{\hexlattice}\to\R$
denote the metric induced by the Euclidean distance between the centres of two faces.

\begin{theorem}[Logarithmic delocalisation of the Lipschitz function]
	\label{thm:soft_Lip_deloc}
	Let $x\in[1/\sqrt2,1]$.
	Then, the random Lipschitz function delocalises \emph{logarithmically}:
	there exist universal constants~$c,C>0$ (not depending on $x$) such that, for any domain~$\domain$ and any face~$u\in \faces\domain$,
	\[
		c\cdot\log\operatorname{dist}(u,\partial_\face\domain)
		\leq
		\Var_{\muLip^0_{\domain,x}}[h(u)]
		\leq
		C\cdot\log\operatorname{dist}(u,\partial_\face\domain).
	\]
	The same estimates hold also for the expected number of loops
	surrounding~$u$ in the corresponding loop~\O{2} model since this expectation
	equals the variance in the above display.
\end{theorem}

\subsection{The six-vertex model}
\label{subsec:intro:6v}

Let $\squarelattice=(\vertices{\squarelattice},\edges{\squarelattice})$ denote the
square lattice whose faces $\faces{\squarelattice}$ are centred at $\Z^2$.
Faces in this context are also called \emph{squares}.
A square is called \emph{even} if the coordinate sum of its centre is even, and it is called \emph{odd} otherwise.
Even and odd squares are also thought of as being \emph{black} and \emph{white}
respectively, so that $\squarelattice$ resembles an infinite chessboard.
A \emph{domain} is a finite subgraph
$\domain=(\vertices{\domain},\edges{\domain})\subset\squarelattice$ consisting precisely of the
sets of vertices and edges which are on or contained inside a cycle on
$\squarelattice$. For a given domain $\domain$, this cycle
is denoted by $\partial\domain\subset\domain$.
Let $\face(\domain)\subset\faces{\squarelattice}$ denote the set of squares enclosed
by $\partial\domain$,
and let $\partial_\face\domain\subset\faces{\domain}$ denote the set of squares which share
a \emph{vertex} with $\partial\domain$.
Two squares are \emph{adjacent} if they share an edge.

A \emph{graph homomorphism} is a parity-preserving function $h:\faces{\domain}\to\Z$ such that
\[
	|h(v)-h(u)|=1
\]
for any two adjacent squares $u,v\in\faces{\domain}$.
Define~$\spaceHom^{0,1}(\domain)$ as the set of graph homomorphisms on $\domain$ with~$0,1$ boundary conditions: they satisfy
\[
	h(u)\in\{0,1\} \qquad\forall u\in \partial_\face\domain.
\]
Remark that the parity constraint forces the height function to take a value zero
at even boundary squares and a value one at odd boundary squares.

Define the set of \emph{interior vertices} by $\intvert(\domain):=V(\domain)\setminus V(\partial\domain)$.
Consider a vertex $v\in \intvert(\domain)$.
The gradient of $h$ between the four squares incident to $v$
can take one of six possible values.
For~$i=1,\dots, 6$, define~$n_i(h)$ as the number of vertices in $\intvert(\domain)$ that have type~$i$ for~$h$; see Fig.~\ref{fig:six_vertex}.

\begin{figure}
	\centering
	\includegraphics[scale=1.5]{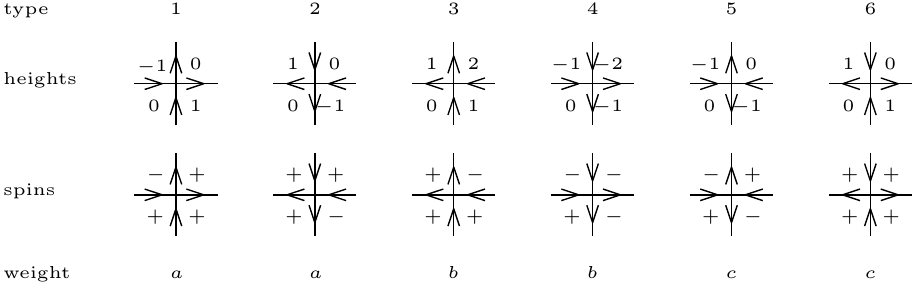}
	\caption{The six types of arrow configurations and their weights. The edge orientations determine the heights up to an additive constant and the spins up to the global spin flip.}
	\label{fig:six_vertex}
\end{figure}

\begin{definition}[Random graph homomorphism]
Given some parameters $a,b,c>0$, the probability measure $\muHom_{\domain,a,b,c}^{0,1}$ on
$\spaceHom^{0,1}(\domain)$
is defined by
\begin{equation}
	\label{eq:6v_weight_def}
	\muHom_{\domain,a,b,c}^{0,1}(h):=\tfrac1{Z_{\domain,a,b,c}^{0,1}}\cdot
	a^{n_1(h)+n_2(h)}b^{n_3(h)+n_4(h)}c^{n_5(h)+n_6(h)},
\end{equation}
where $Z_{\domain,a,b,c}^{0,1}$ is the partition function.
\end{definition}

One way to view the random gradient of~$h$ is by orienting each edge of~$\squarelattice$ in such a way that the larger height is on its right.
The pushforward of~$\muHom_{\domain,a,b,c}^{0,1}$ along this bijective map is the \emph{six-vertex model}.
We now state our main result for the six-vertex model.

Let $\operatorname{dist}:\faces{\squarelattice}\times\faces{\squarelattice}\to\R$
denote the metric induced by the Euclidean distance between the centres of two squares.

\begin{theorem}[Delocalisation in the six-vertex model]
	\label{thm:hard_six_deloc}
	Consider $a,b,c>0$ such that~$a,b\leq c\leq a+b$.
	Then, the graph homomorphism \emph{delocalises}: for any sequence of domains $(\domain_k)_k\nearrow\Z^2$ containing
	some fixed square $u\in\faces{\squarelattice}$, we have
	\begin{equation}
		\label{thm:hard_six_deloc:qual}
		\lim_{k\to\infty}\Var_{\muHom^{0,1}_{\domain_k,a,b,c}}[h(u)]=\infty
	\end{equation}
	If furthermore $a=b$, then the variance grows {\em logarithmically}:
	there exist universal constants~$r,R>0$ (not depending on~$a,b,c$) such that, for any finite domain~$\domain$ and any square~$u\in \faces{\domain}$,
	\begin{equation}
		\label{thm:hard_six_deloc:quant}
		r\cdot \log\operatorname{dist}(u,\partial_F\domain) \leq \Var_{\muHom^{0,1}_{\domain,a,b,c}}[h(u)] \leq R\cdot \log\operatorname{dist}(u,\partial_F\domain).
	\end{equation}
\end{theorem}

\subsection{The random-cluster model}
\label{subsec:introRCM}

Work on the square lattice graph $\squarelattice$ described above.
Let $\domain$ be a domain,
and recall that $\intvert(\domain)$ denotes the set of
interior vertices of $\domain$.
The squares $F(\squarelattice)$ decompose as a disjoint union of black squares $F^\black(\squarelattice)$
and white squares $F^\white(\squarelattice)$.
For any domain $\domain$, we write $F^\black(\domain)$ and $F^\white(\domain)$
for the intersections of these two sets with $F(\domain)$ respectively.
Also write $\partial_{F^\black}\domain$ and $\partial_{F^\white}\domain$
for the intersections with $\partial_F\domain$.
Write $\domain^\black=(V(\domain^\black),E(\domain^\black))$ for the graph whose vertex set
is $F^\black(\domain)$ and such that two squares are neighbours if and only if they both contain the same vertex
in $\intvert(\domain)$.
The graph $\domain^\white=(V(\domain^\white),E(\domain^\white))$ is defined similarly,
and we shall also allow $\squarelattice=\domain$ in these definitions.

We view every~$\eta\in \{0,1\}^{E(\domain^\black)}$ as a percolation
configuration by stating that an edge~$e\in E(\domain^\black)$ is {\em open}
if~$\eta_e = 1$ and {\em closed} otherwise.  We identify~$\eta$ with the set of open edges and with the
spanning subgraph of~$\domain^\black$ given by it.
The dual of $\eta$, written $\eta^*$, is a spanning subgraph of $\Omega^\white$ and defined such that, for every edge $e^*\in E(\Omega^\white)$,
\[
	\text{$e^*$ is $\eta^*$-open
if and only if $e$ is $\eta$-closed.}
\]

The non-quantitative part of our delocalisation arguments does not rely on the rotation by~$\pi/2$ and hence applies readily to the {\em asymmetric} random-cluster model.
To define this model, we need some additional notation.
Let~$E_a^\black$ (resp.~$E_b^\black$) be the set of all edges in~$E(\squarelattice^\black)$ parallel to~$e^{\pi i/4}$ (resp.~$e^{3\pi i/4}$).
Note that~$E_a^\black$ and~$E_b^\black$ are disjoint and~$E_a^\black\cup E_b^\black = E(\squarelattice^\black)$.
(The notation is chosen to fit the weights~$a$ and~$b$ in the six-vertex model; specifically in Section~\ref{sec:six-vertex}.)

\begin{definition}[Random-cluster model]
	\label{def:RCM}
	Given~$p_a,p_b\in [0,1]$ and $q>0$, the random-cluster model on~$\domain^\black$
	is supported on $\{0,1\}^{E(\domain^\black)}$ and, in the case of {\em free}
	boundary conditions, is defined by
	\[
		\phifree_{\domain^\black,p_a,p_b,q}(\eta):= \tfrac{1}{\Zfree_{\domain^\black,p_a,p_b,q}}
		p_a^{|\eta\cap E_a^\black|}
		p_b^{|\eta\cap E_b^\black|}
		(1-p_a)^{|E(\domain^\black)\cap E_a^\black\setminus \eta|}
		(1-p_b)^{|E(\domain^\black)\cap E_b^\black\setminus \eta|}
		q^{k(\eta)},
	\]
	where~$|\cdot|$ denotes the cardinality of the set and~$k(\eta)$ is the number
	of connected components in~$\eta$ (including isolated vertices). Below, we refer to the connected components as {\em clusters}.

	Under the wired boundary conditions, the model is defined by
	\[
		\phiwired_{\domain^\black,p,q}(\eta):= \tfrac{1}{\Zwired_{\domain^\black,p_a,p_b,q}}
		p_a^{|\eta\cap E_a^\black|}
		p_b^{|\eta\cap E_b^\black|}
		(1-p_a)^{|E(\domain^\black)\cap E_a^\black\setminus \eta|}
		(1-p_b)^{|E(\domain^\black)\cap E_b^\black\setminus \eta|}
		q^{k(\etawired)},
	\]
	where~$\etawired$ is the graph obtained from~$\eta$ by identifying all vertices of~$\partial_{F^\black}\domain$.
\end{definition}

We restrict to the case~$q\geq 1$; the FKG inequality fails for $q<1$ which
renders essentially all known probabilistic methods useless.
The FKG inequality implies that all measures are stochastically increasing in $p_a$ and~$p_b$,
and that:
\begin{enumerate}
    \item The measure $\phifree_{\domain^\black,p_a,p_b,q}$ is stochastically
    increasing in $\domain$,
    \item The measure $\phiwired_{\domain^\black,p_a,p_b,q}$ is stochastically
    decreasing in $\domain$,
    \item The measure $\phifree_{\domain^\black,p_a,p_b,q}$ is stochastically dominated by $\phiwired_{\domain^\black,p_a,p_b,q}$.
\end{enumerate}
The first two properties imply the existence of the weak
limits
\[
	\phifree_{\domain^\black,p_a,p_b,q} \xrightarrow[\domain\nearrow\squarelattice]{} \phifree_{p_a,p_b,q}, \quad \text{and} \quad
	\phiwired_{\domain^\black,p_a,p_b,q} \xrightarrow[\domain\nearrow\squarelattice]{} \phiwired_{p_a,p_b,q},
\]
where~$\phifree_{p_a,p_b,q}$ and~$\phiwired_{p_a,p_b,q}$ are probability measures on $\{0,1\}^{E(\squarelattice^\black)}$.
The third property implies that~$\phifree_{p_a,p_b,q}$ is stochastically dominated by $\phiwired_{p_a,p_b,q}$.

It is well-known that at the line (see~\cite[Section~6]{Gri06}):
\begin{equation}\label{eq:self-dual-line}
	\tfrac{p_a}{1-p_a} \cdot \tfrac{p_b}{1-p_b}=q,
\end{equation}
the random-cluster model is self-dual: the distribution of $\eta^*$ in $\phiwired_{p_a,p_b,q}$
is identical to the distribution of $\eta+(1,0)$ in $\phifree_{p_a,p_b,q}$
(we shift  $\eta$ to the right by one so that it becomes
a spanning subgraph of $\squarelattice^\white$ rather than $\squarelattice^\black$).

The BKW coupling relates the random-cluster model at the self-dual line to the six-vertex model (see Fig.~\ref{fig:BKWcomplete}).
This enables us to derive from the delocalisation of the six-vertex model the following result that was first established in~\cite{DumSidTas16} (symmetric case) and~\cite{DumLiMan18} (asymmetric case).

\begin{theorem}[Continuity of the phase transition]
    \label{thm:continuity_asymetric}
    Fix~$1\leq q \leq 4$.
	Then:
	\begin{itemize}
		\item For any $p_a,p_b\in[0,1]$, $\phifree_{p_a,p_b,q}=\phiwired_{p_a,p_b,q}$,
		\item At the self-dual line~\eqref{eq:self-dual-line} neither
		$\eta$ nor $\eta^*$ contains an infinite cluster almost surely in $\phifree_{p_a,p_b,q}=\phiwired_{p_a,p_b,q}$.
	\end{itemize}
\end{theorem}

The first statement follows from the second for all~$(p_a,p_b)$ below the self-dual line~\eqref{eq:self-dual-line}, once we use
the above corollaries of the FKG inequality and observe that
$\phifree_{p_a,p_b,q}=\phiwired_{p_a,p_b,q}$ as soon as $\phiwired_{p_a,p_b,q}$ does not exhibit
an infinite cluster.
For the points~$(p_a,p_b)$ above the self-dual line~\eqref{eq:self-dual-line}, the first statement follows by a
dual argument.
Notice that the theorem does not imply that the phase transition occurs at the self-dual line.
This is known~\cite{BefDum12,DumLiMan18}, but our arguments do not rely on this.
The theorem directly implies continuity of the phase
transition of the Potts model with two, three, and four colours on the rectangular lattice.

\begin{remark}
	For $0<q<1$ and for $(p_a,p_b)$ on the self-dual line~\eqref{eq:self-dual-line},
	our proofs yield a construction of a self-dual shift-invariant full-plane Gibbs measure $\phi$
	of the random-cluster model
	in which neither $\eta$ nor $\eta^*$ percolates.
	However, we cannot interpret this measure as a full-plane limit with free or wired
	boundary conditions, by lack of a suitable FKG inequality.
	The same lack of monotonicity does not allow us to derive from this anything away from the self-dual line.
\end{remark}

\section{Delocalisation of Lipschitz functions on the triangular lattice}
\label{section:loop_on}

\subsection{Notation}
\label{sec:notation-lip}
Each vertex of $\hexlattice$ belongs to precisely one vertical edge.
The vertices of~$\hexlattice$ therefore have a natural bipartition
into those at the \emph{top} and those at the \emph{bottom} of a vertical edge.
We define~$\upvert(\hexlattice)$ as the part that consists of top endpoints:
\[
	\upvert(\hexlattice):=
	\{k+\ell e^{i\pi/3} \colon k,\ell \in \Z\}
	-i/\sqrt3
	\subset
	\vertices{\hexlattice};
\]
this set has the structure of a triangular lattice once endowed with the nearest-neighbour connectivity.
For a domain~$\domain$ on~$\hexlattice$, we write
\[
	\upvert(\domain):=(\upvert(\hexlattice)\cap \vertices\domain)\setminus \vertices{\partial\domain};
	\qquad
	\partial_\upvert\domain:=\upvert(\hexlattice)\cap \vertices{\partial\domain}.
\]
The natural dual to a site percolation on the triangular lattice is formed
by the complementary vertices.
For a given set~$\xi\subseteq\upvert(\domain)$
we write~$\xi^*:=\upvert(\domain)\setminus\xi$ for this dual set.

\subsection{Spin representation}
\label{sec:lip-spin-rep}

Lipschitz functions have a two-spin representation introduced by Manolescu and the first author~\cite{GlaMan21}.
This representation already appeared implicitly in~\cite{DomMukNie81} as a relation between the
loop~\O{2} and Ashkin--Teller models.
As we will show below in Lemma~\ref{lemma:fkg-spins-lip}, the marginals of this spin
representation satisfy the FKG inequality for all~$x\leq 1$.
This key property places Lipschitz functions in the framework of percolation
models and eventually enables the use the Zhang's non-coexistence
argument~\cite[Lemma~11.12]{Gri99a}.
At $x=1$, a variation of this strategy was realised in~\cite{GlaMan21}.

 \begin{definition}[Spin configurations]
	Let $\domain$ be a domain.
	A \emph{spin configuration} is a function $\sigma\in\{\pm 1\}^{\faces{\domain}}$;
	its \emph{domain wall} is the spanning subgraph~$\omega[\sigma]$ of~$\domain$ given by the set of disagreement edges of~$\sigma$ in $\edges{\domain}\setminus\edges{\partial\domain}$.
	We shall also write $\upvert[\sigma]$ for the set of
	vertices in $\upvert(\domain)$ incident to edges of $\omega[\sigma]$.
	We say that a pair of configurations~$\sigma^\black,\sigma^\white \in \{\pm 1\}^{\face(\domain)}$ is {\em consistent}, and denote this by~$\sigma^\black \perp \sigma^\white$, if
	\begin{equation}\label{eq:lip-ice}
		\text{either $\sigma^\black(u) = \sigma^\black(v)$ or $\sigma^\white(u) = \sigma^\white(v)$ for any adjacent $u,v\in\face(\domain)$.}
	\end{equation}
	This is equivalent to asking that the domain walls $\omega[\sigma^\black]$ and $\omega[\sigma^\white]$ are disjoint,
	and furthermore implies that $\upvert[\sigma^\black]$ and $\upvert[\sigma^\white]$ are disjoint.
	Write~$\spaceSpin(\domain)$ for the set of all consistent pairs of spin configurations on~$\domain$.
\end{definition}

The consistency relation is analogous to the {\em ice rule}~\eqref{eq:6v-ice} in the six-vertex model.

\begin{definition}[Spin measure]\label{def:spin-measure-lip}
	The spin measure on~$\domain$ with parameter~$x>0$ under free boundary conditions is a probability measure on~$\spaceSpin(\domain)$ defined by
	\begin{equation}
		\label{eq:lip-meas-via-vertices}
		\muSpin_{\domain,x}(\sigma^\black,\sigma^\white)
		=
		\tfrac1{Z_{\domain,x}}\cdot (x^2)^{|\upvert[\sigma^\black]\cup\upvert[\sigma^\white]|},
	\end{equation}
	where~$Z_{\domain,x}$ is the partition function.
	We call $\sigma^\black$ and $\sigma^\white$ \emph{black} and
\emph{white} spins, respectively.
	Let us also introduce fixed boundary conditions for~$\sigma^\black$, $\sigma^\white$, or both,
	by defining:
	\begin{align*}
		\muSpin_{\domain,x}^\blackp
			:=&
			\muSpin_{\domain,x}\left(\blank\mid\{\sigma^\black|_{\partial_\face \domain}\equiv +\}\right);
			\\
		\muSpin_{\domain,x}^\whitep
			:=&
			\muSpin_{\domain,x}\left(\blank\mid\{\sigma^\white|_{\partial_\face \domain}\equiv +\}\right);
			\\
		\muSpin_{\domain,x}^{\blackp\,\whitep}
			:=&
			\muSpin_{\domain,x}\left(\blank\mid\{\sigma^\black|_{\partial_\face \domain}\equiv\sigma^\white|_{\partial_\face \domain} \equiv +\}\right);
	\end{align*}
	similar definitions apply when $+$ is replaced by $-$.
\end{definition}

\begin{definition}
	Let $\domain$ denote a domain and $h\in\spaceLip^0(\domain)$ a Lipschitz function.
	Its \emph{spin representation} $(\sigma^\black[h],\sigma^\white[h]) \in \spaceSpin(\domain)$ is defined by
	\[
		\{\sigma^\black[h] = + \} = \{h \in \{0,1\} + 4\Z\};
		\qquad
		\{\sigma^\white[h] = + \} = \{h \in \{0,-1\} + 4\Z\}.
	\]
\end{definition}

	Observe that this implies $\omega[h]=\omega[\sigma^\black[h]]\cup\omega[\sigma^\white[h]]$.
The following proposition relates the spin measure to the random Lipschitz function and the loop~\O{2} model.

\begin{proposition}
	\label{prop:lip-spin-loop}
	Let $\domain$ be a domain and~$x>0$.
	Then,
	\begin{enumerate}
		\item $\muSpin_{\domain,x}^{\blackp\,\whitep}$ is the pushforward of $\muLip_{\domain,x}^0$
		along $h\mapsto (\sigma^\black[h],\sigma^\white[h])$, and
		\item $\muLoop_{\domain,2,x}$ is the pushforward of $\muSpin_{\domain,x}^{\blackp\,\whitep}$ along
		$(\sigma^\black,\sigma^\white)\mapsto \omega[\sigma^\black]\cup\omega[\sigma^\white]$.
	\end{enumerate}
\end{proposition}

\begin{proof}
	\begin{enumerate}[wide, labelwidth=!, labelindent=0pt]
		\item  It is straightforward that the map is a bijection.
		It suffices to show that it also preserves the weights of the configurations.
		Indeed, on each loop of~$\omega[h]$, the vertices of the two partite classes of~$\hexlattice$ alternate; since the loops cannot touch~$\partial\domain$, exactly one half of their vertices is in~$\upvert{(\domain)}$.
		In conclusion,
		\begin{equation}\label{eq:loop-length-via-y-vertices}
			|\omega[h]|=|\omega[\sigma^\black[h]]\cup\omega[\sigma^\white[h]]|=2|\upvert[\sigma^\black[h]]\cup\upvert[\sigma^\white[h]]|.
		\end{equation}
		\item The preimage of any element $\omega\in\spaceLoop{\domain}$ has cardinality
		$2^{\ell(\omega)}$, since the sign of either black or white spin changes along each loop.
		To prove that the map preserves the weight of each configuration
		up to this combinatorial factor, we use~\eqref{eq:loop-length-via-y-vertices} again.
		\qedhere
	\end{enumerate}
\end{proof}

\begin{remark*}
	Equation~\eqref{eq:lip-meas-via-vertices} suggests that
	the effective parameter is~$x^2$,
	and if $x\in [1/\sqrt2,1]$, then $x^2\in[1/2,1]$.
	From the point of view of the symmetric ($a=b$) six-vertex model, one should regard~$x^2$ as the ratio~$a/c$.
	The known critical value~$c/a=c/b=2$ for the six-vertex model then corresponds to
	the conjectured critical value~$x=1/\sqrt2$ for the loop~\O{2} model.
	Our method reveals a certain non-planarity of the interaction emerging~$x<1/\sqrt2$ and suggests localisation.
\end{remark*}

\subsection{Graphical representation and super-duality}
\label{sec:duality}

Fix some domain~$\domain$.
We now introduce our crucial new ingredient which allows us
to extend~\cite{GlaMan21} to the range $x\in[1/\sqrt2,1]$:
a graphical representation of~$\muSpin_{\domain,x}$ that satisfies a duality relation.
We will need an external source of randomness:
an independent family $U=(U_y)_{y\in \upvert(\domain)}$ of i.i.d.\ random variables
having the distribution $U([0,1])$.
With a slight abuse of notation, we incorporate this family into all our
existing measures without a change of notation.

For~$\sigma\in\{\pm1\}^{\face(\domain)}$ and~$A\subset\upvert(\domain)$,
we say that~$\sigma$ \emph{agrees} on~$A$, and write~$\sigma\perp A$, if, for every~$y\in A$,
the spin configuration~$\sigma$ assigns the same value to the three faces around~$y$;
we write~$\sigma\not\perp A$ otherwise.
If~$A=\{y\}$, then we simply write~$\sigma\perp y$ and~$\sigma\not\perp y$.
Observe that if $\sigma^\black\perp\sigma^\white$,
then at least one of $\sigma^\black\perp y$ and $\sigma^\white\perp y$ holds true for any~$y\in \upvert(\domain)$.

\begin{definition}[Black and white percolations] \label{def:b-w-perco-lip}
	Given a triplet $(\sigma^\black,\sigma^\white,U)$,
	the \emph{black percolation}~$\xi^\black$ and the \emph{white percolation}~$\xi^\white$ are subsets of~$\upvert(\domain)$ defined by
	\[
		\begin{cases}
			\text{$\xi^\black(y) = 0$ and $\xi^\white(y) = 1$} &\text{if $\sigma^\black\not\perp y$,}\\
			\text{$\xi^\black(y) = 1$ and~$\xi^\white(y) = 0$} &\text{if $\sigma^\white\not\perp y$,}\\
			\text{$\xi^\black(y) = \ind{U_y\leq x^2}$ and~$\xi^\white(y) = \ind{U_y> 1-x^2}$} &\text{if $\sigma^\black\perp y$ and~$\sigma^\white\perp y$.}
		\end{cases}
	\]
\end{definition}

\begin{figure}
	\includegraphics[width=0.9\textwidth]{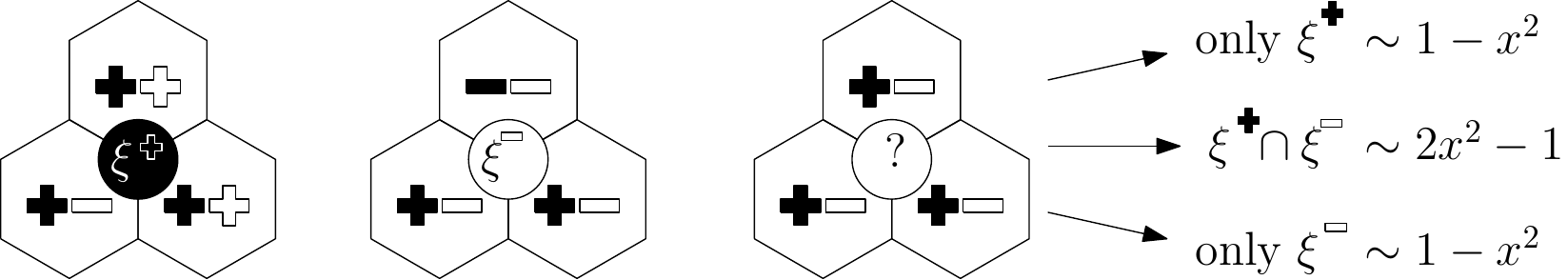}
	\caption{%
		The sampling rule for~$(\xi^\black,\xi^\white)$ at one $\upvert$-vertex given~$(\sigma^\black,\sigma^\white)$.
	}
	\label{fig:activated-vertices}
\end{figure}

See Fig.~\ref{fig:activated-vertices} and~\ref{fig:loop_percolations} for an illustration.
By definition, we have $\sigma^\black \perp \xi^\black$,
and therefore~$\xi^\black$ is the disjoint union of
\begin{align*}
	\xi^\blackp:= \{y\in \xi^\black\colon\text{$\sigma^\black \equiv + $ around $ y$}\}
	\quad \text{and} \quad
	\xi^\blackm:= \{y\in \xi^\black\colon\text{$\sigma^\black \equiv - $ around $ y$}\}.
\end{align*}
In a similar fashion, the set $\xi^\white$ is the disjoint union of
$\xi^\whitep$ and $\xi^\whitem$.

By integrating over $U$, we observe that, conditioned on $(\sigma^\black,\sigma^\white)$, the set $\xi^\black$ is an independent site
percolation with the opening probability at a vertex $y$ being~$0$, $1$, and~$x^2$
respectively in the three cases (this fact is used to prove Lemma~\ref{lemma:lip-marginal-black}).
Until Section~\ref{sec:extremality-lip}, we will need only this joint distribution of $(\sigma^\black,\sigma^\white,\xi^\black)$.
The crucial property of the coupling of~$\xi^\black$ and~$\xi^\white$ that implies delocalisation is the following {\em super-duality} (see Fig.~\ref{fig:loop_percolations}):
whenever $x\in[1/\sqrt2,1]$, we deterministically have
\begin{equation}
	\label{eq:super_duality}
	\upvert(\domain)=\xi^\blackp\cup\xi^\blackm\cup\xi^\whitep\cup\xi^\whitem=\xi^\black\cup\xi^\white.
\end{equation}
Indeed, this is an immediate consequence of the definitions
since $x^2 \geq 1-x^2$ when~$x\geq 1/\sqrt{2}$.
At $x=1/\sqrt 2$, every vertex $y\in\upvert(\domain)$
is contained in exactly one of $\xi^\black$ and $\xi^\white$ and~\eqref{eq:super_duality} turns into an exact duality.

\begin{lemma}\label{lemma:lip-marginal-black}
	Let~$\domain$ be a domain and~$0<x\leq 1$. Then,
	\begin{align}\label{eq:loopexpansion}
		\muSpin_{\domain,x}(\sigma^\black,\sigma^\white,\xi^\black) = \frac1{Z_{\domain,x}}
		&\cdot (x^2)^{|\xi^\black|}
		\cdot (1-x^2)^{|(\xi^\black)^*\setminus\upvert[\sigma^\black]|}
		\cdot (x^2)^{|\upvert[\sigma^\black]|}
		\\
		&\cdot \ind{\sigma^\black\perp\xi^\black}
		\cdot \ind{\sigma^\white\perp(\xi^\black)^*}
		\cdot \ind{\sigma^\black\perp\sigma^\white}.\nonumber
	\end{align}
	The laws of~$(\sigma^\black,\sigma^\white,\xi^\black)$
	under~$\muSpin_{\domain,x}^\blackp$, $\muSpin_{\domain,x}^\whitep$,
	and~$\muSpin_{\domain,x}^{\blackp\,\whitep}$ are obtained by inserting
	indicators for the boundary values, dropping the factor~$\ind{\sigma^\black\perp\sigma^\white}$, and updating the partition function.
\end{lemma}

\begin{proof}
	For any
	spin configurations $\sigma^\black,\sigma^\white\in\{\pm1\}^{\faces{\domain}}$,
	\begin{equation}
		\label{eq:spin-meas-via-Y-lip}
		\muSpin_{\domain,x}(\sigma^\black,\sigma^\white) = \frac1{Z_{\domain,x}}
		\cdot (x^2)^{|\upvert[\sigma^\black]|}
		\cdot (x^2)^{|\upvert[\sigma^\white]|}
		\cdot \ind{\sigma^\black\perp\sigma^\white}.
	\end{equation}
	The definition of $\xi^\black$ implies that
	conditional on $(\sigma^\black,\sigma^\white)$, the probability of observing a particular site percolation
	$\xi^\black$ equals
	\begin{multline}
		(x^2)^{|\xi^\black\setminus (\upvert[\sigma^\black]\cup\upvert[\sigma^\white])|}
		\cdot
		(1-x^2)^{|(\xi^\black)^*\setminus (\upvert[\sigma^\black]\cup\upvert[\sigma^\white])|}
		\cdot \ind{\sigma^\black\perp\xi^\black}
		\cdot \ind{\sigma^\white\perp(\xi^\black)^*}
		\\
		\label{eq:rewriting-xi-Y-lip}
		=
		(x^2)^{|\xi^\black\setminus \upvert[\sigma^\white]|}
		\cdot
		(1-x^2)^{|(\xi^\black)^*\setminus\upvert[\sigma^\black]|}
		\cdot \ind{\sigma^\black\perp\xi^\black}
		\cdot \ind{\sigma^\white\perp(\xi^\black)^*},
	\end{multline}
	where we use that the indicators imply~$\xi^\black \cap \upvert[\sigma^\black] = \emptyset$ and~$\upvert[\sigma^\white]\subseteq \xi^\black$.
	The latter also gives $|\xi^\black|=|\xi^\black\setminus \upvert[\sigma^\white]|+|\upvert[\sigma^\white]|$, and
	we obtain~\eqref{eq:loopexpansion} by taking the product of~\eqref{eq:spin-meas-via-Y-lip} and~\eqref{eq:rewriting-xi-Y-lip}.

	Other boundary conditions are clearly enforced
	by inserting more indicators and updating the partition functions.
	It suffices to establish the claim that the restriction~$\sigma^\black\perp\sigma^\white$ becomes redundant:
	for any triplet $(\sigma^\black,\sigma^\white,\xi^\black)$,
	if either $\sigma^\black|_{\partial_\face\domain}\equiv +$ or $\sigma^\white|_{\partial_\face\domain}\equiv +$, then
	\[
		\ind{\sigma^\black\perp\xi^\black}
		\cdot \ind{\sigma^\white\perp(\xi^\black)^*}
		\cdot \ind{\sigma^\black\perp\sigma^\white}
		=
		\ind{\sigma^\black\perp\xi^\black}
		\cdot \ind{\sigma^\white\perp(\xi^\black)^*}.
	\]
	Assume, in order to derive a contradiction, that there exists a triplet for which the two expressions are not equal.
	Then~$\sigma^\black\perp\xi^\black$ and $\sigma^\white\perp(\xi^\black)^*$, and simultaneously there must exist an edge $yz\in\edges{\domain}$ which belongs to both~$\omega[\sigma^\black]$ and~$\omega[\sigma^\white]$.
	This edge cannot be incident to $\partial\domain$ since either~$\sigma^\black$ or~$\sigma^\white$ is constant on~$\partial_\face\domain$.
	Thus, one endpoint of $yz$, say~$y$, lies in $\upvert(\domain)$.
	However, $\sigma^\black\perp\xi^\black$ and~$\sigma^\white\perp(\xi^\black)^*$ imply~$y\not\in (\xi^\black \cup (\xi^\black)^*) = \upvert(\domain)$.
	This is the desired contradiction, which proves the claim.
\end{proof}

\subsection{The Markov property}
\label{sec:markov-lip}

The dual~$\domain^*$ of a domain $\domain$ is defined as the graph
on~$\vertices{\domain^*}:=\face(\domain)$ with edges~$\edges{\domain^*}$ linking
adjacent faces of~$\domain$.  Given~$\xi\subseteq\upvert(\domain)$,
define~$\triedges(\xi)$ to be the spanning subgraph of~$\domain^*$ whose set of
edges is the union of triplets of edges forming (upward oriented) triangles
around vertices in~$\xi$: for any~$uv\in \edges{\domain^*}$, we
have~$uv\in\triedges(\xi)$ if and only if one of the common vertices of the
faces~$u$ and~$v$ is in~$\xi$.  We also
define~$\triedges(\domain):=\triedges(\upvert(\domain))$.

\begin{lemma}[Sampling~$\sigma^\white$ given~$\xi^\black$]
	\label{lemma:loop_decoupling}
	Let~$\domain$ be a domain and $0<x\leq 1$.
	Consider~$\tau^\black \in \{\pm 1\}^{\face(\domain)}$
	and~$\zeta^\black\subset\upvert(\domain)$ such
	that~$\muSpin_{\domain,x}^\blackp(\sigma^\black=\tau^\black,\xi^\black=\zeta^\black)>0$.
	Then,
	\[
		\text{the law of $\sigma^\white$ under }
		\muSpin_{\domain,x}^\blackp(\blank |  \sigma^\black=\tau^\black,\xi^\black=\zeta^\black)
	\]
	is given by independent fair coin flips valued $\pm$ for the
	connected components of~$\triedges((\xi^\black)^*)$.
	If the boundary condition $\blackp$ is replaced by $\whitep$,
	then the only difference is that all clusters of~$\triedges((\xi^\black)^*)$
	intersecting~$\partial_\face\domain$ are deterministically assigned the value $+$.
\end{lemma}

\begin{proof}
	By Lemma~\ref{lemma:lip-marginal-black},
	\begin{align}
		\muSpin_{\domain,x}^\blackp(\sigma^\white\, \vert \, \sigma^\black,\xi^\black)
		\propto \ind{\sigma^\white\perp(\xi^\black)^*} =
		\ind{\text{$\sigma^\white$ is constant on each cluster of $\triedges((\xi^\black)^*)$}}.
	\end{align}
	The boundary condition $\whitep$ introduces the extra indicator~$\ind{\sigma^\white|_{\partial_\face\domain} \equiv +}$.
\end{proof}

\begin{figure}
	\includegraphics[width=0.31\textwidth]{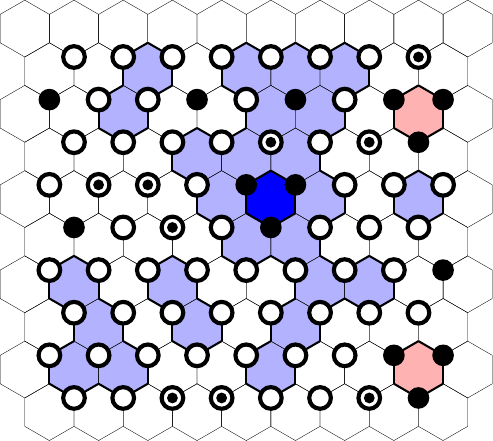}\,\,%
	\includegraphics[width=0.31\textwidth]{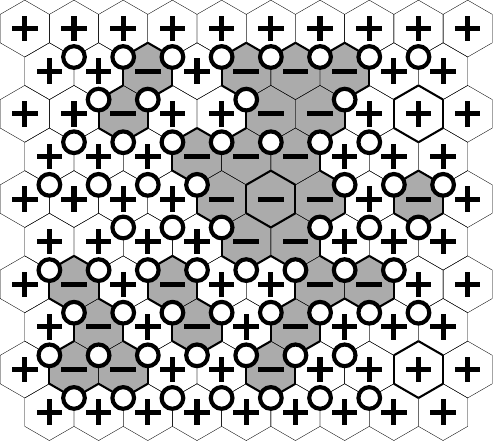}\,\,%
	\includegraphics[width=0.31\textwidth]{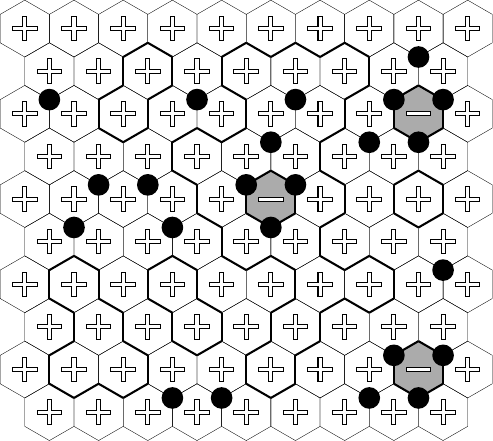}
	\caption{%
	A Lipschitz function~$h$ with the corresponding black and white spin configurations~$(\sigma^\black,\sigma^\white)$ and a sample of the site percolations~$(\xi^\black,\xi^\white)$ on $\upvert(\domain)$.
	\textsc{Left}: Colours of the faces describe the heights:
	the heights $-2$ to $1$ are dark blue, light blue, white, and red respectively.
	Super-duality holds true: every~$\upvert$-type vertex is either black or white, and some are assigned both colours.
	\textsc{Middle}: The black spin configuration~$\sigma^\black$ is determined by~$h$, the white percolation~$\xi^\white$ is sampled given~$h$.
	All $\upvert$-vertices belonging to the loops in~$\omega[\sigma^\black]$ must be in~$\xi^\white\setminus\xi^\black$.
	\textsc{Right}:
	The white spin configuration~$\sigma^\white$ together with the black percolation~$\xi^\black$.}
	\label{fig:loop_percolations}
\end{figure}

\begin{lemma}[Domain Markov property]
	\label{lemma:loop_markov}
	Let $\domain,\domain'$ be domains with $\domain\subset \domain'\setminus
	\partial\domain'$ and let~$0<x\leq 1$.  Then, under
	\[\muSpin_{\domain',x}(\blank|\partial_\upvert\domain\subseteq
	\xi^\blackp),\] the law of the triplet
	$(\sigma^\black,\sigma^\white,\xi^\black)$ restricted to faces in
	$\face(\domain)$ and vertices in $\upvert(\domain)$:
	\begin{enumerate}
		\item Is independent of the restriction to the complementary faces and vertices, and
		\item Is precisely equal to~$\muSpin_{\domain,x}^{\blackp}$.
	\end{enumerate}
	In particular, this remains true when boundary conditions are imposed on $\muSpin_{\domain',x}$.
\end{lemma}

\begin{proof}
	Let~$A$ denote the set of faces in~$\domain'$ that are adjacent to edges in~$\partial\domain$.
	By Lemma~\ref{lemma:lip-marginal-black},
	\begin{align*}
		\muSpin_{\domain',x}(\sigma^\black,\sigma^\white,\xi^\black \, & d\vert\, \partial_\upvert\domain\subseteq \xi^\blackp)
		\\
		\propto
		&(x^2)^{|\xi^\black\setminus\partial_\upvert\domain|}
		\cdot(1-x^2)^{|(\xi^\black)^*\setminus(\upvert[\sigma^\black]\cup\partial_\upvert\domain)|}
		\cdot (x^2)^{|\upvert[\sigma^\black]\setminus\partial_\upvert\domain|}
		\\
		\cdot
		&\ind{\sigma^\black\perp(\xi^\black\setminus\partial_\upvert\domain)}
		\cdot \ind{\sigma^\white\perp((\xi^\black)^*\setminus\partial_\upvert\domain)}
		\cdot \ind{\sigma^\black\perp\sigma^\white}
		\cdot \ind{\partial_\upvert\domain\subset\xi^\black}
		\cdot \ind{\left.\sigma^\black\right|_A\equiv +}.
	\end{align*}
	Here we used that~$\partial_\upvert\domain \subseteq \xi^\black$ which implies~$\sigma^\black|_A\equiv +$ and hence~$\sigma^\black\perp\partial_\upvert\domain$.

	Note that all factors except the last three can be written as a product of two terms:
	one depending only on $\face(\domain)$ and~$\upvert(\domain)$
	and the other only on $\face(\domain')\setminus \face(\domain)$ and~$\upvert(\domain')\setminus \upvert(\domain)$.
	The independence property then follows by noting that
	(similarly to Lemma~\ref{lemma:lip-marginal-black}) the condition~$\sigma^\black\perp\sigma^\white$ is redundant everywhere except on~$\partial\domain'$,
	and therefore can be replaced by a factor depending only on $\face(\domain')\setminus \face(\domain)$.
	Finally, taking the product of the parts of the terms that depend on~$\face(\domain)$ and~$\upvert(\domain)$, we obtain precisely $\muSpin_{\domain,x}^\blackp$, by Lemma~\ref{lemma:lip-marginal-black}.
\end{proof}

\subsection{The FKG inequality}
\label{subsec:loop_FKG}

\emph{Positive association} of probability measures plays an important role in the analysis of lattice models.
This property is often referred to as the {\em FKG inequality}, after an influential work of
Fortuin, Kasteleyn, and Ginibre~\cite{ForKasGin71} that provided a general framework for establishing positive association.

Below we give the relevant definitions in a general form.
Let $A:=\prod_{i\in I}A_i$ denote a set which is a product of
subsets $A_i\subset\mathbb R$ over the countable index set $I\ni i$.
This set inherits the standard pointwise partial ordering $\leq$
and the pointwise minimum $\wedge$ and maximum $\vee$ operations from $\mathbb R^I\supset A$.
For simplicity we call $A$ a \emph{product lattice},
and we call it a \emph{finite product lattice} if $I$ is finite.
Let $\P$ denote a probability measure and $\E$ the corresponding expectation functional.

\begin{enumerate}
	\item Let~$A$ be a product lattice and~$X$ an~$A$-valued random variable.
	We say that $X$ is \emph{positively associated}
	or that it \emph{satisfies the FKG inequality} if,
	for any bounded increasing functions
	$f,g:A\to\R$,
	\[
		\E[f(X)\cdot g(X)]\geq\E[f(X)]\cdot\E[g(X)].
	\]
	We also assign this property to $\P$ if the sample space is a product
	lattice and if the identity map satisfies the FKG inequality.

	\item Let~$A$ be a finite product lattice.
	A function $f:A\to(0,\infty)$ is said to \emph{satisfy the FKG lattice condition}
	if, for any~$a,b\in A$,
	\begin{equation}\label{eq:fkg-lattice}
		f(a\vee b)\cdot f(a\wedge b)\geq f(a)\cdot f(b).
	\end{equation}
	An $A$-valued random variable~$X$ is said to satisfy \emph{the FKG lattice condition}
	if~\eqref{eq:fkg-lattice} holds for the function $a\mapsto\P(X=a)$.
	We also extend this notion to $\P$ if $X$ is the identity map.
\end{enumerate}

A major contribution of~\cite{ForKasGin71} is the observation that the FKG
lattice condition implies the FKG inequality when~$a\mapsto\P(X=a)$ is strictly positive (see~\cite[Theorem~2.19]{Gri06}).
The former is often straightforward to check,
which has allowed to establish the FKG inequality for a
variety of models.
We also refer to~\cite[Section~4]{GeoHagMae01} for a survey.

\begin{remark}
	Strassen's theorem (see~\cite[Theorem~4.6; text below Definition~4.10]{GeoHagMae01})
	implies that an $A$-valued random variable $X$ satisfies the FKG inequality
	if and only if
	\[
		\E[f(X)\cdot g(X)]\geq\E[f(X)]\cdot\E[g(X)]
	\]
	for any bounded increasing functions $f,g:A\to\R$ \emph{which depend on finitely many
	coordinates in the index set $I$}.
	As a consequence, the FKG inequality is preserved under weak limits.
	These facts are used in this text without notice.
\end{remark}

Recall the definitions of~$\sigma^\black$, $\xi^\blackp$, and~$\xi^\blackm$.
For brevity, we shall write \[(\sigma^\black,\xi^\blackp,-\xi^\blackm)
:=
(\sigma^\black,1_{\xi^\blackp},-1_{\xi^\blackm})\]
when no confusion is likely to arise.
This puts the triplet in the framework of a random variable
taking values in a set of real-valued functions as described above.
The key novelty of the current article consists in a proof of the FKG inequality for this triplet.

\begin{proposition}[FKG for triplets]
	\label{prop:fkg-joint-lip}
	Let $\domain$ be a domain and $0< x\leq 1$.
	Then, under~$\muSpin_{\domain,x}$, the triplet $(\sigma^\black,\xi^\blackp,-\xi^\blackm)$ satisfies the FKG inequality.
	The same holds also for $(\sigma^\white,\xi^\whitep,-\xi^\whitem)$ and when boundary conditions of the form $\blackp,\blackm$  and/or $\whitep,\whitem$ are imposed.

	In fact, the triplet $(\sigma^\black,\xi^\blackp,-\xi^\blackm)$ also satisfies the FKG inequality
	under the measures
	\[
		\muSpin_{\domain,x}(\blank|\{\sigma^\black|_A=\tau\})
		\quad
		\text{and}
		\quad
		\muSpin_{\domain,x}^\whitep(\blank|\{\sigma^\black|_A=\tau\})
	\]
	for any $A\subset \faces{\domain}$ and $\tau:A\to\{\pm1\}$.
\end{proposition}

What matters mostly in this proposition is the fact that $\xi^\blackp$ and
$\xi^\blackm$ satisfy the FKG inequality.  At~$x=1$, $\sigma^\black$
determines~$\xi^\black$, and the FKG inequality for the triplet follows from the
FKG lattice condition for $\sigma^\black$ proven in~\cite{GlaMan21}.  We start
by extending the FKG lattice condition for~$\sigma^\black$ to all~$0<x\leq 1$
and derive Proposition~\ref{prop:fkg-joint-lip} via an extra argument
from~\cite[Proof of Lemma~3.2]{Lam22}.

\begin{lemma}[FKG for spins]
	\label{lemma:fkg-spins-lip}
	Let $\domain$ be a domain and let $0<x\leq 1$.
	Then, under~$\muSpin_{\domain,x}$, $\sigma^\black$ satisfies the FKG lattice condition.
	The same holds also for~$\sigma^\white$ and when boundary conditions of the form $\blackp,\blackm$  and/or $\whitep,\whitem$ are imposed.
\end{lemma}

\begin{proof}
	By symmetry, we can focus on~$\sigma^\black$.
	It suffices to prove the statement for~$\muSpin_{\domain,x}$ and~$\muSpin_{\domain,x}^\whitep$.
	Indeed, the boundary conditions $\whitem$ are analogous to $\whitep$,
	and the FKG lattice condition is preserved when imposing black boundary
	conditions of any kind.

	We start with an important abstract ingredient.
	Let $G=(V,E)$ be a finite graph.
	For any $a:E\to[0,1]$, write
	$\Zising(a)$ for the Ising model partition function given by
	\begin{equation}\label{eq:fk-ising}
		\Zising(a):=\sum_{\sigma^\white\in\{+,-\}^V}\prod_{xy\in E \, \colon \sigma_x^\white\neq\sigma_y^\white}
		%(\ind{\sigma_x^\white = \sigma_y^\white} + a_{xy}\cdot \ind{\sigma_x^\white \neq \sigma_y^\white}).
		\mathrm{w}_{xy}.
%		\quad
%		\text{where}
%		\quad
%		w_{xy}(\mathrm{w}):=
%		\begin{cases}
%			1 &\text{if $\sigma_x^\white=\sigma_y^\white$,}\\
%			a_{xy} &\text{if $\sigma_x^\white\neq\sigma_y^\white$.}
%		\end{cases}
	\end{equation}
	The function
	\(
		\mathrm{w}\mapsto\Zising(\mathrm{w})
	\)
	is clearly increasing in $\mathrm{w}$,
	and satisfies the FKG lattice condition due to the second Griffiths inequality~\cite[Lemma~6.1]{LamOtt21}.

	We now express~$\muSpin_{\domain,x}(\sigma^\black)$ using the partition function of the Ising model on~$\domain^*$.
	Indeed, summing~\eqref{eq:lip-meas-via-vertices} over all possible values for~$\sigma^\white$ and using that~$2|\upvert[\sigma^\white]|=|\omega[\sigma^\white]\cap\upvert(\domain)|$,
	we get
	\begin{align}
		\muSpin_{\domain,x}(\sigma^\black) \cdot Z_{\domain,x}
			&=
			(x^2)^{|\upvert[\sigma^\black]|} \cdot \sum_{\sigma^\white\in\{\pm1\}^{\faces{\domain}}}\ind{\sigma^\black\perp\sigma^\white}\cdot(x^2)^{|\upvert[\sigma^\white]|}
			\\&= (x^2)^{|\upvert[\sigma^\black]|} \cdot \Zising(\mathrm{w}(\sigma^\black)),\label{eq:fkg-proof-measure-via-ising}
	\end{align}
	where $\mathrm{w}(\sigma^\black):\edges{\domain^*}\to[0,1]$ is defined by
	\[
		\mathrm{w}(\sigma^\black)_{uv} =
			J_{uv}
			\cdot
			\ind{\sigma^\black(u)= \sigma^\black(v)}
			;
			\qquad
			J_{uv}:=
			\begin{cases}
				x &\text{if $uv\in \triedges(\domain)$,}\\
				1 &\text{otherwise.}
			\end{cases}
	\]
%	\begin{align}
%		\begin{cases}
%			0&\text{if }\sigma^\black(u)\neq \sigma^\black(v),\\
%			x&\text{if }\sigma^\black(u)= \sigma^\black(v) \text{ and } uv\in (Y(\domain))^*,\\
%			1&\text{if }\sigma^\black(u)= \sigma^\black(v) \text{ and } uv\not\in (Y(\domain))^*.
%		\end{cases}
%	\end{align}
	The function $\sigma^\black\mapsto (x^2)^{|\upvert[\sigma^\black]|}$ in~\eqref{eq:fkg-proof-measure-via-ising} satisfies the FKG
	lattice condition for all~$0<x\leq 1$,
	and therefore it suffices to prove the same for~$\sigma^\black\mapsto  \Zising(\mathrm{w}(\sigma^\black))$.
	We claim that the following factorisation holds:
	\begin{equation}
		\label{eq:fkg-proof-ising-decomp-pm}
		2\Zising(\mathrm{w}(\sigma^\black))= \Zising(\mathrm{w}^+(\sigma^\black))\cdot\Zising(\mathrm{w}^-(\sigma^\black)),
	\end{equation}
	where~$\mathrm{w}^\pm(\sigma^\black)$ is defined by
	\[
		\mathrm{w}^\pm(\sigma^\black)_{uv} = J_{uv} \cdot\ind{\sigma^\black(u)= \sigma^\black(v)=\pm}.
	\]
	Fix~$\sigma^\black$.
	The global strategy to prove~\eqref{eq:fkg-proof-ising-decomp-pm}
	is to view $\sigma^\white$ conditional on $\sigma^\black$ as the product
	of two independent Ising models.

	Observe first that whenever~$\sigma^\white\perp\sigma^\black$,
	we may partition the domain walls~$\omega[\sigma^\white]$ into
	connected components that lie inside $\{\sigma^\black=+\}$ and those that lie inside $\{\sigma^\black=-\}$:
	\[
		\omega^+[\sigma^\white]:=\{e\in \omega[\sigma^\white]\colon\text{$\sigma^\black \equiv +$ on $e^*$}\}
		\quad
		\text{and}
		\quad
		\omega^-[\sigma^\white]:=\{e\in \omega[\sigma^\white]\colon\text{$\sigma^\black \equiv -$ on $e^*$}\}.
	\]
	Since~$\domain$ is simply-connected, we can find a spin configuration $\zeta^+:\faces\domain\to \{\pm\}$
	such that $\omega[\zeta^+]=\omega^+[\sigma^\white]$.
	In fact, there exists exactly one other such spin configuration, namely $-\zeta^+$.
	Similar considerations apply to $\zeta^-:=\sigma^\white/\zeta^+$.
	Observe that $(\zeta^+,\zeta^-)$ and $(-\zeta^+,-\zeta^-)$ are the only two
	pairs which separate the domain walls in the prescribed way and which factorise $\sigma^\white$.
	Conversely, if $\zeta^+,\zeta^-:\faces\domain\to \{\pm 1\}$ are two spin configurations such that
	$\omega^+[\zeta^+]=\omega[\zeta^+]$ and~$\omega^-[\zeta^-]=\omega[\zeta^-]$, then $\sigma^\white:=\zeta^+\zeta^-$
	satisfies $\sigma^\white\perp\sigma^\black$.
	In addition, the contribution of~$\sigma^\white$
	to~$\Zising(\mathrm{w}(\sigma^\black))$ is equal to the product of the contributions
	of~$\zeta^+$ to~$\Zising(\mathrm{w}^+(\sigma^\black))$ and of~$\zeta^-$
	to~$\Zising(\mathrm{w}^-(\sigma^\black))$.  Summing over all
	pairs~$(\zeta^+,\zeta^-)$ yields~\eqref{eq:fkg-proof-ising-decomp-pm}.

	By~\eqref{eq:fkg-proof-ising-decomp-pm}, it suffices to prove that the map $\sigma^\black\mapsto \Zising(a^\pm(\sigma^\black))$
	satisfies the FKG lattice condition.
	Observe that
	\[
		\mathrm{w}^+(\sigma\vee\sigma')\geq \mathrm{w}^+(\sigma)\vee \mathrm{w}^+(\sigma');
		\qquad
		\mathrm{w}^+(\sigma\wedge\sigma')=\mathrm{w}^+(\sigma)\wedge \mathrm{w}^+(\sigma').
	\]
	Therefore, monotonicity and the FKG lattice condition for~$a\mapsto\Zising(a)$ imply
	\begin{align}
		\Zising(\mathrm{w}^+(\sigma\vee\sigma'))
		\cdot
		\Zising(\mathrm{w}^+(\sigma\wedge\sigma'))
		&\geq
		\Zising(\mathrm{w}^+(\sigma)\vee \mathrm{w}^+(\sigma'))
		\cdot
		\Zising(\mathrm{w}^+(\sigma)\wedge \mathrm{w}^+(\sigma'))\\
	   & \geq
		\Zising(\mathrm{w}^+(\sigma))
		\cdot
		\Zising(\mathrm{w}^+(\sigma')).
	\end{align}
	Analogously, the same holds for~$\mathrm{w}^-(\sigma)$.
	This proves the lemma for the measure $\muSpin_{\domain,x}$.

	It remains to prove the same results for $\muSpin_{\domain,x}^\whitep$.
	Notice that the distribution of $\sigma^\black$ under this measure is the same as under
	\[
		\muSpin_{\domain,x}(\blank|
			\{\sigma^\white|_{\partial_\face\domain}\equiv +\}
			\cup
			\{\sigma^\white|_{\partial_\face\domain}\equiv -\}
			);
	\]
	we now can run the same proof as above with
	\(
		J'_{uv}:=J_{uv}\cdot\ind{
			uv\not\subseteq\partial_\face\domain
		}
	\).
\end{proof}

We are now ready to prove Proposition~\ref{prop:fkg-joint-lip}.

\begin{proof}[Proof of Proposition~\ref{prop:fkg-joint-lip}]
	By symmetry, we can focus on~$(\sigma^\black,\xi^\blackp,-\xi^\blackm)$. Let $\P$ denote~$\muSpin_{\domain,x}$.
	Our proof follows the pattern of the proof of~\cite[Lemma~3.2]{Lam22}:
	we first prove two auxiliary claims,
	then tie them together with Lemma~\ref{lemma:fkg-spins-lip}
	to yield the proposition.

	\begin{claim}
		\label{claim:one}
	Conditional on $\sigma^\black$, the pair $(\xi^\blackp,-\xi^\blackm)$ satisfies the FKG inequality.
	\end{claim}

	\begin{proof}
	Fix~$\sigma^\black=\tau^\black$.
	Since $\xi^\blackp$ and $\xi^\blackm$ are conditionally independent,
	it suffices to prove the FKG inequality for the former.
	Define~$E^+$ as the set of edges~$uv\in\edges{\domain^*}$ such that~$(\sigma^\black(u),\sigma^\black(v))\neq (+,+)$.
	Recall from the proof of Lemma~\ref{lemma:fkg-spins-lip} that the conditional law of~$\sigma^\white$ can be written as the product of two independent Ising configurations~$\zeta^+$ and~$\zeta^-$ such that~$\zeta^+$ agrees along edges in~$E^+$.
	Summing over possible values of~$\zeta^+$ and applying Lemma~\ref{lemma:lip-marginal-black} yields
	\begin{align}
		\P(\xi^\blackp \big| \sigma^\black = \tau^\black)
			&\propto \sum_{\zeta^+\in\{\pm1\}^{\faces{\domain}}}(x^2)^{|\xi^\blackp|}\cdot (1-x^2)^{|(\xi^\blackp)^*|}\cdot\ind{\zeta^+\perp (\xi^\blackp)^*}\cdot\ind{\triedges(\xi^\blackp)\cap E^+=\emptyset} \nonumber \\
			&= 2^{k((\xi^\blackp)^*)}\cdot(x^2)^{|\xi^\blackp|}\cdot (1-x^2)^{|(\xi^\blackp)^*|}\cdot\ind{\triedges(\xi^\blackp)\cap E^+=\emptyset},
			\label{eq:fkg-proof-xi-given-sigma-lip}
	\end{align}
	where~$k((\xi^\blackp)^*)$ denotes the number of clusters in the graph on~$\faces\domain$ given by the union of~$\triedges((\xi^\blackp)^*)$ and~$E^+$.
	It is a standard fact that measures of this type satisfy the FKG lattice condition;
	see~\cite[Theorem~3.8]{Gri06} for the proof for FK-percolation.
	This proves Claim~\ref{claim:one}.
	\renewcommand\qedsymbol{}
	\end{proof}

	\begin{claim}
		\label{claim:two}
		The distribution of $(\xi^\blackp,-\xi^\blackm)$ is stochastically increasing in $\sigma^\black$.
	\end{claim}

	\begin{proof}
	Since $\xi^\blackp$ and $\xi^\blackm$ are conditionally independent,
	it suffices to show that the conditional distribution of $\xi^\blackp$
	is increasing in $\sigma^\black$.
	Indeed, $E^+$ is decreasing in~$\sigma^\black$, and so are~$\ind{\triedges(\xi^\blackp)\cap E^+=\emptyset}$ and~$k((\xi^\blackp)^*)$.
	This allows one to check the Holley criterion for~\eqref{eq:fkg-proof-xi-given-sigma-lip} and prove Claim~\ref{claim:two}.
	\renewcommand\qedsymbol{}
	\end{proof}

	The remainder of the proof of Proposition~\ref{prop:fkg-joint-lip} is standard.
	Let $f$ and $g$ denote two bounded increasing functions.
	Let $\P^{\sigma^\mblack}$ denote the measure $\P(\blank|\sigma^\black)$,
	and let $f^{\sigma^\mblack}(\xi^\black):=f(\sigma^\black,\xi^\black)$ and
	$g^{\sigma^\mblack}(\xi^\black):=g(\sigma^\black,\xi^\black)$.
	Claim that
	\begin{align*}
		\P(fg)=\int\P^{\sigma^\mblack}(f^{\sigma^\mblack}\cdot g^{\sigma^\mblack}) d\P(\sigma^\black)
		&\geq
		\int\P^{\sigma^\mblack}(f^{\sigma^\mblack}) \cdot \P^{\sigma^\mblack}(g^{\sigma^\mblack}) d\P(\sigma^\black)
		\\
		&\geq
		\int\P^{\sigma^\mblack}(f^{\sigma^\mblack})d\P(\sigma^\black)\int\P^{\sigma^\mblack}(g^{\sigma^\mblack}) d\P(\sigma^\black)
		=\P(f)\P(g).
	\end{align*}
	Indeed, the first inequality follows from Claim~\ref{claim:one} since~$f^{\sigma^\mblack}$ and~$g^{\sigma^\mblack}$ are increasing in $(\xi^\blackp,-\xi^\blackm)$;
	the second inequality follows from Lemma~\ref{lemma:fkg-spins-lip} since~$\P^{\sigma^\mblack}(f^{\sigma^\mblack})$ and~$\P^{\sigma^\mblack}(g^{\sigma^\mblack})$ are increasing in~$\sigma^\black$ by Claim~\ref{claim:two};
	the equalities come from the tower property.

	The proof is identical when~$\P=\muSpin_{\domain,x}(\blank|\{\sigma^\black|_A=\tau\})$;
	in particular, the conditioning does not break the FKG inequality for $\sigma^\black$
	because we proved the FKG lattice condition for this lattice-valued random variable.
	For boundary conditions~$\whitep$ or $\whitem$ we just need to add to~$E^+$ all edges between adjacent boundary faces in~$\partial_\face\domain$.
\end{proof}

We write $\mu \preceq_\black \nu$ if the distribution of the triplet $(\sigma^\black,\xi^\blackp,-\xi^{\blackm})$
in $\mu$ is stochastically dominated by its distribution in $\nu$.
The Markov property and the FKG inequality allow us to draw a comparison between boundary conditions.
% For measures~$\mu,\nu$ on consistent quadruples~$(\sigma^\black,\sigma^\white,\xi^\black,\xi^\white)$,
% we write~$\mu \preceq_\black \nu$ if the marginal of~$\mu$ on~$(\sigma^\black,\xi^\blackp)$ stochastically dominates the marginal of~$\nu$ on~$(\sigma^\black,\xi^\blackp)$.

\begin{corollary}[Comparison between boundary conditions]
	\label{cor:cbc-lip}
	Let $\domain$ and $\domain'$ be domains such that~$\domain\subset\domain'\setminus\partial\domain'$, and let~$0<x\leq 1$.
	Then, all of the following hold true:
	\begin{enumerate}
		\item $\muSpin_{\domain,x}^{\blackm\,\whitep} \preceq_\black \muSpin_{\domain,x}^{\blackp\,\whitep}$,
		% \comment{Is this one used for Theorem~2/3?}
		\item $\left.\muSpin_{\domain',x}^{\blackp}\right|_\domain\preceq_\black  \muSpin_{\domain,x}^{\blackp}$,
		% \comment{For convergence under $\black+$ b.c.}
		\item \label{cbc:elaborate}$\left.\muSpin_{\domain',x}^{\whitep}(\blank|\{\sigma^\black|_A=\tau\})\right|_\domain\preceq_\black  \muSpin_{\domain,x}^{\blackp}$
		for any $A\subset\partial_\face\domain'$ and $\tau:A\to\{\pm1\}$.
		% \comment{For proving Theorem 1}
	\end{enumerate}
	% \begin{equation}
	% 	\tag{CBC} \label{eq:cbc-lip}
	% 	\muSpin_{\domain,x}^{\blackm\,\whitep} \preceq_\black \muSpin_{\domain,x}^{\blackp\,\whitep} \quad \text{and} \quad
	% 	\left.\muSpin_{\domain',x}^{\blackp}\right|_\domain,\,
	% 	\left.\muSpin_{\domain',x}^{\blackp\,\whitep}\right|_\domain \preceq_\black \muSpin_{\domain,x}^{\blackp}.
	% \end{equation}
\end{corollary}

\begin{proof}
	The first statement is obvious because the two measures may be viewed as the measure
	$\muSpin_{\domain,x}^{\whitep}$ conditioned on a decreasing and an increasing event respectively.
	For the remaining statements, observe that the measure
	$\muSpin_{\domain',x}$ conditioned
	on the increasing event $\partial_\upvert\domain\subseteq \xi^\blackp$
	produces precisely $\muSpin_{\domain,x}^{\blackp}$ within $\Omega$;
	this is the Markov property (Lemma~\ref{lemma:loop_markov}).
	The same holds true with boundary conditions imposed.
\end{proof}

\subsection{Infinite-volume limit}

The comparison between boundary conditions naturally leads
to the existence of an infinite-volume limit,
which we now describe.
In order to do this, we extend the notion of a Gibbs measure and the weak topology
to the spin representation.
Observe that the spin interaction is local, and therefore all probability kernels in the specification
of this model are continuous.

\begin{proposition}[Convergence under $\blackp$ boundary conditions]
	\label{prop:inf-vol-lip}
	Let~$0<x\leq 1$.
	Then, for any increasing sequence of domains~$(\domain_k)_k\nearrow\hexlattice$, we have the following weak convergence:
	\[
		\muSpin_{\domain_k,x}^\blackp\xrightarrow[k\to\infty]{}\muSpin_{x}^\blackp,
	\]
	where $\muSpin_x^\blackp$ is a Gibbs measure for the spin representation independent
	of the choice of the sequence $(\domain_k)_k$ and invariant under the symmetries of~$\upvert(\hexlattice)$.
	 Moreover:
	 \begin{enumerate}
		\item The law of $(\sigma^\black,\xi^\black)$ is extremal and ergodic.
		\item Both $(\sigma^\black,\xi^\blackp,-\xi^\blackm)$ and~$(\sigma^\white,\xi^\whitep,-\xi^\whitem)$
		satisfy the FKG inequality in $\muSpin_x^\blackp$.
		\item\label{prop:inf-vol-lip:three}
		The distribution of~$\sigma^\white$ given~$(\sigma^\black,\xi^\black)$
		is obtained by assigning~$\pm$ to the clusters
		of~$\triedges((\xi^\black)^*)$ via independent fair coin flips.
		\item\label{prop:inf-vol-lip:four}
		The distribution of~$\sigma^\black$ given~$(\sigma^\white,\xi^\white)$
		is obtained by assigning~$+$ to the infinite cluster of~$\triedges((\xi^\white)^*)$ (if it exists) and~$\pm$ all finite clusters via independent fair coin flips.
	 \end{enumerate}
\end{proposition}

\begin{proof}
	Since the underlying sample space consists of spins and the auxiliary
	family of independent random variables $U$, it suffices to prove convergence of the distribution of the spins.
	Convergence in law for the vertex percolations then follows as a corollary.
	The vertex percolations do play an important role in this convergence proof.

	We first construct the limit of the marginal distributions on black spins and vertices.
	By comparison between boundary conditions (Corollary~\ref{cor:cbc-lip}), the
	law of $(\sigma^\black,\xi^\blackp,-\xi^\blackm)$
	under~$\muSpin_{\domain,x}^\blackp$
	is stochastically decreasing
	in $\domain$.
	By standard arguments~\cite[Proposition~4.10b,
	Theorem~4.19, Corollary~4.23]{Gri06},
	the distribution of this triplet tends (in the weak topology) to some measure $\mu$
	as $\domain\nearrow\H$,
	and this limit is translation-invariant, ergodic, extremal, and satisfies the FKG inequality.

	Our objective is however to prove convergence of $(\sigma^\black,\sigma^\white)$.
	In finite volume, the conditional distribution of $\sigma^\white$ given $\sigma^\black$
	is given by flipping fair coins for the connected components of $\triedges((\xi^\black)^*)$
	(Lemma~\ref{lemma:loop_decoupling}).
	This property commutes with the volume limit if there is at most one infinite connected component.
	Uniqueness of the infinite connected component (if it exists) under $\mu$ is guaranteed by the argument of
	Burton and Keane~\cite{BurKea89}.
	This proves that $\muSpin_{x}^\blackp$ is well-defined and that the distribution
	of $(\sigma^\black,\sigma^\white)$ in this measure can be directly expressed in terms of
	$\mu$ and the above sampling algorithm.

	The FKG inequality for $(\sigma^\white,\xi^\whitep,-\xi^\whitem)$
	in $\muSpin_{x}^\blackp	$ follows from~\cite[Proposition~4.10b]{Gri06}.

	Finally, Parts~\ref{prop:inf-vol-lip:three} and~\ref{prop:inf-vol-lip:four} follow their finite-volume counter-part Lemma~\ref{lemma:loop_decoupling} via standard arguments and we omit the details.
\end{proof}

The measures $\muSpin_x^\blackm$, $\muSpin_x^\whitep$, and $\muSpin_x^\whitem$ are defined analogously.

\subsection{Extremality and delocalisation}
\label{sec:extremality-lip}

Assume the context of Proposition~\ref{prop:inf-vol-lip}.
If $\triedges((\xi^\black)^*)$ contains an infinite cluster almost surely,
then there is a nontrivial tail-measurable random variable which indicates the white spin colour
of this cluster.
In other words, the flip symmetry for the white spins is broken.
It is easy to see that in this scenario no face is surrounded by infinitely many loops,
and that the corresponding height function is localised.
This should occur for all~$0<x<1/\sqrt{2}$ according to the predicted phase diagram.
For~$1/\sqrt2\leq x\leq 1$, we use the super-duality property~\eqref{eq:super_duality} to rule out the infinite cluster in~$\triedges((\xi^\black)^*)$.
This readily implies delocalisation.

\begin{proposition}[Extremality]\label{prop:ergo-lip}
	Let~$1/\sqrt{2} \leq x \leq 1$. Then, the measure $\muSpin_x^\blackp$ is
	extremal and ergodic.
\end{proposition}

\begin{proof}
	Proposition~\ref{prop:inf-vol-lip} says that the law of $\xi^\black$ under~$\muSpin_x^\blackp$ is extremal and ergodic.
	Suppose that all clusters of $(\xi^\black)^*$ are almost surely finite.
	Extremality and ergodicity
	of $\muSpin_x^\blackp$ then follows from the fact that
	$\sigma^\white$ is obtained by assigning~$\pm 1$ to the clusters of~$\triedges((\xi^\black)^*)$ (which are also finite) independently.

	Let $1/\sqrt{2} \leq x \leq 1$.
	In order to derive a contradiction, assume that
	\begin{equation}
		\label{eq:whatwewanttocontradict}
		\muSpin_x^\blackp(\text{$(\xi^\black)^*$ has an infinite cluster})>0.
	\end{equation}
	Using the super-duality~\eqref{eq:super_duality}, the symmetry between~$\xi^\whitep$ and~$\xi^\whitem$ under~$\muSpin_x^\blackp$ and the fact that each cluster of~$\xi^\white$ is a cluster of~$\xi^\whitep$ or a cluster of~$\xi^\whitem$,
	 \begin{equation}
		\label{eq:whatwewanttocontradict-2}
		\muSpin_x^\blackp(\text{$\xi^\whitep$ has an infinite cluster})>0.
	\end{equation}
	By the comparison between boundary conditions, for any domain~$\domain$, the distribution of~$\xi^\whitep$ under~$\muSpin_{\domain,x}^\whitep$ dominates that under~$\muSpin_x^\blackp$.
	Taking the limit~$\domain\nearrow \hexlattice$ as in Proposition~\ref{prop:inf-vol-lip} (but for~$\xi^\whitep$) and using black/white symmetry and~\eqref{eq:whatwewanttocontradict-2}, we get
	\[
		\muSpin_x^\blackp(\text{$\xi^\blackp$ has an infinite cluster})
		= \muSpin_x^\whitep(\text{$\xi^\whitep$ has an infinite cluster}) >0.
	\]
	Together with~\eqref{eq:whatwewanttocontradict}, this contradicts Zhang's non-coexistence theorem.
\end{proof}

\begin{corollary}[Delocalisation]\label{cor:inf-vol-limits-equal-lip}
	For any~$1/\sqrt{2}\leq x \leq 1$,
	we have
	\(
		\muSpin_x^\blackp = \muSpin_x^\blackm = \muSpin_x^\whitep = \muSpin_x^\whitem
	\)
	and, for every face, there exists almost surely exhibit infinitely many white and infinitely many black loops surrounding it.
\end{corollary}

\begin{proof}
	It is clearly enough to show that~$\muSpin_x^\blackp = \muSpin_x^\whitep$.
	Fix any domain~$\Lambda$ and consider any sequence of domains~$\domain_k\nearrow\hexlattice$ that contain~$\Lambda$.
	Let~$C_\Lambda$ be the event that $\omega[\sigma^\black]\cup \omega[\sigma^\white]$ contains a loop surrounding~$\Lambda$.
	Note that, for any~$k\geq 1$,
	\[
		\muSpin_{\domain_k,x}^\blackp(\cdot \, \vert \, C_\Lambda)|_\Lambda
		= \muSpin_{\domain_k,x}^\whitep(\cdot \, \vert \, C_\Lambda)|_\Lambda.
	\]
	Indeed, the operation of swapping the colours of the outermost loop surrounding~$\Lambda$ and all loops inside of it is measure-preserving and swaps~$\sigma^\black$ and~$\sigma^\white$ on~$\Lambda$.
	By Proposition~\ref{prop:ergo-lip}, the conditioning becomes trivial as~$k\to\infty$. Since~$\Lambda$ was arbitrary, the statement follows.
\end{proof}

\subsection{Proof of Theorem~\ref{thm:loop_soft_deloc}}
\label{subsec:proof_of_thm:loop_soft_deloc}

\begin{lemma}
	\label{lemma:thm1aux}
	Let $1/\sqrt2\leq x\leq 1$.
	Consider an increasing sequence of domains $(\Omega_k)_k\nearrow\hexlattice$
	and spin configurations~$\tau_k:\partial_\face\Omega_k\to\pm 1$.
	Then,
	\[
		\mu_k:=\muSpin_{\Omega_k,x}^\whitep(\blank|\{\sigma^\black|_{\partial_\face\Omega_k}=\tau_k\})
		\xrightarrow[k\to\infty]{}
		\muSpin_x^\blackp.
	\]
	In particular, $\muSpin_{\Omega_k,x}^{\blackp\,\whitep}\to\muSpin_x^\blackp$ as~$k\to\infty$,
	Moreover, if $\rho$ is a full-plane Gibbs measure on spins
	under which $\{\sigma^\white=-\}$ does not percolate almost surely,
	then $\rho=\muSpin_x^\blackp$.
\end{lemma}

\begin{proof}
	Let $\mu$ denote any subsequential limit of the law of $(\sigma^\black,\xi^\black)$
	under $\mu_k$ as $k\to\infty$.
	If $\Omega$ is a domain, then Corollary~\ref{cor:cbc-lip}, Part~\ref{cbc:elaborate}, implies that
	\[
		\muSpin_{\Omega,x}^\blackm\preceq_\black\mu|_\Omega\preceq_\black \muSpin_{\Omega,x}^\blackp.
	\]
	Since this is true for any domain $\Omega$, we get
	\(
		\muSpin_{x}^\blackm\preceq_\black\mu\preceq_\black \muSpin_{x}^\blackp
	\).
	Corollary~\ref{cor:inf-vol-limits-equal-lip},
	therefore, implies that~$\mu$ and~$\muSpin_{x}^\blackp$ have the same marginal on~$(\sigma^\black,\xi^\black)$.
	Recall now the statement of Proposition~\ref{prop:inf-vol-lip}, Part~\ref{prop:inf-vol-lip:three}.
	Since all clusters of~$(\xi^\black)^*$ are $\muSpin_{x}^\blackp$-almost surely (and thus $\mu$-almost surely) finite by Proposition~\ref{prop:ergo-lip},
	the sampling statement in Proposition~\ref{prop:inf-vol-lip}, Part~\ref{prop:inf-vol-lip:three} holds true for any subsequential limit of the sequence $(\mu_k)_k$,
	and therefore we obtain the desired convergence $\mu_k\to\muSpin_{x}^\blackp$ as $k\to\infty$.

	The first of the two corollaries is obvious.
	Now consider the measure $\rho$.
	Run the following exploration process:
	first reveal the spins at all faces whose centre is at a distance at least $R$ from
	the origin, then reveal the spins at neighbours of revealed faces $u$
	where $\sigma^\white(u)=-$ until there are no faces left to reveal.
	This induces boundary conditions as in the statement of the lemma,
	and so the conditional law of the unexplored faces
	converges to $\muSpin_x^\blackp$ as $R\to\infty$.
	This means that $\rho=\muSpin_x^\blackp$.
\end{proof}

Define $\muLoop_{2,x}$ as the pushforward of $\muSpin_{x}^\blackp$ along the map
$(\sigma^\black,\sigma^\white)\mapsto \omega[\sigma^\black]\cup\omega[\sigma^\white]$.

\begin{proof}[Proof of Theorem~\ref{thm:loop_soft_deloc}]
	Consider any increasing sequence of domains~$(\domain_k)_k\nearrow\hexlattice$.
	Since $\muLoop_{\domain_k,2,x}$ is the pushforward of $\muSpin_{\domain_k,x}^{\blackp\,\whitep}$
	(Proposition~\ref{prop:lip-spin-loop}) and since $\muLoop_{2,x}$ is defined as the pushforward
	of $\muSpin_{x}^\blackp$, Lemma~\ref{lemma:thm1aux}
	implies the desired convergence~$\muLoop_{\domain_k,2,x}\to\muLoop_{2,x}$ as~$k\to\infty$.
	It is then straightforward that~$\muLoop_{2,x}$ is invariant to the symmetries of~$\hexlattice$, since applying these symmetries to~$\domain_k$ results in the same limiting measure.
	By Proposition~\ref{prop:ergo-lip}, the measure~$\muLoop_{2,x}$ is ergodic and tail trivial.
	Corallary~\ref{cor:inf-vol-limits-equal-lip} implies~\eqref{eq:inf-many-loops-lip} and hence that~$\muLoop_{2,x}$ is a Gibbs measure; it is then standard that tail triviality implies extremality.

	It remains to show that the loop~\O{2} model has no other Gibbs measures.
	At~$x=1$, this was proved in~\cite[Theorem~1.3]{GlaMan21};
	the proof presented here is slightly different.
	Let $\rho$ be a Gibbs measure for the loop~\O{2} model.
	It suffices to consider the case when $\rho$ exhibits bi-infinite paths
	almost surely.
	The idea is to turn $\rho$ into a Gibbs measure $\tilde\rho$ for the spin model
	so that we may apply Lemma~\ref{lemma:thm1aux}.

	\begin{proof}[Definition of $\tilde\rho$]\renewcommand\qedsymbol{}
	A sample $(\sigma^\black,\sigma^\white)$ from~$\tilde{\rho}$ is obtained as follows:
	\begin{enumerate}
		\item Sample a loop configuration~$\omega$ from~$\rho$.
		\item Define a loop configuration~$\omega^\black\subseteq \omega$ as follows: all
		bi-infinite paths of~$\omega$ belong to~$\omega^\black$; each finite loop
		of~$\omega$ belongs to~$\omega^\black$ with probability~$1/2$
		independently of the others.
		Let~$\omega^\white:=\omega\setminus \omega^\black$ denote the complementary
		white loop configuration.
		\item Define the spin configuration~$\sigma^\black$
		on~$\faces\hexlattice$ as follows: the spin at the origin is~$+$ or~$-$ with
		probability~$1/2$; all other spins are then determined by the
		requirement~$\omega[\sigma^\black] = \omega^\black$.
		Let~$\sigma^\white$ be the unique spin configuration
		on~$\faces\hexlattice$ such that the spins adjacent to infinite paths
		of~$\omega^\black$ are plus and~$\omega[\sigma^\white]=\omega^\white$.
	\end{enumerate}
	\end{proof}

	Observe first that
	the set $\{\sigma^\white=-\}$ does not percolate in $\tilde\rho$,
	since no bi-infinite paths are coloured white in the above construction.
	This observation and Lemma~\ref{lemma:thm1aux} jointly imply that $\tilde\rho=\muSpin_x^\blackp$
	and consequently $\rho=\muLoop_{2,x}$ as soon as we establish
	that $\tilde\rho$ is indeed a Gibbs measure for the spin measure.
	This is the purpose of the rest of the proof.

	By irreducibility of the model, it suffices to demonstrate that $\tilde\rho$
	is invariant under the Glauber dynamics.
	The latter selects a face and a colour,
	and resamples the spin of that colour at that face according to the Gibbs distribution.

	Let $u\in\faces{\hexlattice}$ denote the face chosen by the Glauber
	dynamics, and suppose that it chooses to update the black spin (the proof
	for the white spin is similar).  For any spin configuration~$\sigma^\black
	\in \{\pm 1\}^{\faces\hexlattice}$, let~$\sigma^\black_\pm$ denote the spin
	configuration which equals $\sigma^\black$ except that it assigns the
	value~$\pm$ to~$u$.  Write $\domain$ for the domain having $u$ as its only
	face.

	\begin{proof}[Definition of $\kappa$]\renewcommand\qedsymbol{}
	For any consistent pair~$(\sigma^\black,\sigma^\white)$,
	define the probability measure~$\kappa^{\sigma^\black\sigma^\white}$ on spins as follows.
	\begin{itemize}
		\item If $\omega[\sigma^\white]$ contains an edge of $\Omega$, then $\kappa^{\sigma^\black\sigma^\white}=\delta_{(\sigma^\black,\sigma^\white)}$.
		\item If $\omega[\sigma^\white]$ contains no edge of $\Omega$, then define $\kappa^{\sigma^\black\sigma^\white}$ as the unique probability measure with support $\{(\sigma^\black_+,\sigma^\white),(\sigma^\black_-,\sigma^\white)\}$ such that
		\begin{equation}\label{eq:glauber}
			\frac{\kappa^{\sigma^\black\sigma^\white}(\sigma^\black_+,\sigma^\white)}
			{\kappa^{\sigma^\black\sigma^\white}(\sigma^\black_-,\sigma^\white)}
			=
			\frac{x^{|\omega[\sigma^\black_+]\cap \domain|}}
			{x^{|\omega[\sigma^\black_-]\cap \domain|}}.
		\end{equation}
	\end{itemize}
	Observe that $\kappa:(\sigma^\black,\sigma^\white)\mapsto\kappa^{\sigma^\black\sigma^\white}$ is the probability kernel of the above Glauber dynamics.
	\end{proof}

	To conclude this proof, it suffices to demonstrate that $\tilde\rho$ is invariant under $\kappa$, that is, $\tilde\rho=\tilde\rho\kappa$.
	For this, we define another
	probability
	kernel~$\pi$ on spins.

	\begin{proof}[Definition of $\pi$]\renewcommand\qedsymbol{}
	For any consistent pair $(\sigma^\black,\sigma^\white)$ with no infinite cluster in~$\{\sigma^\white=-\}$, the probability
	measure~$\pi^{\sigma^\black\sigma^\white}$ samples a random pair~$(\tilde\sigma^\black,\tilde\sigma^\white)$  as
	follows:
	\begin{enumerate}
		\item\label{PISTEPONE} First, sample a loop configuration $\omega$ from
		$\muLoop_{\Omega,n,x}^{\omega[\sigma^\black]\cup\omega[\sigma^\white]}$.
		\item \label{PISTEPTWO}Define~$\omega^\black$ as the union of:
		\begin{itemize}
			\item All bi-infinite
			paths in~$\omega$,
			\item All loops in~$\omega[\sigma^\black]$ that do not
			intersect~$\domain$,
			\item Each loop of~$\omega[\sigma^\black]$ intersecting $\domain$
			is added with probability $1/2$ independently of all other loops.
		\end{itemize}
		\item Let~$\omega^\white:=\omega\setminus \omega^\black$
		denote the set of complementary loops.
		\item Define~$(\tilde\sigma^\black,\tilde\sigma^\white)$ as
		the unique pair of spin configurations that coincides with
		$(\sigma^\black,\sigma^\white)$ at all but a finite number of faces and
		such that~$\omega[\sigma^\black]=\omega^\black$
		and~$\omega[\sigma^\white]=\omega^\white$.
	\end{enumerate}
\end{proof}

Now make two claims.
\begin{enumerate}
	\item\label{CLAIMONE} The measure $\tilde\rho$ is invariant under $\pi$, that is, $\tilde\rho=\tilde\rho\pi$.
	\item\label{CLAIMTWO} For any consistent pair $(\sigma^\black,\sigma^\white)$ with infinite cluster in~$\{\sigma^\white=-\}$,
	the probability measure $\pi^{\sigma^\black\sigma^\white}$ is invariant under~$\kappa$,
	that is, $\pi^{\sigma^\black\sigma^\white}=\pi^{\sigma^\black\sigma^\white}\kappa$.
\end{enumerate}
Then, the desired invariance of~$\tilde\rho$ under~$\kappa$ (and hence the Gibbs property) follows from associativity of probability kernels:
\[
	\tilde\rho=\tilde\rho\pi=\tilde\rho(\pi\kappa)=(\tilde\rho\pi)\kappa=\tilde\rho\kappa.
\]

Claim~\ref{CLAIMONE} follows from the Gibbs property of $\rho$
and the definitions of $\tilde\rho$ and $\pi$.
More precisely, the distribution of the loops is invariant
under the update in Step~\ref{PISTEPONE} in the definition of $\pi$,
while the remaining steps are consistent with the construction of
$\tilde\rho$ from $\rho$.

	It remains to show Claim~\ref{CLAIMTWO}.
	Let~$(\sigma^\black,\sigma^\white)$ and $(\tilde\sigma^\black,\tilde\sigma^\white)$ denote consistent pairs such that $\{\sigma^\white=-\}$ and $\{\tilde\sigma^\white=-\}$ do not percolate.
	We shall suppose that $\tilde\sigma^\black=\tilde\sigma^\black_+$ without loss of generality.
	It suffices to show that if $(\tilde\sigma^\black,\tilde\sigma^\white)$ is in the support of $\pi^{\sigma^\black\sigma^\white}$,
	then
	\begin{equation}
		\label{eq:glauber-invariance}
		\frac{
			\pi^{\sigma^\black\sigma^\white}(\tilde\sigma^\black_-,\tilde\sigma^\white)
		}{
			\pi^{\sigma^\black\sigma^\white}(\tilde\sigma^\black_+,\tilde\sigma^\white)
		}
		=
		\frac{
			\kappa^{\sigma^\black\sigma^\white}(\tilde\sigma^\black_-,\tilde\sigma^\white)
		}{
			\kappa^{\sigma^\black\sigma^\white}(\tilde\sigma^\black_+,\tilde\sigma^\white)
		}
	\end{equation}
	If $(\tilde\sigma^\black_-,\tilde\sigma^\white)$ is not consistent then this is trivial,
	since it implies that both sides equal zero.
	If the same pair is consistent, then
	\begin{equation}
		\label{eq:explicitdefinitionofpi}
		\pi^{\sigma^\black\sigma^\white}(\tilde\sigma^\black_{\pm},\tilde\sigma^\white)\propto
		x^{|\omega[\tilde\sigma^\black_{\pm}]\cap\domain|+|\omega[\tilde\sigma^\white]\cap\domain|}.
	\end{equation}
	Indeed, the loop weight $2^{\ell(\omega;\domain)}$ in the definition of the measure~$\muLoop_{\domain,n,x}^{\omega[\sigma^\black]\cup\omega[\sigma^\white]}$ appearing in Step~\ref{PISTEPONE}
	in the definition of $\pi$ may be written as a product of two factors:
	a factor corresponding to finite loops,
	and a factor corresponding to bi-infinite paths.
	The first factor cancels with the coin flips in Step~\ref{PISTEPTWO} in the definition of $\pi$;
	the second factor does not play a role since the number of bi-infinite paths cannot change by local manipulations.
	Finally, \eqref{eq:glauber} and~\eqref{eq:explicitdefinitionofpi} imply~\eqref{eq:glauber-invariance} and this finishes the proof.
\end{proof}

\subsection{Proof of Theorem~\ref{thm:soft_Lip_deloc}}
\label{subsec:proof_of_thm:soft_Lip_deloc}

\begin{proof}[Proof of Theorem~\ref{thm:soft_Lip_deloc}]
	By the coupling of the loop~\O{2} model and the random Lipschitz (Proposition~\ref{prop:lip-spin-loop}),
	\begin{equation}\label{eq:var-via-loops-lip}
		\Var_{\muLip^0_{\domain,x}}[h(u)] = \E_{\muLoop_{\domain,2,x}}[\#\text{loops around $u$}].
	\end{equation}
	Theorem~\ref{thm:loop_soft_deloc} asserts that the right-hand side diverges
	as~$\domain\nearrow\hexlattice$.  It remains to show the logarithmic bounds
	on fluctuations.  The case~$x=1$ was treated in~\cite{GlaMan21} by adapting
	the dichotomy theorem of~\cite{DumSidTas16} to the setting with a weaker
	domain Markov property.
	The proof in~\cite{GlaMan21} relies heavily on the analysis of so-called
	\emph{double crossings}. For example, a \emph{double crossing of pluses} of $\sigma^\black$
	consists of a path through the hexagonal lattice such that $\sigma^\black$
	is valued plus on all faces adjacent to it.
	Morally, the adaptation of the proofs in~\cite{GlaMan21} to
	our setting (namely $x\in[1/\sqrt2,1]$) comes down to one simple change:
	all double crossings are replaced by the corresponding crossings for our newly introduced
	percolations $\xi^\black$ and $\xi^\white$.
	In particular, a double crossing by pluses by $\sigma^\black$ is replaced by
	a crossing by $\xi^\blackp$.
	Note that~\cite{GlaMan21} uses red and blue colours;
	we associate the colour red with black and the colour blue with white.

	The proof in~\cite{GlaMan21} is rather long and we shall explain how to adapt
	it rather than repeat it in its entirety.
	The next paragraph outlines the key adaptations.
	Then, we derive the fundamental crossing estimate,
	and give a sketch of the proof.
	Details about the adaptations of the proofs of particular lemmata of~\cite{GlaMan21} are provided in Appendix~\ref{app:dicho}.

\subsubsection*{Overview of the key similarities and adaptations}
	\begin{itemize}
		\item The two-spin representation with symmetry between~$\sigma^\black$
		and~$\sigma^\white$ remains the same.
		\item The domain Markov property of double plus circuits
		in~\cite{GlaMan21} is replaced by the same domain Markov property with
		circuits in~$\xi^\blackp$ (Lemma~\ref{lemma:loop_markov}).
		\item The FKG inequality for~$\sigma^\black$ jointly with black double plus crossings
		extends to all~$0<x\leq 1$~(Proposition~\ref{prop:fkg-joint-lip}),
		in the sense that we prove the FKG inequality for the joint distribution of~$(\sigma^\black,\xi^\blackp,-\xi^\blackm)$.
		\item The duality between crossings of~$\{\sigma^\black = +\}$ and~$\{\sigma^\black = - \}$ remains the same.
		\item The super-duality between double-crossings of~$\sigma^\black$
		and~$\sigma^\white$ is replaced by the super-duality of~$\xi^\black$
		and~$\xi^\white$ in~\eqref{eq:super_duality}.
		\item Crossing estimates for double-crossings of symmetric domains (such as squares or rhombi) under various boundary conditions
		are replaced by
		Lemma~\ref{lemma:loop_rhombus_crossing} (stated next), which asserts that the same crossing estimates hold
		true for $\xi^\black$-crossings in the more general setup.
	\end{itemize}

\subsubsection*{The crossing estimate}

Let us now introduce a \emph{rhombus of size~$m\in\Z_{\geq 0}$};
its precise dimensions do not matter, but for the sake of concreteness we
set
\[
	R_m:=
\{k+\ell e^{i\pi/3} : k,\ell \in\Z\cap[-m,m]\}
-i/\sqrt3
\subset \upvert(\hexlattice)
\]
and identify~$R_m$ with the induced subgraph of~$\upvert(\hexlattice)$.
We define the sides of~$R_m$ in a natural way:
left, right, top, and bottom sides are given by are the parts of the boundary obtained by setting~$k=-m$, $k=m$, $\ell=m$, and $\ell=-m$ respectively.
For~$\xi\subset\upvert(\hexlattice)$, we define $\calH_m(\xi)$ (resp. $\calV_m(\xi)$)
as the event that~$\xi\cap R_m$ contains a path connecting the left and right (resp. the top and bottom) sides of~$R_m$.
It is a standard consequence of planar duality on the triangular lattice
that, for any $\xi\subset \upvert(\hexlattice)$,
\[
	\text{exactly one of $\calH_m(\xi)$ and~$\calV_m(\xi^*)$ occurs.}
\]

\begin{lemma}[Crossings in symmetric domains]
	\label{lemma:loop_rhombus_crossing}
	Assume~$x\in [1/\sqrt2,1]$.

\begin{enumerate}
	\item For any domain $\domain$ and~$m\geq 1$ such that $R_m\subset \upvert(\domain)$, one has
	\begin{equation}
		\muSpin_{\domain,x}^\blackp(\calH_m(\xi^\blackp))\geq \tfrac14.\label{eq:hor-ver-duality-lip}
	\end{equation}
	\item  Assume that a domain~$\domain$ and vertices~$A,B,C,D$ on its boundary are such the axial symmetry over~$AC$ leaves~$\domain$ invariant and maps~$B$ to~$D$.
	Then,
	\begin{equation}
		\muSpin_{\domain,x}^{\blackp\, \whitep}(AB \xleftrightarrow{\xi^\blackp} CD) \geq \tfrac12.\label{eq:symmetric-domains-lip}
	\end{equation}
\end{enumerate}
\end{lemma}

\begin{proof}
	\begin{enumerate}[wide, labelwidth=!, labelindent=0pt]
		\item  Consider any domain~$\domain'$ that is symmetric with respect to the line passing through a diagonal of~$R_m$ and such that~$\domain\subset \domain'\setminus \partial\domain'$ (eg. pick large enough~$N\in\N$ and define~$\domain'$ as the smallest domain such that
	$R_N\subset\upvert(\domain')$).
	The super-duality~\eqref{eq:super_duality} implies~$(\xi^\black)^* \subseteq \xi^\white$.
	Thus, by the standard percolation duality~\eqref{eq:hor-ver-duality-lip}, we have
	\[
		\muSpin_{\domain',x}^{\blackp\,\whitep}(
			\calH_m(\xi^\black)\cup\calV_m(\xi^\white)
		) = 1.
	\]
	By symmetry, the two events have the same probability
	of at least $1/2$.
	No vertex in $\xi^\blackp$ can have a neighbour in $\xi^\blackm$,
	since that would imply that black spin of some face is simultaneously
	valued $+$ and $-$.
	Thus, we have $\calH_m(\xi^\black)=\calH_m(\xi^\blackp)\cup \calH_m(\xi^\blackm)$.
	Since Corollary~\ref{cor:cbc-lip} implies
	\[
		\muSpin_{\domain',x}^{\blackp\,\whitep}(
		\calH_m(\xi^\blackp)
		)
		\geq
		\muSpin_{\domain',x}^{\blackp\,\whitep}(
		\calH_m(\xi^\blackm)
		),
	\]
	we conclude that
	\[
		\muSpin_{\domain',x}^{\blackp\,\whitep}(
			\calH_m(\xi^\blackp)
		) \geq 1/4.
	\]
	Another application of Corollary~\ref{cor:cbc-lip} yields the statement:
	\[
		\muSpin_{\domain,x}^\blackp(\calH_m(\xi^\blackp)) \geq \muSpin_{\domain',x}^{\blackp\,\whitep}(\calH_m(\xi^\blackp)) \geq \tfrac14.
	\]

	\item Similarly to the above, the super-duality~\eqref{eq:super_duality} implies that
	\[
		\muSpin_{\domain,x}^{\blackp\,\whitep}(AB \xleftrightarrow{\xi^\blackp} CD)
		+ \muSpin_{\domain,x}^{\blackp\,\whitep}(AC \xleftrightarrow{\xi^\whitep} CB)
		 = 1.
	\]
	The two events clearly have the same probability and the statement follows.
	\end{enumerate}
\end{proof}

	\paragraph{{\bf Renormalisation inequality and dichotomy.}}
	For~$n\in\N$, let~$\Lambda_n$ denote the domain consisting of faces at distance at most~$n$ from~$(0,0)$ (in the graph distance of the dual lattice~$\hexlattice^*$).
	Define~$\Circ^\blackp(n,2n)$ to be the event that there exists a~$\xi^\blackp$ circuit that is contained in~$\Lambda_{2n}$ and surrounds~$\Lambda_n$.
	For any~$\rho>2$, define~$\alpha_n:=\muSpin_{\Lambda_{\rho n},x}^\blackm(\Circ^\blackp(n,2n))$.
	The main statement that we need to prove is~\cite[Theorem~4.1]{GlaMan21}: there exists~$\rho>2$ and~$C>1$ such that, for all~$n\geq 1$,
	\begin{equation}\label{eq:renorm-ineq-lip}
		\alpha_{(\rho+2)n}\leq  C\cdot \alpha_n^2.
	\end{equation}
	Indeed, by induction, this renormalisation inequality implies the following dichotomy (\cite[Corollary~4.2]{GlaMan21}):
	either~$\inf_n \alpha_n > 0$ or~$\alpha_n$ exhibits a stretched-exponential decay (ExpDec).
	By the latter we mean existence of~$c,C>0$ and~$n_0\in \N$ such that~$\alpha_n\leq Ce^{-n^c}$, for all~$n=(\rho+2)^kn_0$, where~$k\in\N$.
	The former can be excluded using duality for double crossings and the symmetry between~$\sigma^\black$ and~$\sigma^\white$.
	The box-crossing property implies the logarithmic bound since the variance of the height function is greater or equal than the number of alternating~$\blackp\,\blackp$ and~$\blackm\,\blackm$ circuits.
	\\

	\paragraph{{\bf Sketch of the proof of~\eqref{eq:renorm-ineq-lip}.}}
	The proof goes by developing the {\em RSW theory} and proving the so-called {\em pushing lemma} (introduced in~\cite{DumSidTas16}).
	The RSW estimates are derived by showing that two long vertical $\blackp\,\blackp$ crossings starting at a (small) linear distance from each other can be connected to each other by a $\blackp\,\blackp$ path with a uniformly positive probability.
	The argument goes by exploring the leftmost and the rightmost such vertical crossings, ruling out several pathological behaviours, pushing away the boundary conditions and reducing a question to crossing of a certain planar (but going outside of~$\hexlattice$) symmetric domain that has~$\blackp$ on its inner boundary.
	By the duality (Lemma~\ref{lemma:loop_rhombus_crossing}), such domain has a~$\blackp\,\blackp$ crossing with probability at least~$1/2$.
	The pushing lemma is proven on cylindrical domains to make use of invariance to horizontal shifts:
	by the RSW estimates, it is enough to find a vertical $\blackp\,\blackp$ crossing under~$\blackp\,\blackp/\blackm\,\blackm$ boundary conditions at the top/bottom of the cylinder;
	by duality, absence of such crossing implies existence of a horizontal~$\whitep\,\whitep$ or~$\whitem\,\whitem$ crossing;
	pushing away the boundary conditions, one obtains a symmetric cylinder with~$\blackp\,\blackp/\whitep\,\whitep$ or~$\blackp\,\blackp/\whitem\,\whitem$ boundary conditions;
	again, by the duality and black/white symmetry, a box at the symmetry axis is crossed by~$\blackp\,\blackp$ with probability at least~$1/4$.\\

	We refer to Appendix~\ref{app:dicho} for more details.
\end{proof}

\section{Delocalisation of graph homomorphisms in the six-vertex model}
\label{sec:six-vertex}

This section follows as much as possible the structure of Section~\ref{section:loop_on}.

\subsection{Notation}

\subsubsection{Planar duality}
In Section~\ref{section:loop_on}, we proved Theorems~\ref{thm:loop_soft_deloc} and~\ref{thm:soft_Lip_deloc}
using planar duality arguments for site percolations on the triangular lattice $\upvert(\hexlattice)$.
The main change here is that the planar duality arguments apply to bond percolations on a square lattice
and its dual. The primal square lattice is the diagonal-adjacency graph on black squares of $\squarelattice$;
the same construction on white squares gives the dual graph.

Recall the notations introduced in Sections~\ref{sec:results-lip} and~\ref{subsec:introRCM}.
Write
\[
	d^\black:\intvert(\domain)\to E(\domain^\black);
	\qquad
	d^\white:\intvert(\domain)\to E(\domain^\white)
\]
for the bijections which assign a black and a white diagonal edges respectively to each interior
vertex.
Any percolation $\xi\in\{0,1\}^{E(\domain^\black)}$
is identified with the subset $\{\xi=1\}$
and the corresponding spanning subgraph of $\domain^\black$.
The \emph{dual} of any such percolation is the configuration
$\xi^*\in\{0,1\}^{E(\domain^\white)}$ defined by
\[
	\xi^*(d^\white_x):=1-\xi(d^\black_x).
\]
Note that $\xi$ and $\xi^*$ are really dual to each other in the usual sense.
Identical definitions with $\black$ and $\white$ interchanged apply.

Recall~$E_a^\black,E_b^\black \subset E(\squarelattice^\black)$ defined in Section~\ref{subsec:introRCM}.
Define~$E_a^\white, E_b^\white \subset E(\squarelattice^\white)$ similarly and take~$E_a:=E_a^\black \cup E_a^\white$ and~$E_b:=E_b^\black \cup E_b^\white$.

\subsection{The spin representation}
\label{sec:spin-rep-sixv}

\begin{definition}[Spin configurations]
	Let $\domain$ be a domain.
	A \emph{black spin configuration} is a function $\sigma^\black\in\{\pm 1\}^{V(\domain^\black)}$.
	For any $d_x^\black\in E(\domain^\black)$, we say that $\sigma^\black$
	\emph{agrees} at $d_x^\black$, and write $\sigma^\black\perp d_x^\black$,
	if $\sigma^\black$ assigns the same value to the two endpoints of~$d_x^\black$.
	For $A\subset E(\domain^\black)$, we write $\sigma^\black\perp A$
	if $\sigma^\black$ agrees at all edges in $A$.
	The \emph{domain wall} $\omega[\sigma^\black]\in\{0,1\}^{E(\domain^\white)}$
	of $\sigma^\black$ is defined by
	\[
		\omega[\sigma^\black]:=\{d^\white_x:\sigma^\black\not\perp d^\black_x\}.
	\]
	Similar definitions apply with the two colours interchanged.
	We shall always consider pairs $(\sigma^\black,\sigma^\white)$ of black and white spin configurations.
	Such a pair is called \emph{consistent}, and we write $\sigma^\black\perp\sigma^\white$,
	if it satisfies the following {\em ice rule}:
	\begin{equation}\label{eq:6v-ice}
		\sigma^\black\perp d_x^\black \qquad\text{or}\qquad
		\sigma^\white\perp d_x^\white
		\qquad
		\forall x\in \intvert(\domain).
	\end{equation}
	Write~$\spaceSpin(\domain)$ for the set of all consistent pairs of spin configurations on~$\domain$.
\end{definition}

\begin{definition}[Spin measure]\label{def:spin-measure-six}
	The spin measure on~$\domain$ with parameters~$a,b,c>0$ under free boundary conditions is the probability measure on~$\spaceSpin(\domain)$ defined by
	\begin{equation}
		\label{eq:6v-meas-via-edges}
		\muSpin_{\domain,a,b,c}(\sigma^\black,\sigma^\white)
		:=
		\frac1{Z_{\domain,a,b,c}}\cdot
		\left(\tfrac{a}{c}\right)^{|(\omega[\sigma^\black]\cup\omega[\sigma^\white])\cap \edgesA|}
		\left(\tfrac{b}{c}\right)^{|(\omega[\sigma^\black]\cup\omega[\sigma^\white])\cap \edgesB|},
%		\muSpin_{\domain,a,b,c}(\sigma^\black,\sigma^\white)=\frac1{Z_{\domain,a,b,c}}\cdot
%		x^{|(\omega[\sigma^\black]\cup\omega[\sigma^\white])\cap Y(\domain)|},
	\end{equation}
	where~$Z_{\domain,a,b,c}$ is the partition function.
	Let us also introduce fixed boundary conditions for~$\sigma^\black$, $\sigma^\white$, or both,
	by defining:
	\begin{align*}
		\muSpin_{\domain,a,b,c}^\blackp
			:=&
			\muSpin_{\domain,a,b,c}\left(\blank\mid\{\sigma^\black|_{\partial_{\face^\black} \domain}\equiv +\}\right);
			\\
		\muSpin_{\domain,a,b,c}^\whitep
			:=&
			\muSpin_{\domain,a,b,c}\left(\blank\mid\{\sigma^\white|_{\partial_{\face^\white} \domain}\equiv +\}\right);
			\\
		\muSpin_{\domain,a,b,c}^{\blackp\,\whitep}
			:=&
			\muSpin_{\domain,a,b,c}\left(\blank\mid\{\sigma^\black|_{\partial_{\face^\black} \domain}\equiv+,\,\sigma^\white|_{\partial_{\face^\white} \domain} \equiv +\}\right);
	\end{align*}
	similar definitions apply when $+$ is replaced by $-$.
\end{definition}

\begin{definition}
	\label{def:heights-to-spins-6v}
	Let $\domain$ be a domain and $h\in\spaceHom^{0,1}(\domain)$ a graph homomorphism.
	Its \emph{spin representation} $(\sigma^\black[h],\sigma^\white[h])\in \spaceSpin(\domain)$ is defined such that (see Fig.~\ref{fig:sixvertex_percolations})
	\begin{equation}\label{eq:heights-to-spins-6v}
		\{\sigma^\black[h] = + \} = \{h \in  4\Z\};
		\qquad
		\{\sigma^\white[h] = + \} = \{h \in  4\Z+1\}.
	\end{equation}
\end{definition}

\begin{proposition}
	\label{prop:6v-spin-bij}
	Let $\domain$ be a domain and~$a,b,c>0$.
	Then, $\muSpin_{\domain,a,b,c}^{\blackp\,\whitep}$ is the pushforward of $\muHom_{\domain,a,b,c}^{0,1}$ along the map $h\mapsto (\sigma^\black[h],\sigma^\white[h])$.
\end{proposition}

\begin{proof}
	The map bijectively maps $\spaceHom^{0,1}(\domain)$
	to the support of $\muSpin_{\domain,a,b,c}^{\blackp\,\whitep}$
	and preserves the weight of each configuration.
\end{proof}

\subsection{Graphical representation and super-duality}

In the six-vertex model, an analogue of the black and the white percolations for the loop~\O{2} model was introduced by Lis~\cite{Lis22}; see Fig.~\ref{fig:sixvertex_percolations}.
We will use these percolations to derive the domain
Markov property and extend the delocalisation proof from the uniform case~$a=b=c$ to
general parameters.  Let~$(U_x)_{x\in \intvert(\domain)}$ be i.i.d.\ random
variables with distribution $U([0,1])$.  We incorporate this additional randomness
in the measures defined above without a change of notation.

\begin{definition}[Black and white percolations] \label{def:b-w-perco-6v}
	Given a triplet $(\sigma^\black,\sigma^\white,U)$,
	the \emph{black percolation}~$\xi^\black \in \{0,1\}^{E(\domain^\black)}$ and the \emph{white percolation}~$\xi^\white \in \{0,1\}^{E(\domain^\white)}$ are defined as follows: for every~$x\in \intvert(\domain)$,
	\[
		\begin{cases}
			\text{$\xi^\black(d_x^\black) = 0$ and $\xi^\white(d_x^\white) = 1$}
			&\text{if $\sigma^\black \not\perp d_x^\black$,}\\
			\text{$\xi^\black(d_x^\black) = 1$ and $\xi^\white(d_x^\white) = 0$}
			&\text{if $\sigma^\white \not\perp d_x^\white$,}\\
			\text{$\xi^\black(d_x^\black) = \ind{U_y\leq a/c}$ and~$\xi^\white(d_x^\white) = \ind{U_y> 1- b/c}$}
			&\text{if $\sigma^\black  \perp d^\black_x$, $\sigma^\white  \perp d^\white_x$ and~$d_x^\black \in \edgesA$,}\\
			\text{$\xi^\black(d_x^\black) = \ind{U_y\leq b/c}$ and~$\xi^\white(d_x^\white) = \ind{U_y> 1- a/c}$}
			&\text{if $\sigma^\black  \perp d^\black_x$, $\sigma^\white  \perp d^\white_x$ and~$d_x^\black \in \edgesB$.}\\
		\end{cases}
	\]
	We have $\sigma^\black \perp \xi^\black$, and therefore,~$\xi^\black$ is the disjoint union of
\begin{align*}
	\xi^\blackp:= \{uv \in \xi^\black\colon \sigma^\black(u) = \sigma^\black(v) = +\}
	\quad \text{and} \quad
	\xi^\blackm:= \{uv \in \xi^\black\colon \sigma^\black(u) = \sigma^\black(v) = -\}.
\end{align*}
The same definitions
apply to~$\xi^\white$.
\end{definition}

\begin{figure}
\includegraphics[width=\textwidth]{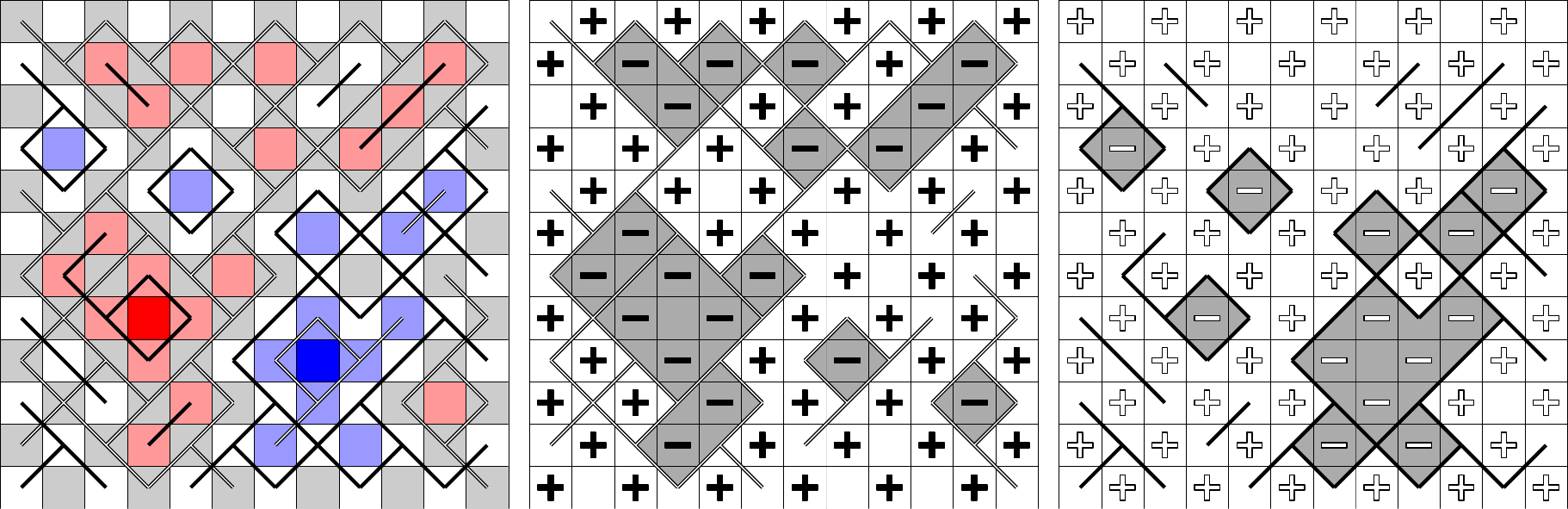}
	\caption{%
	A graph homomorphism~$h$ with the corresponding black and white
	spin configurations~$(\sigma^\black,\sigma^\white)$ and a sample of the
	black and white bond percolations~$(\xi^\black,\xi^\white)$.  \textsc{Left}:
	Colours of the faces describe the heights: the heights $-2$ to $3$ are dark
	blue, light blue, white, grey, light red, and dark red respectively.
	Super-duality holds true: for every edge~$e$ of the black sublattice,
	either~$e\in \xi^\black$ or~$e^*\in\xi^\white$; sometimes both
	hold true.  \textsc{Middle}: The black spin configuration~$\sigma^\black$ is
	determined by~$h$, the white bond percolation~$\xi^\white$ is sampled
	given~$h$.  All edges in the domain walls of~$\sigma^\black$ must be
	in~$\xi^\white$ and their duals cannot be in~$\xi^\black$.  \textsc{Right}:
	The white spin configuration~$\sigma^\white$ together with the black bond
	percolation~$\xi^\black$.
	}
	\label{fig:sixvertex_percolations}
\end{figure}

% Similarly, the set $\xi^\white$ is the disjoint union of the sets $\xi^\whitep$ and $\xi^\whitem$.

Remark that \(\omega[\sigma^\white]\subset \xi^\black\) and \(\omega[\sigma^\black]\subset\xi^\white\).
Note that when $a,b\leq c$, conditional on $(\sigma^\black,\sigma^\white)$,  the random variable $\xi^\black$ is an independent bond
percolation with the parameter at an edge $uv$ being~$0$, $1$, $a/c$ or~$b/c$
respectively in the four cases.
As in the random Lipschitz function, the coupling between
$\xi^\black$ and $\xi^\white$ is crucial for the super-duality relation:
whenever $a,b\leq c \leq a+b$, we have
\begin{equation}
	\label{eq:super_duality-6v}
		\intvert(\domain)
		=
		\{x:d_x^\black\in\xi^\black\}\cup\{x:d_x^\white\in\xi^\white\}.
\end{equation}
The duality is exact when $a+b=c$.

\begin{lemma}\label{lemma:marginal-black-6v}
	Let~$\domain$ be a domain on~$\Z^2$ and~$a,b\leq c$. Then,
	\begin{align}\label{eq:marginal-black-6v}
		\muSpin_{\domain,a,b,c}(\sigma^\black,\sigma^\white,\xi^\black) = \frac1{Z_{\domain,a,b,c}}
		&\cdot \left(\tfrac{a}{c}\right)^{|\xi^\black\cap \edgesA|}
		\cdot \left(1-\tfrac{a}{c}\right)^{|((\xi^\black)^*\setminus \omega[\sigma^\black])\cap \edgesB|}
		\cdot \left(\tfrac{a}{c}\right)^{|\omega[\sigma^\black]\cap \edgesA|}\\
		&\cdot \left(\tfrac{b}{c}\right)^{|\xi^\black\cap \edgesB|}
		\cdot \left(1-\tfrac{b}{c}\right)^{|((\xi^\black)^*\setminus \omega[\sigma^\black])\cap \edgesA|}
		\cdot \left(\tfrac{b}{c}\right)^{|\omega[\sigma^\black]\cap \edgesB|}\\
		&\cdot \ind{\sigma^\black\perp\xi^\black}
		\cdot \ind{\sigma^\white\perp(\xi^\black)^*}.\nonumber
	\end{align}
	The laws of~$(\sigma^\black,\sigma^\white,\xi^\black)$
	under~$\muSpin_{\domain,a,b,c}^\blackp$, $\muSpin_{\domain,a,b,c}^\whitep$,
	and~$\muSpin_{\domain,a,b,c}^{\blackp\,\whitep}$ are obtained by inserting
	indicators for the boundary values and updating the partition function.
\end{lemma}

Remark that, unlike for Lipschitz functions, we never need the indicator~$\ind{\sigma^\black\perp\sigma^\white}$.
The proof is essentially the same as that of Lemma~\ref{lemma:lip-marginal-black};
we provide the proof for completeness.

\begin{proof}
	Definitions~\ref{def:spin-measure-six} and~\ref{def:b-w-perco-6v} describe
	the probability of a pair of spin configurations
	$(\sigma^\black,\sigma^\white)$ and the conditional probability of a
	percolation configuration~$\xi^\black$ given~$(\sigma^\black,\sigma^\white)$
	respectively (recall that the opening probabilities are $0$, $1$, $a/c$, and $b/c$ in the
	four cases).
	The joint law is given by multiplying the two probabilities:
	\begin{align}
		\label{lineone}
		\muSpin_{\domain,a,b,c}(\sigma^\black,\sigma^\white,\xi^\black) ={}
		\frac1{Z_{\domain,a,b,c}}
		&
		\cdot
		\ind{\sigma^\black\perp\sigma^\white}
		\cdot
		\left(\tfrac{a}{c}\right)^{|(\omega[\sigma^\black]\cup\omega[\sigma^\white])\cap \edgesA|}
		\left(\tfrac{b}{c}\right)^{|(\omega[\sigma^\black]\cup\omega[\sigma^\white])\cap \edgesB|}
		\\
		&
		\label{linetwo}
		\cdot
		\ind{\sigma^\black\perp\xi^\black}
		\cdot \ind{\sigma^\white\perp(\xi^\black)^*}
		\\
		&
		\label{linethree}
		\cdot
		\left(\tfrac{a}c\right)^{|(\xi^\black\setminus\omega[\sigma^\white])\cap E_a|}
		\cdot
		\left(1-\tfrac{a}c\right)^{|((\omega[\sigma^\black])^*
	\setminus\xi^\black)\cap E_a|}
	\\&
	\label{linefour}
	\cdot
	\left(\tfrac{b}c\right)^{|(\xi^\black\setminus\omega[\sigma^\white])\cap E_b|}
	\cdot
	\left(1-\tfrac{b}c\right)^{|((\omega[\sigma^\black])^*
\setminus\xi^\black)\cap E_b|}.
\end{align}
Indeed,~\eqref{lineone} captures the probability of $(\sigma^\black,\sigma^\white)$;~\eqref{linetwo} captures the first two cases in Definition~\ref{def:b-w-perco-6v},
and~\eqref{linethree}-\eqref{linefour} capture the third and fourth cases in Definition~\ref{def:b-w-perco-6v}
respectively.
To obtain Lemma~\ref{lemma:marginal-black-6v},
observe that, knowing that $\sigma^\black\perp\xi^\black$ and
$\sigma^\white\perp(\xi^\black)^*$:
\begin{enumerate}
	\item The indicator $\ind{\sigma^\black\perp\sigma^\white}$ becomes superfluous;
	\item We have $|(\omega[\sigma^\black]\cup\omega[\sigma^\white])\cap \edgesA|+|(\xi^\black\setminus\omega[\sigma^\white])\cap E_a|=|\xi^\black\cap \edgesA|+|\omega[\sigma^\black]\cap \edgesA|$;
	\item We have~$|((\omega[\sigma^\black])^*\setminus\xi^\black)\cap E_a|
	= |((\xi^\black)^*\setminus \omega[\sigma^\black])\cap \edgesB|$;
	\item The previous two equations also hold true with $a$ and $b$ interchanged.
\end{enumerate}
This finishes the proof.
\end{proof}

\subsection{Markov property}

\begin{lemma}[Sampling~$\sigma^\white$ given~$\xi^\black$]
	\label{lemma:decoupling-6v}
	Let~$\domain$ be a domain and $0<a,b\leq c$.
	Let~$\tau^\black \in \{\pm 1\}^{V(\domain^\black)}$
	and~$\zeta^\black\subset\edges{\domain^\black}$ be such
	that~$\muSpin_{\domain,a,b,c}(\sigma^\black=\tau^\black,\xi^\black=\zeta^\black)>0$.
	Then,
	\[
		\text{the law of $\sigma^\white$ under }
		\muSpin_{\domain,a,b,c}(\blank |  \sigma^\black=\tau^\black,\xi^\black=\zeta^\black)
	\]
	is given by independent fair coin flips valued $\pm$ for the
	connected components of~$(\xi^\black)^*$.
	If the boundary condition $\whitep$ is imposed,
	then the only difference is that all clusters of~$(\xi^\black)^*$
	intersecting~$\partial_{F^\white}\domain$ are deterministically assigned the value $+$.
\end{lemma}

\begin{proof}
	By Lemma~\ref{lemma:marginal-black-6v},
	\begin{align}
		\muSpin_{\domain,a,b,c}(\sigma^\white\, \vert \, \sigma^\black,\xi^\black)
		\propto \ind{\sigma^\white\perp(\xi^\black)^*} =
		\ind{\text{$\sigma^\white$ is constant on each cluster of $(\xi^\black)^*$}}.
	\end{align}
	The boundary condition $\whitep$ introduces the extra indicator~$\ind{\sigma^\white|_{\partial_{F^\white}\domain} \equiv +}$.
\end{proof}

An \emph{even domain} is a domain $\domain$ whose squares $F(\domain)$ are precisely those
squares on or enclosed by a cycle through $\squarelattice^\black$.
We denote the set of edges on this cycle by~$\partial^\black_E\domain\subset\squarelattice^\black$.
Note that an arbitrary domain $\domain$ is even if and only if all squares sharing an \emph{edge} with $\partial\domain$
are black.
\emph{Odd domains} are defined similarly.

\begin{lemma}[Domain Markov property]
	\label{lemma:markov-6v}
	Let $\domain'$ denote any domain,
	and suppose that $\domain\subset\domain'\setminus\partial\domain'$
	denotes an even domain.
	Suppose that~$0<a,b\leq c$.
	Then, under
	\[
		\muSpin_{\domain',a,b,c}(\blank|\partial_E^\black\domain\subset \xi^\blackp),
	\]
	the law of the triplet $(\sigma^\black,\sigma^\white,\xi^\black)$ restricted to faces in $\face(\domain)$ and edges in $E(\domain^\black)$:
	\begin{enumerate}
		\item Is independent of the restriction to the complementary faces and edges, and
		\item Is precisely equal to~$\muSpin_{\domain,a,b,c}^{\blackp}$.
	\end{enumerate}
	In particular, this remains true when boundary conditions are imposed on $\muSpin_{\domain',a,b,c}$.
\end{lemma}

The proof is very similar to that of Lemma~\ref{lemma:loop_markov}. We provide it for completeness.

\begin{proof}
	By Lemma~\ref{lemma:marginal-black-6v},
	\begin{align*}
		\muSpin_{\domain',a,b,c}(\sigma^\black,\sigma^\white,\xi^\black\,\vert\, \partial_E^\black\domain\subset \xi^\blackp)
		\propto &
		\left(\tfrac{a}{c}\right)^{|(\xi^\black\setminus\partial_E^\black\domain)\cap \edgesA|}
		\cdot \left(1-\tfrac{a}{c}\right)^{|((\xi^\black)^*\setminus \omega[\sigma^\black])\cap \edgesB|}
		\cdot \left(\tfrac{a}{c}\right)^{|\omega[\sigma^\black]\cap \edgesA|}\\
		\cdot
		& \left(\tfrac{b}{c}\right)^{|(\xi^\black\setminus\partial_E^\black\domain)\cap \edgesB|}
		\cdot \left(1-\tfrac{b}{c}\right)^{|((\xi^\black)^*\setminus \omega[\sigma^\black])\cap \edgesA|}
		\cdot \left(\tfrac{b}{c}\right)^{|\omega[\sigma^\black]\cap \edgesB|}\\
		\cdot
		& \, \ind{\sigma^\black\perp(\xi^\black\setminus\partial_E^\black\domain)}
		\cdot \ind{\sigma^\white\perp(\xi^\black)^*}
		\cdot \ind{\partial_E^\black\domain\subset\xi^\black}
		\cdot \ind{\left.\sigma^\black\right|_{\partial_{F^\black} \domain}\equiv +}.
	\end{align*}
	Since~$\partial_E^\black\domain \subseteq \xi^\black$, all factors except the last two can be written as a product of two terms:
	one depending only on $\face(\domain)$, $E(\domain^\black)\setminus \partial_E^\black \domain$ and~$E(\domain^\white)$
	and the other only on $\face(\domain')\setminus \face(\domain)$, $E((\domain')^\black)\setminus E(\domain^\black)$ and~$E((\domain')^\white)\setminus E(\domain^\white)$.
	Finally, taking the product of the parts of the terms that depend on~$\face(\domain)$, $E(\domain^\black)$ and~$E(\domain^\white)$, we obtain precisely $\muSpin_{\domain,a,b,c}^\blackp$, by Lemma~\ref{lemma:marginal-black-6v}.
\end{proof}

\subsection{FKG and extremality}

It was shown in~\cite[Theorem~4]{GlaPel19} that the black and white spins satisfy the FKG inequality.
Moreover, in exactly the same way as in the loop~\O{2} model, the statement can be upgraded till the FKG inequality for the triplet~$(\sigma^\black, \xi^\blackp, -\xi^\blackm)$.

%\begin{lemma}[\cite[Theorem~4]{GlaPel18}, FKG for spins]
%	\label{lemma:fkg-spins-6v}
%	Let $\domain$ be a domain and~$0<a,b\leq c$.
%	Then $\sigma^\black$ satisfies the FKG lattice condition in $\muSpin_{\domain,a,b,c}$.
%	The same holds also for~$\sigma^\white$ and when boundary conditions of the form $\blackp,\blackm$  and/or $\whitep,\whitem$ are imposed.
%\end{lemma}

\begin{proposition}[FKG for triplets]
	\label{prop:six_vertex_fkg_inequality_new}
	\label{prop:fkg-joint-six}
	Let~$\domain$ be a domain and~$0<a,b\leq c$.
	Then, under~$\muSpin_{\domain,a,b,c}$, the triplet $(\sigma^\black,\xi^\blackp,-\xi^\blackm)$ satisfies the FKG inequality.
	The same holds also for $(\sigma^\white,\xi^\whitep,-\xi^\whitem)$ and when boundary conditions of the form $\blackp,\blackm$  and/or $\whitep,\whitem$ are imposed.

	In fact, the triplet $(\sigma^\black,\xi^\blackp,-\xi^\blackm)$ also satisfies the FKG inequality
	under the measures
	\[
		\muSpin_{\domain,a,b,c}(\blank|\{\sigma^\black|_A=\tau\})
		\quad
		\text{and}
		\quad
		\muSpin_{\domain,a,b,c}^\whitep(\blank|\{\sigma^\black|_A=\tau\})
	\]
	for any $A\subset \vertices{\domain^\black}$ and $\tau:A\to\{\pm1\}$.
\end{proposition}

Below we provide the proof of the FKG inequality for~$\sigma^\black$ for completeness and because we will need to extend the statement to a slightly different setting in Section~\ref{sec:continuity}.
The argument is very similar to that of Lemma~\ref{lemma:fkg-spins-lip}.

\begin{proof}
	Recall that~$\Zising(\mathrm{w})$ defined by~\eqref{eq:fk-ising} is increasing and satisfies the FKG lattice condition.
	We sum~\eqref{eq:6v-meas-via-edges} over~$\sigma^\white\in\{\pm1\}^{\vertices{\domain^\white}}$ to get
	\begingroup
	\setlength{\jot}{-10pt}
	\begin{align}
		\muSpin_{\domain,a,b,c}(\sigma^\black) \cdot Z_{\domain}
			&=
			\left(\tfrac{a}{c}\right)^{|\omega[\sigma^\black]\cap \edgesA|}
			\left(\tfrac{b}{c}\right)^{|\omega[\sigma^\black]\cap \edgesB|}
			\cdot \hspace{-5mm}\sum_{\sigma^\white\in\{\pm1\}^{\vertices{\domain^\white}}}
			\hspace{-4mm}\ind{\sigma^\black\perp\sigma^\white}
			\cdot
			\left(\tfrac{a}{c}\right)^{|\omega[\sigma^\white]\cap \edgesA|}
			\left(\tfrac{b}{c}\right)^{|\omega[\sigma^\white]\cap \edgesB|}
			\\
			&=
			\left(\tfrac{a}{c}\right)^{|\omega[\sigma^\black]\cap \edgesA|}
			\left(\tfrac{b}{c}\right)^{|\omega[\sigma^\black]\cap \edgesB|} \cdot
			\Zising(\mathrm{w}(\sigma^\black)),\label{eq:fkg-proof-measure-via-ising-6v}
	\end{align}
	\endgroup
	where $\mathrm{w}(\sigma^\black):\edges{\domain^\white}\to[0,1]$ is defined by
	\[
		\mathrm{w}(\sigma^\black)_{uv} =
			J_{uv}
			\cdot
			\ind{\sigma^\black \perp (uv)^*}
			;
			\qquad
			J_{uv}:=
			\begin{cases}
				a/c &\text{if $uv\in \edgesA$,}\\
				b/c &\text{otherwise.}
			\end{cases}
	\]
%	\begin{align}
%		\begin{cases}
%			0&\text{if }\sigma^\black(u)\neq \sigma^\black(v),\\
%			x&\text{if }\sigma^\black(u)= \sigma^\black(v) \text{ and } uv\in (Y(\domain))^*,\\
%			1&\text{if }\sigma^\black(u)= \sigma^\black(v) \text{ and } uv\not\in (Y(\domain))^*.
%		\end{cases}
%	\end{align}
	The function
	$\sigma^\black\mapsto
	\left(\tfrac{a}{c}\right)^{|\omega[\sigma^\black]\cap \edgesA|} \left(\tfrac{b}{c}\right)^{|\omega[\sigma^\black]\cap \edgesB|}$ in~\eqref{eq:fkg-proof-measure-via-ising-6v} satisfies the FKG
	lattice condition whenever~$0<a,b\leq c$
	and therefore it suffices to prove the same for~$\sigma^\black\mapsto  \Zising(\mathrm{w}(\sigma^\black))$.
	We claim that the following factorisation holds:
	\begin{equation}
		\label{eq:fkg-proof-ising-decomp-pm-6v}
		2\Zising(\mathrm{w}(\sigma^\black))= \Zising(\mathrm{w}^+(\sigma^\black))\cdot\Zising(\mathrm{w}^-(\sigma^\black)),
	\end{equation}
	where~$\mathrm{w}^\pm(\sigma^\black)$ is defined by
	\[
		\mathrm{w}^\pm(\sigma^\black)_{uv} = J_{uv} \cdot\ind{\sigma^\black\, \equiv \,\pm  \text{ at the endpoints of } (uv)^*}.
	\]

	Fix~$\sigma^\black$.
	As in the proof of Lemma~\ref{lemma:fkg-spins-lip}, we represent $\sigma^\white$ conditional on $\sigma^\black$ as the product
	of two independent Ising models.
	For any~$\sigma^\white\perp\sigma^\black$,
	we partition the domain walls~$\omega[\sigma^\white]$ into
	connected components that lie inside $\{\sigma^\black=+\}$ and those that lie inside $\{\sigma^\black=-\}$:
	\[
		\omega^+[\sigma^\white]:=\{e\in \omega[\sigma^\white]\colon\text{$\sigma^\black \equiv +$ on $e^*$}\}
		\quad
		\text{and}
		\quad
		\omega^-[\sigma^\white]:=\{e\in \omega[\sigma^\white]\colon\text{$\sigma^\black \equiv -$ on $e^*$}\}.
	\]
	Since~$\domain^\white$ is simply-connected, we can find a spin configuration $\zeta^+:\vertices{\domain^\white}\to \{\pm\}$
	such that $\omega[\zeta^+]=\omega^+[\sigma^\white]$.
	In fact, there exists exactly one other such spin configuration, namely~$-\zeta^+$.
	Similar considerations apply to $\zeta^-:=\sigma^\white/\zeta^+$.
	Observe that $(\zeta^+,\zeta^-)$ and $(-\zeta^+,-\zeta^-)$ are the only two
	pairs which separate the domain walls in the prescribed way and which factorise $\sigma^\white$.
	Conversely, if $\zeta^+,\zeta^-:\vertices\domain^\white\to \{\pm 1\}$ are two spin configurations such that
	$\omega^+[\zeta^+]=\omega[\zeta^+]$ and~$\omega^-[\zeta^-]=\omega[\zeta^-]$, then $\sigma^\white:=\zeta^+\zeta^-$
	satisfies $\sigma^\white\perp\sigma^\black$.
	In addition, the contribution of~$\sigma^\white$
	to~$\Zising(\mathrm{w}(\sigma^\black))$ is equal to the product of the contributions
	of~$\zeta^+$ to~$\Zising(\mathrm{w}^+(\sigma^\black))$ and of~$\zeta^-$
	to~$\Zising(\mathrm{w}^-(\sigma^\black))$.  Summing over all
	pairs~$(\zeta^+,\zeta^-)$ yields~\eqref{eq:fkg-proof-ising-decomp-pm-6v}.

	The end of the proof of Lemma~\ref{lemma:fkg-spins-lip} shows that both maps $\sigma^\black\mapsto \Zising(a^+(\sigma^\black))$
	and~$\sigma^\black\mapsto \Zising(a^-(\sigma^\black))$ satisfy the FKG lattice condition.
	Then, so does~$\sigma^\black\mapsto \Zising(a^+(\sigma^\black))$ by~\eqref{eq:fkg-proof-ising-decomp-pm-6v} and the proof in the case of~$\muSpin_{\domain,a,b,c}$ is finished.

	To prove the same for $\muSpin_{\domain,a,b,c}^\whitep$, we note that the distribution of~$\sigma^\black$ under this measure is the same as under
	\[
		\muSpin_{\domain,a,b,c}(\blank|
			\{\sigma^\white|_{\partial_{\face^\white}\domain}\equiv +\}
			\cup
			\{\sigma^\white|_{\partial_{\face^\white}\domain}\equiv -\}
			);
	\]
	we now can run the same proof as above with
	\(
		J'_{uv}:=J_{uv}\cdot\ind{
			\{u,v\}\not\subset\partial_{\face^\white}\domain
		}
	\).
\end{proof}

We now state the equivalent of Proposition~\ref{prop:inf-vol-lip}.

\begin{proposition}[Convergence under $\blackp$ boundary conditions]
	\label{prop:inf-vol-6v}
	Let~$0<a,b\leq c$.
	Then, for any sequence of even domains~$(\domain_k)_k\nearrow\hexlattice$,
	we have the following weak convergence:
	\[
		\muSpin_{\domain_k,a,b,c}^\blackp\xrightarrow[k\to\infty]{}\muSpin_{a,b,c}^\blackp,
	\]
	where $\muSpin_{a,b,c}^\blackp$ is a Gibbs measure for the spin representation independent
	of the choice of the sequence $(\domain_k)_k$ and invariant under the symmetries of
	 $\squarelattice^\black$.
	 Moreover:
	 \begin{enumerate}
		\item The law of $(\sigma^\black,\xi^\black)$ is extremal.
		\item Both $(\sigma^\black,\xi^\blackp,-\xi^\blackm)$ and~$(\sigma^\white,\xi^\whitep,-\xi^\whitem)$
		satisfy the FKG inequality in $\muSpin_{a,b,c}^\blackp$.
		\item \label{prop:inf-vol-6v:three}
		The distribution of~$\sigma^\white$ given~$(\sigma^\black,\xi^\black)$ is obtained by assigning~$\pm$ to the clusters of~$(\xi^\black)^*$ via independent fair coin flips.
		\item \label{prop:inf-vol-6v:four}
		The distribution of~$\sigma^\black$ given~$(\sigma^\white,\xi^\white)$
		is obtained by assigning~$\pm$ to the clusters
		of~$(\xi^\white)^*$ via independent fair coin flips,
		except that the infinite cluster (if it exists) is assigned the
		value $+$.
	 \end{enumerate}
\end{proposition}

The proof is very similar to that of Proposition~\ref{prop:inf-vol-lip}. We provide it for completeness.

\begin{proof}
	As in Proposition~\ref{prop:inf-vol-lip}, it is enough to establish convergence for the joint distribution of spins~$(\sigma^\black,\sigma^\white$.
	We start by the marginal on~$\sigma^\black$ and~$\xi^\black$.
	Due to the FKG property (Proposition~\ref{prop:fkg-joint-six}), the
	law of $(\sigma^\black,\xi^\blackp,-\xi^\blackm)$
	under~$\muSpin_{\domain,a,b,c}^\blackp$
	is stochastically decreasing
	in (even domains) $\domain$.
	By standard arguments~\cite[Proposition~4.10b,
	Theorem~4.19, Corollary~4.23]{Gri06},
	the distribution of this triplet converges weakly to some measure $\mu$
	as $\domain\nearrow\Z^2$,
	and this limit is translation-invariant, ergodic, extremal, and satisfies the FKG inequality.

	In finite volume, the conditional distribution of $\sigma^\white$ given $\sigma^\black$
	is given by flipping fair coins for the connected components of $(\xi^\black)^*$
	(Lemma~\ref{lemma:decoupling-6v}).
	This property commutes with taking the limit as~$\domain\nearrow \Z^2$ since $\mu$ exhibits at most one infinite connected component in~$(\xi^\black)^*$, by the argument of
	Burton and Keane~\cite{BurKea89},
	Thus, $\muSpin_{a,b,c}^\blackp$ is well-defined and its marginal on~$(\sigma^\black,\sigma^\white)$ can be sampled from~$\mu$ via the above algorithm.

	The FKG inequality for $(\sigma^\white,\xi^\whitep,-\xi^\whitem)$
	in $\muSpin_{x}^\blackp	$ follows from~\cite[Proposition~4.10b]{Gri06}.
	Finally, Parts~\ref{prop:inf-vol-6v:three} and~\ref{prop:inf-vol-6v:four} follow from Lemma~\ref{lemma:decoupling-6v} via standard arguments.
\end{proof}

For~$a,b\leq c\leq a+b$, we obtain full extremality of the measure $\muSpin_{a,b,c}^\blackp$.

\begin{proposition}[Extremality]\label{prop:ergo-6v}
	Assume that~$0<a,b\leq c \leq a+b$.
	Then, the measure $\muSpin_{a,b,c}^\blackp$ is extremal and ergodic.
%	\begin{equation}
%		\label{6v:largewhitecircuits}
%		\muSpin_{a,b,c}^\blackp(\text{each face is surrounded by infinitely many disjoint $\omega[\sigma^\white]$-cycles}) = 1.
%	\end{equation}
\end{proposition}

We argue as in the proof of Proposition~\ref{prop:ergo-lip}.
The main difference is that we must replace the Zhang argument in the proof by the more general
non-coexistence theorem of Sheffield~\cite[Theorem~9.3.1]{She05} (see also~\cite[Theorem~1.5]{DumRaoTas19}) which does not require invariance
under the $\pi/2$ rotation.
We provide the details for completeness.

\begin{proof}
	By Proposition~\ref{prop:inf-vol-6v}, the law of $\xi^\black$ under~$\muSpin_{a,b,c}^\blackp$ is extremal and ergodic.
	As in Proposition~\ref{prop:ergo-lip}, extremality and ergodicity of~$\sigma^\white$ will follow once we show that all clusters of $(\xi^\black)^*$ are finite.
	Assume the opposite:
	\begin{equation}
		\label{eq:whatwewanttocontradict-6v}
		\muSpin_x^\blackp(\text{$(\xi^\black)^*$ has an infinite cluster})>0
	\end{equation}
	Using the super-duality~\eqref{eq:super_duality-6v}, the symmetry between~$\xi^\whitep$ and~$\xi^\whitem$ under~$\muSpin_x^\blackp$ and the fact that each cluster of~$\xi^\white$ is a cluster of~$\xi^\whitep$ or a cluster of~$\xi^\whitem$,
	 \begin{equation}
		\label{eq:whatwewanttocontradict-2-6v}
		\muSpin_x^\blackp(\text{$\xi^\whitep$ has an infinite cluster})>0.
	\end{equation}
	By the comparison between boundary conditions, for any domain~$\domain$, the distribution of~$\xi^\whitep$ under~$\muSpin_{\domain,x}^\whitep$ dominates that under~$\muSpin_x^\blackp$.
	Taking the limit~$\domain\nearrow \hexlattice$ as in Proposition~\ref{prop:inf-vol-6v} (but for~$\xi^\whitep$) and using black/white symmetry and~\eqref{eq:whatwewanttocontradict-2-6v}, we get
	\[
		\muSpin_x^\blackp(\text{$\xi^\blackp$ has an infinite cluster})
		= \muSpin_x^\whitep(\text{$\xi^\whitep$ has an infinite cluster}) >0.
	\]
	Together with~\eqref{eq:whatwewanttocontradict-6v}, this contradicts Sheffield's non-coexistence theorem.
\end{proof}

Recall that in the loop~\O{2} model ergodicity readily implies delocalisation and~$\muSpin_x^\blackp = \muSpin_x^\blackm$ (Corollary~\ref{cor:inf-vol-limits-equal-lip}).
The situation is more involved in the six-vertex model because~$\sigma^\black$ and~$\sigma^\white$ are supported on different lattices and hence we cannot swap them easily.

\subsection{Delocalisation via $\T$-circuits}

\begin{proof}[Proof of qualitative delocalisation (Equation~(\ref{thm:hard_six_deloc:qual}) in Theorem~\ref{thm:hard_six_deloc})]
	We follow the alternative proof for delocalisation in the uniform
	case~$a=b=c$ suggested in~\cite[Section~9]{GlaPel19}.  The innovation
	of that argument lies in so-called {\em $\T$-circuits} (that is, circuits in the
	triangular connectivity) and a related coupling (see Fig.~\ref{fig:T-circuits}).

	Omit~$a,b,c$ in the proof for brevity, and fix $(\Omega_k)_k$.
	We first show delocalisation assuming
	\begin{equation}
		\label{eq:largebothcircuits-6v}
		\muSpin^\blackp\left(\parbox{25em}{each face is surrounded by infinitely many disjoint $\xi^\black$-circuits and infinitely many disjoint $\xi^\white$-circuits}\right) = 1.
	\end{equation}
%
%	By the FKG inequality and~\eqref{eq:no-inf-cluster-of-pluses-6v},
%	the set  $\{\sigma^\black = -\}$ does not $\squarelattice^\black$-percolate $\muSpin^\blackp$-almost surely.
%	This implies the following counterpart to~\eqref{6v:largewhitecircuits}:
%	\begin{equation}
%		\label{6v:largeblackcircuits}
%		\muSpin^\blackp(\text{each face is surrounded by infinitely many disjoint $\omega[\sigma^\black]$-cycles}) = 1.
%	\end{equation}
%	Since $\omega[\sigma^\white]\subset\xi^\black$ and $\omega[\sigma^\black]\subset\xi^\white$,~\eqref{6v:largewhitecircuits}
%	and~\eqref{6v:largeblackcircuits} imply
%	\begin{equation}
%		\label{eq:largebothcircuits-6v}
%		\muSpin^\blackp\left(\parbox{25em}{each face is surrounded by infinitely many disjoint $\xi^\black$-cycles and infinitely many disjoint $\xi^\white$-cycles}\right) = 1.
%	\end{equation}

	Then, Proposition~\ref{prop:inf-vol-6v} implies $\muSpin^\blackp=\muSpin^\blackm$ and, as in the proof of Lemma~\ref{lemma:thm1aux}, we get
	\begin{equation}
		\label{eq:contrathing}
		\muSpin^{\blackp\,\whitep}_{\domain_k}
		\xrightarrow[k\to\infty]{}
		\muSpin^\blackp.
	\end{equation}
%	We now derive~\eqref{thm:hard_six_deloc:qual} from~\eqref{eq:contrathing} and~\eqref{eq:largebothcircuits-6v}.
	In a slight abuse of notation, we allow the circuits to consist of only one face~--- our target face~$u$.
	For an integer~$k\geq 1$, sample $(\sigma^\black,\sigma^\white,\xi^\black,\xi^\white)$ according to~$\muSpin^{\blackp\, \whitep}_{\domain_k}$ and explore it as follows:
	\begin{enumerate}
		\item Define~$\gamma_1$ as the largest circuit of~$\xi^\white$ going around or through~$u$.
		\item For~$i\geq 1$, inductively define~$\gamma_{2i}$ as the largest circuit of~$\xi^\black$ inside~$\gamma_{2i-1}$ and~$\gamma_{2i+1}$ as the largest circuit of~$\xi^\white$ inside~$\gamma_{2i}$.
		\item We stop when arrive at a circuit that contains~$u$.
	\end{enumerate}
	Let $N_{k}$ be the random variable equal to the number of successfully constructed circuits.
	For~$i=1,\dots,N_k$, define~$\calD_i$ as the set of faces in~$\domain_k$ on~$\gamma_i$ or outside of it.
	Note that, for every~$i\geq 1$ even (resp. odd), $\sigma^\black$ (resp.~$\sigma^\white$) takes a constant value on~$\gamma_i$.
	The spin at~$\gamma_1$ has to be~$+$, since otherwise there must be a larger circuit of a different parity.
	For all~$i\geq 2$, by the domain Markov property (Lemma~\ref{lemma:markov-6v}), the spin at~$\gamma_i$, conditioned on the values~$(\sigma^\black,\sigma^\white,\xi^\black,\xi^\white)$ on~$\calD_i$, is~$+$ or~$-$ with probability~$1/2$.

	Let~$\P_{\domain_k}$ denote the coupling of~$\muHom^{0,1}_{\domain_k}$ and~$\muSpin^{\blackp\, \whitep}_{\domain_k}$ induced by Proposition~\ref{prop:6v-spin-bij}.
	The correspondence implies that the height function~$h$ takes constant value on~$\gamma_i$ (that we denote by~$h(\gamma_i)$).
	Moreover,~$h(\gamma_1) = 1$ if~$\gamma_1$ is a~$\xi^\whitep$-circuit and
	$h(\gamma_1) = -1$ otherwise.
	We now show that, for~$i\geq 2$,
	\begin{equation}\label{eq:height-given-exterior-circuits}
		\P_{\domain_k}(h(\gamma_i) = h(\gamma_{i-1}) + 1  \,\vert \, h_{|\calD_i})
		= \P_{\domain_k}(h(\gamma_i) = h(\gamma_{i-1}) - 1  \, \vert \, h_{|\calD_i}) = 1/2.
	\end{equation}
	We only treat the case when~$\gamma_i$ is a~$\xi^\whitep$-circuit, since the case of a~$\xi^\whitem$-circuit is analogous.
	Without loss of generality, we can assume that~$\sigma^\white \equiv +$ on~$\gamma_{i-1}$.
	The cluster~$\calC_i$ of~$\{\sigma^\white = +\}$ on~$\squarelattice^\white$ containing~$\gamma_{i-1}$ must include a face adjacent to a face on~$\gamma_i$.
	Indeed, otherwise there exists a circuit of edges of~$\squarelattice^\black$ whose exterior boundary consists of faces of~$\calC_i$, and on its interior boundary the spin~$\sigma^\white$ takes a constant value~$-$.
	Clearly, this circuit must belong to~$\xi^\black$, which leads to a contradiction.

	Finally, applying~\eqref{eq:height-given-exterior-circuits} consecutively for~$i=2,3,\dots,N_k$, we see that, given~$N_k$ and~$\gamma_1$, \dots, $\gamma_{N_k}$, the law of~$h(u)$ is that of the simple random walk starting at~$h(\gamma_1)$ and making~$N_k-1$ steps.
	Recall that, in the case of the random Lipschitz function, the corresponding statement with the circuits~$\gamma_i$ replaced by the loops is almost trivial.
	By the total variance formula,
	\begin{align}
		\Var_{\domain_k}[h(u)]
		&= \E_{\domain_k}[\Var_{\domain_k}[h(u) \vert N_k, \gamma_1]] + \Var_{\domain_k}[\E_{\domain_k}[h(u) \vert N_k, \gamma_1]] \notag\\
		&= \E_{\domain_k}[N_k - 1] + \Var_{\domain_k}[\ind{h(\gamma_1)=1} - \ind{h(\gamma_1)=-1}].
		\label{eq:var-of-graph-hom}
	\end{align}
	By~\eqref{eq:largebothcircuits-6v}, the measure~$\muSpin^\blackp$ exhibits infinitely many alternating~$\xi^\black$- and~$\xi^\white$-circuits almost surely.
	Then, for any~$L\geq 1$, for all~$k$ large enough,
	\[
		\muSpin_{\domain_k}^\blackp(N_k > L) > \tfrac12.
	\]
	Thus, the right-hand side of~\eqref{eq:var-of-graph-hom} diverges yielding~\eqref{thm:hard_six_deloc:qual}.

	It remains to show~\eqref{eq:largebothcircuits-6v}.
	By the super-duality~\eqref{eq:super_duality-6v}, it is enough to rule out existence of infinite clusters in~$\xi^\black$ and in~$\xi^\white$.
	The statement about~$\xi^\white$ follows easily form the ergodicity established in Proposition~\ref{prop:ergo-6v}:
	if an infinite cluster in~$\xi^\white$ exists, then it is unique by the argument of Burton and Keane~\cite{BurKea89}
	and, due to the invariance under a global flip of the white spin, this cluster is in~$\xi^\whitep$ with probability exactly~$1/2$.
	We now assume that~$\xi^\black$ contains an infinite cluster.
	By the above reasoning, this cluster is unique and must belong to~$\xi^\blackp$ almost surely.
	Define a finite- and an infinite-volume measures on graph homomorphisms under~$0$ boundary conditions.
	\begin{itemize}
		\item $\muHom_{\domain_k}^0$:
		sample a pair~$(\sigma^\black,\sigma^\white)$ according to~$\muSpin^\blackp$,
		assign height~$0$ to the boundary cluster of~$\{\sigma^\black = +\}$,
		and assign heights to all other faces according to~\eqref{eq:heights-to-spins-6v}.
		\item $\muHom^0$:
		sample a pair~$(\sigma^\black,\sigma^\white)$ according to~$\muSpin^\blackp$,
		assign height~$0$ to the infinite cluster of~$\{\sigma^\black = +\}$,
		and assign heights to all other faces according to~\eqref{eq:heights-to-spins-6v}.
	\end{itemize}
	It is easy to see that~$\muHom_{\domain_k}^0$ converges to~$\muHom^0$.
	Indeed, for any~$n\geq 1$ and~$\epsilon >0$, take~$N\geq n$ large enough so that, for any domain~$\domain \supset \Lambda_N$, the measures~$\muSpin_\domain^\blackp$ and~$\muSpin^\blackp$ can be coupled to agree on~$\Lambda_n$ with probability~$1-\eps$ and both measures with probability~$1-\eps$ exhibit a unique cluster of~$\{\sigma^\black = +\}$ crossing from~$\Lambda_n$ to~$\partial\Lambda_N$ and this cluster is connected to~$\partial \domain$ (for~$\muSpin_\domain^\blackp$) or to infinity (for~$\muSpin^\blackp$).
	Assuming this coupling and on the above event, the height functions agree on~$\Lambda_n$ and the statement follows.
	We omit some details because this is reminiscent to the classical convergence argument for Potts measures sampled from the random-cluster model via the Edwards--Sokal coupling~\cite[Theorem~4.91]{Gri06}.

	The finite-volume measures~$\muHom_{\domain_k}^0$ have FKG~\cite[Proposition~5.1]{GlaPel19}
	and therefore so does the limit~\cite[Proposition~4.10b]{Gri06}
	(see also~\cite[Lemma~9.2.1]{She05}).
	Similarly, define~$\muHom^1$: sample a pair~$(\sigma^\black,\sigma^\white)$ according
	to~$\muSpin^\whitep$, assign height~$1$ to the infinite cluster
	of~$\{\sigma^\white = +\}$, and assign heights to the other squares
	according to~\eqref{eq:heights-to-spins-6v}.
	Write $\preceq$ for stochastic domination of heights; we claim it is sufficient to prove
	\begin{equation}\label{eq:0-and-1-meas-equal-6v}
		\muHom^1\preceq \muHom^0  .
	\end{equation}
	Indeed, iterating this inequality yields $\muHom^4\preceq\muHom^0$
	where the former measure is simply obtained by shifting the height function
	up by four, which is clearly contradictory.

	\begin{figure}
		\begin{center}
			\includegraphics{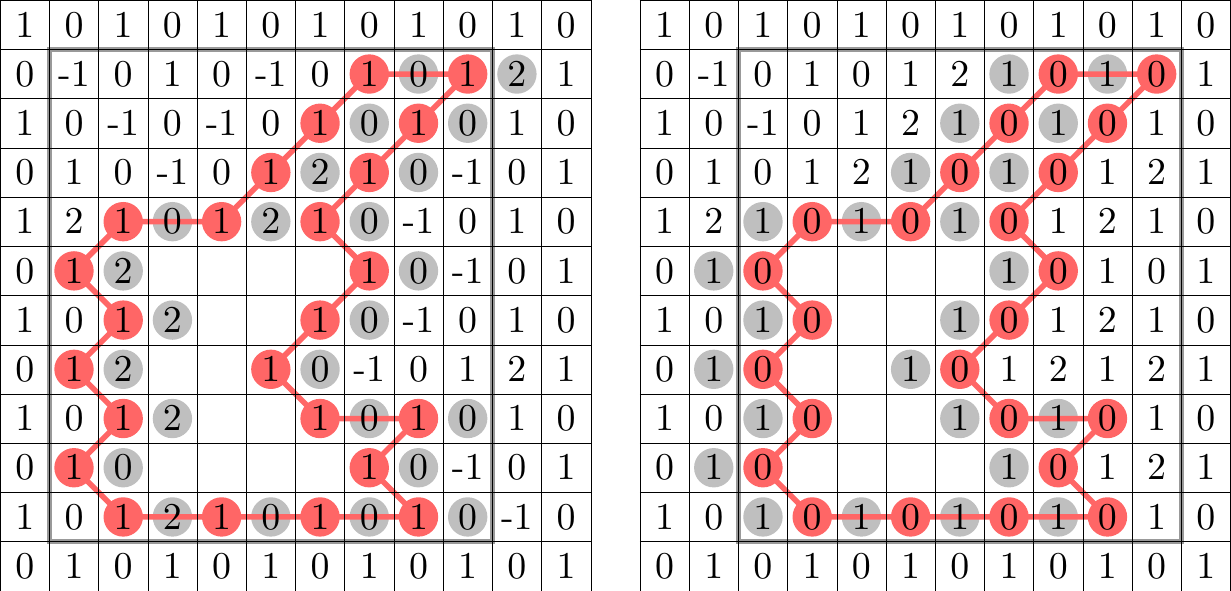}
			\caption{
				A coupling via $\T$-circuits for~$\muHom^0$ and~$\muHom^1$.
				{\scshape Left}: A sample $h^0\sim\muHom^0$.
				The circuit $\gamma_N$, in red,
				is the largest $\T^\white$-circuit through $\{h^0\geq 1\}$
				that is contained in $\Lambda_N$, the grey box.
				{\scshape Right}: The measures $\muHom^0$ and~$\muHom^1$
				can be coupled such that $\gamma_N+(1,0)$
				is the largest $\T^\black$-circuit through $\{h^1\leq0\}$
				contained in $\Lambda_N+(1,0)$,
				where $h^1\sim\muHom^1$.
				{\scshape Both}: The union of $\gamma_N$ and $\gamma_N+(1,0)$
				contains a path of faces on which $h^0\geq h^1$,
				so that the conditional distribution of $h^0$ on the unexplored
				part (the faces without height) dominates that of $h^1$.
			}
			\label{fig:T-circuits}
		\end{center}
	\end{figure}

	We show~\eqref{eq:0-and-1-meas-equal-6v} by applying the non-coexistence theorem to the {\em site percolation} duality on a certain triangulation.
	Define the graph~$\T^\black$ on the even faces $F^\black(\squarelattice)$ by connecting
	each face centred at~$(i,j)$ to the six faces centred at $(i\pm1,j\pm 1)$ and~$(i\pm
	2,j)$.  The graph~$\T^\black$ is isomorphic to the triangular lattice.
	Define the graph~$\T^\white$ on $F^\white(\squarelattice)$ in a similar fashion,
	that is, by shifting all edges in $\T^\black$ by $e_1:=(1,0)$.

	Symmetry and non-coexistence (\cite[Theorem~9.3.1]{She05}
	and~\cite[Theorem~1.5]{DumRaoTas19})
	imply that
	the set
	$\{h\in\{-1,-3,\dots\}\}$
	does almost surely not percolate for $\muHom^0$
	in the $\T^\white$-connectivity.
	By self-duality of site percolation on the triangular lattice,
	this means that $\muHom^0$-almost surely the set
	$\{h\in\{1,3,\dots\}\}$ contains infinitely many disjoint $\T^\white$-cycles around the origin.

	For $n\in\Z_{\geq 1}$ we let $\Lambda_n$ denote the set of faces
	whose centre is at a Euclidean distance at most $n$
	from $(0,0)$.
	To prove~\eqref{eq:0-and-1-meas-equal-6v},
	it suffices to find,
	for fixed $n\in\Z_{\geq 1}$ and $\epsilon>0$,
	a coupling $\pi$ of $h^0\sim\muHom^0$
	and $h^1\sim \muHom^1$ in such a way that
	\begin{equation}
		\label{condoncoupling}
		\pi(h^0|_{\Lambda_n}\geq h^1|_{\Lambda_n})\geq 1-\epsilon.
	\end{equation}
	We now construct this coupling.

	Let $P_N\subset\T^\white$ denote the set of edges in $\T^\white$
	connecting faces in $\Lambda_N\cap \{h\in\{1,3,\dots\}\}$.
	Let $\gamma_N\subset P_N$ denote the largest circuit
	of $P_N$ surrounding $(0,0)$, or set $\gamma_N:=\varnothing$
	if such a circuit does not exist.
	Write $I(\gamma_N)\subset\faces{\squarelattice}$ for the set of faces
	strictly surrounded by $\gamma_N$.
	Since $\cup_N P_N$ contains infinitely many disjoint circuits around
	the origin $\muHom^0$-almost surely, we may find a fixed $N\in\Z_{\geq 1}$
	so that
	\[
		\muHom^0(\Lambda_{n+8}\subset I(\gamma_N))
		\geq 1-\varepsilon.
	\]
	Observe that $\gamma_N$ may be explored by exploring the connected
	component of $P_N^*$ intersecting the complement of $\Lambda_N$,
	which reveals the heights in $\faces{\squarelattice}\setminus I(\gamma_N)$
	but not those in $I(\gamma_N)$.

	We rely on the following fundamental symmetry:
	if~$h\sim \muHom^0$, then the law of the height function $\tilde h$ defined by
	\begin{equation}
		\label{eq:shifteq}
		\tilde{h}(u):= 1-h(u-e_1)
	\end{equation}
	is $\muHom^1$.
	Define the coupling $\pi$ as follows (see Fig.~\ref{fig:T-circuits}).
	\begin{enumerate}
		\item Explore $\gamma_N$ for $h^0$ in the measure $\muHom^0$,
		revealing $h^0|_{\faces{\squarelattice}\setminus I(\gamma_N)}$.
		\item This means that the law of $h^0$ in $I(\gamma_N)$ conditional on the revealed
		heights is given exactly by the Gibbs specification induced by~\eqref{eq:6v_weight_def}.
		\item The height function $h^1$ on $\faces{\squarelattice}\setminus (I(\gamma_N)+e_1)$
		is defined through~\eqref{eq:shifteq}, that is,
		\[
			h^1(u):= 1-h^0(u-e_1)\qquad\forall u\in \faces{\squarelattice}\setminus (I(\gamma_N)+e_1).
		\]
		\item By the previous three statements and~\eqref{eq:shifteq}, if the
		values of $h^1$ on $I(\gamma_N)+e_1$ are sampled
		according to the Gibbs specification, then $h^1$ has the law of $\muHom^1$.
	\end{enumerate}
	At this stage, $h^0$ is defined on the complement of $I(\gamma_N)$,
	and $h^1$ is defined on the complement on $I(\gamma_N)+e_1$.
	Write $F(\gamma_N)$ for the set of faces visited by $\gamma_N$.
	Then the definitions of $\gamma_N$ and $h^1$ imply that:
	\begin{equation}
		\label{heightfdom}
		h^0|_{F(\gamma_N)}\geq 1
		;
		\qquad
		h^1|_{F(\gamma_N)+e_1}\leq 0.
	\end{equation}
	Now define $I:=I(\gamma_N)\cap (I(\gamma_N)+e_1)$,
	and write $\partial I\subset \faces{\squarelattice}\setminus I$ for the set of
	faces which share a vertex with a face in $I$ (such faces are adjacent or diagonally
	adjacent to $I$).
	Note the following key property: the definition of $\T$-circuits implies
	that $\partial I\subset F(\gamma_N)\cup (F(\gamma_N)+e_1)$.
	We now continue the construction of our coupling.
	\begin{enumerate}[resume]
		\item Explore the values of $h^0$ and $h^1$
		on $\faces{\squarelattice}\setminus I$ that are not yet known, by sampling them
		from the Gibbs distribution (independently for the two height functions).
		\item Since $\partial I\subset F(\gamma_N)\cup (F(\gamma_N)+e_1)$,
		we observe that~\eqref{heightfdom} and the fact that graph homomorphisms
		are $1$-Lipschitz implies
		\[
			h^0|_{\partial I}\geq h^1|_{\partial I}.
		\]
		\item Now the six-vertex model satisfies a Markov property over $\partial I$ and the conditional distributions of~$h^0$ and~$h^1$ satisfy the FKG inequality.
		By the comparison between boundary conditions, the conditional law of $h^0$ on $I$
		dominates the conditional law of $h^1$ on $I$.
		Thus, by Strassen's theorem, we may extend $\pi$ to a coupling
		such that $h^0|_I\geq h^1|_I$.
	\end{enumerate}
	Now $\pi$ satisfies~\eqref{condoncoupling}
	because
	\[
		\pi(h^0|_{\Lambda_n}\geq h^1|_{\Lambda_n})
		\geq
		\pi(\Lambda_n\subset I)
		\geq\muHom(\Lambda_{n+8}\subset I(\gamma_N))
		\geq 1-\epsilon.
	\]
	This finishes the proof.
\end{proof}

\subsection{Logarithmic delocalisation}
\begin{proof}[Proof of~(\ref{thm:hard_six_deloc:quant}) in Theorem~\ref{thm:hard_six_deloc}]
	Finally, we note that in the symmetric case~$a=b$, the argument of~\cite{GlaMan21} can be adapted to prove a logarithmic bound on delocalisation; see the proof of Theorem~\ref{thm:soft_Lip_deloc} for a sketch and Appendix~\ref{app:dicho} for more details.
	The main difference is that instead of duality for black and white {\em site} percolations on the triangular lattice in the loop~\O{2} model, we use duality for black and white {\em bond} percolations on the square lattice in the six-vertex model.
	There is however one additional subtlety coming from a different lattice geometry: in the six-vertex model, we do not have duality between crossings of~$\{\sigma^\black = +\}$ and~$\{\sigma^\black = -\}$ because they live at the vertices of the square lattice.
	This becomes crucial in the Step~$1$ of the RSW part in the proof of Theorem~\ref{thm:soft_Lip_deloc}: instead of~$(\Z^2)^\black$-crossings of~$\{\sigma^\black = +\}$ (that is, in the diagonal connectivity of~$\Z^2$), the duality guarantees only~$\T^\black$-crossings.
	We now explain why this is sufficient to run the Step~$2$ .
	We use the notation from the proof of Theorem~\ref{thm:soft_Lip_deloc}.

	Indeed, assume that the bottom and top parts of~$\Gamma_L$ and~$\Gamma_R$ are connected by pluses of~$\sigma^\black$ in~$\T^\black$-connectivity and explore the exterior-most such crossings.
	Then, all even vertices on the exterior boundary of~$\chi_1$ and~$\chi_2$ either have minuses of~$\sigma^\black$ or are located outside of the strip.
	We set~$\sigma^\black$ to be minus at all these vertices.
	By the FKG inequality, this can only decrease the measure.
	It is easy to see that the edges of~$(\Z^2)^\white$ separating~$\chi_1$ from these minuses at the exterior boundary form a path that belongs to~$\xi^\white$.
	Denote this path by~$\overline{\chi}_1$ and define the path~$\overline{\chi}_2$ similarly.
	Note that if~$\overline{\chi}_1\subseteq \xi^\whitep$ and~$\overline{\chi}_2\subseteq \xi^\whitem$ (or vice versa), then they are not connected to each other in~$\xi^\white$ and hence the super-duality~\eqref{eq:super_duality-6v} implies that~$\Gamma_L$ and~$\Gamma_R$ are connected by a path of~$\xi^\blackp$, which is the required statement.

	Thus, by symmetry, we can assume that~$\overline{\chi}_1, \overline{\chi}_2\subseteq \xi^\whitep$.
	Finally, we forget about the pluses of~$\sigma^\black$ at~$\chi_1$ and~$\chi_2$ (again, by the FKG inequality, this can only stochastically decrease the measure).
	Now, similarly to~\cite[Fig.~$13$]{GlaMan21}, we use reflections and push away the~$\xi^\blackp$-boundary conditions (partially replacing them by~$\xi^\whitep$ conditions) to obtain a planar symmetric domain~$\calD$ with boundary points~$A,B,C,D$ such that: the arcs~$AB$ and~$CD$ belong to~$\xi^\blackp$; the arcs~$BC$ and~$DA$ (containing~$\overline{\chi}_1$ and~$\overline{\chi}_2$ respectively) belong to~$\xi^\whitep$.
	Since the boundary conditions on~$\calD$ are symmetric between~$\xi^\black$ and~$\xi^\white$, a~$\xi^\blackp$ crossing from~$(AB)$ to~$(CD)$ exists with probability~$1/2$.
	This crossing always contains a crossing from~$\Gamma_L$ to~$\Gamma_R$.
	Noting that all the operations described above could only decrease the probability of such event, we get the desired lower bound on its probability in the original measure.
	This finishes the proof of the RSW part and everything else adapts from the loop~\O{2} to the six-vertex model in a straightforward fashion.
\end{proof}

\newcommand\cleft{\operatorname{left}}
\newcommand\cright{\operatorname{right}}
\newcommand\fZ{\mathfrak{Z}}
\newcommand\Ln{L_{\operatorname{non}}}
\newcommand\Lc{L_{\operatorname{contr}}}
\newcommand\phitorus{\phi^{\operatorname{tor}}}

\section{Continuity of the phase transition in the random-cluster model}
\label{sec:continuity}

We shall take a straightforward route to deriving continuity of the phase
transition for $q\in[1,4]$ from the delocalisation result:
we study a single observable in the Baxter--Kelland--Wu (BKW) coupling~\cite{BaxKelWu76}
of the type appearing in the work of Dubédat~\cite[p.~398]{Dub11}.
The handling of boundary conditions is important and nontrivial in the application of the BKW
coupling. We choose to work on the torus, following closely the setting of Lis~\cite{Lis21}.
To this end, we quotient all our full-plane graphs by $2n\Z^2$
for any integer $n\in\Z_{\geq 1}$;
we define
\[
    \squarelattice_n:=\squarelattice/   2n\Z^2,
\]
and define $\squarelattice^\black$ and $\squarelattice^\white$ in a similar fashion.
The interest is in the limit~$n\to\infty$.

The proofs in this section are restricted to the symmetric setting of Theorem~\ref{thm:continuity_asymetric}: $p_a=p_b=p$.
The proofs readily extend to the asymmetric case $p_a\neq p_b$ after updating the weights
appearing in the BKW coupling. These weights are most naturally obtained by applying a linear map
to all our lattices, then relating the new weights to the lengths and angles of the different
edges in these transformed lattices; see~\cite[Section~5.1]{BaxKelWu76} for details.

\subsection{The Baxter--Kelland--Wu coupling}

We define an \emph{oriented loop} as a closed self-avoiding walk that consists
of alternating vertical and horizontal edges of~$\squarelattice_n$.
We do not care about the starting point of each loop, that is, we identify loops
which are equal up to a cyclic permutation of the edges.
An \emph{oriented loop configuration} is a set of oriented loops
such that each edge in~$E(\squarelattice_n)$ appears in exactly one oriented loop (Fig.~\ref{fig:BKWcomplete}).
  Note that no oriented loop can
cross itself or other loops because subsequent edges make a $\pi/2$ angle at
their intersection point.
Write $\calL_n^\rightarrow$ for the set of oriented loop configurations.

For any walk $p$ on $\squarelattice_n^*:=(F(\squarelattice_n),E^*(\squarelattice_n))$
and $L^\rightarrow\in\calL_n^\rightarrow$,
let $\int_pL^\rightarrow$ denote the number of times the path $p$
is crossed by a loop of~$L^\rightarrow$ from right to left minus the number of times the path
is crossed by a loop of~$L^\rightarrow$ from left to right.
Write $H_8(L^\rightarrow)$ for the indicator of the event that
$\int_pL^\rightarrow\in 8\Z$, for {\em any closed} walk $p$
on the dual torus $\squarelattice_n^*$.
We also let $p^k$ denote the walk which starts at the black square at the origin,
and then makes $2k$ steps to the right (in particular, it ends at a black square).

Let $\lambda\in[0,\frac\pi3]$. The \emph{weight} of any oriented loop configuration $L^\rightarrow\in\calL_n^\rightarrow$
is defined through
\begin{equation}
    \label{orientedweight}
    w(L^\rightarrow):=e^{i\frac\lambda4(\cleft(L^\rightarrow)-\cright(L^\rightarrow))}H_8(L^\rightarrow),
\end{equation}
where $\cleft(L^\rightarrow)$ and $\cright(L^\rightarrow)$ denote the number of left-
and right turns of loops in $L^\rightarrow$ respectively.
The indicator $H_8(L^\rightarrow)$ appears for technical reasons and will help us later
when extending the FKG inequality to spin measures on the torus.

The following two quantities are at the centre of our analysis:
\begin{equation}
    \label{eq:crazyobservables}
    \fZ_n:=\sum_{L^\rightarrow\in\calL^\rightarrow_n}w(L^\rightarrow);
    \qquad
    \fZ_{n,k}:=
    \sum_{L^\rightarrow\in\calL^\rightarrow_n}w(L^\rightarrow)e^{i\frac\pi8\int_{p^k}L^\rightarrow}.
\end{equation}
The idea is now as follows: if we forget the loop structure of each oriented loop configuration $L^\rightarrow$
but not the orientations of the edges, then we obtain a six-vertex configuration;
if we forget the orientation of each loop but not the loop structure,
then we obtain a random-cluster percolation configuration.
This allows one to express $\fZ_n$ and $\fZ_{n,k}$ in terms of both models,
leading to an identity between the two.
In order to keep the exposition as simple as possible, we only equate the two
in a certain limit, which enables us to separate entirely the discussion of the
six-vertex model and that of the random-cluster model in relation to~\eqref{eq:crazyobservables}.
For further details on the BKW coupling, we refer to~\cite{Dub11}.

\begin{figure}
	\includegraphics[width=\textwidth]{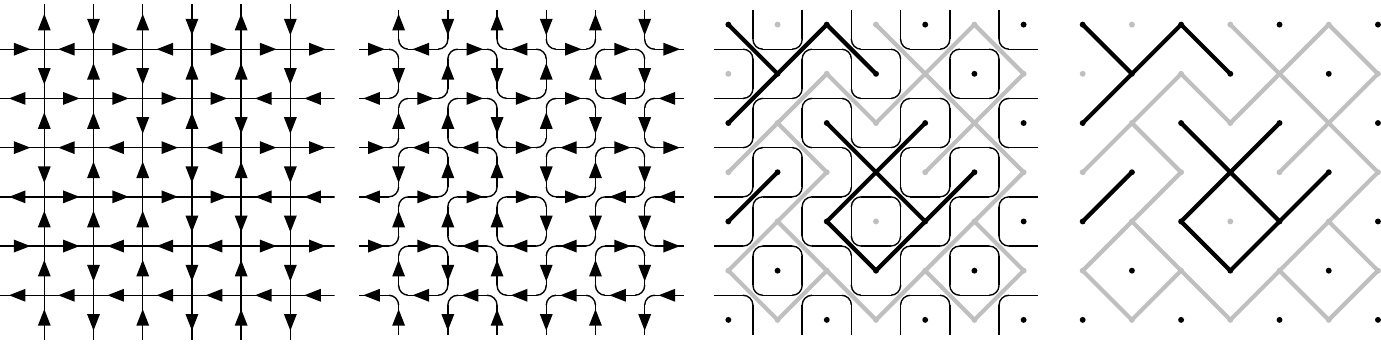}
	\caption{
        The BKW coupling.
        The six-vertex configuration is obtained by forgetting the connectivity
        structure, keeping the orientations of the edges.
        The random-cluster configuration is obtained by interpreting the loops
        as interfaces between primal and dual clusters.
	}
	\label{fig:BKWcomplete}
\end{figure}

\subsection{The six-vertex model}

This subsection establishes the following result.

\begin{lemma}
    \label{lemma:contlemma1}
    For any $\lambda\in[0,\frac\pi3]$, we have
    \[
        \lim_{k\to\infty}\lim_{n\to\infty}\frac{\fZ_{n,k}}{\fZ_n}=0.
    \]
\end{lemma}

\begin{proof}
    Fix $\lambda$. Take $a=b=1$ and~$c=2\cos\frac\lambda2 \in[1,2]$.
    The proof has three steps.
    \begin{enumerate}
        \item\label{RATIO_ONE} The ratio~$\tfrac{\fZ_{n,k}}{\fZ_n}$ equals
        an observable of some spin measure on~$\squarelattice_n$.
        \item\label{RATIO_TWO} As $n\to\infty$, this spin measure converges weakly
        to the measure~$\muSpin^{\blackp}$ studied extensively in Section~\ref{sec:six-vertex}.
        The observable behaves well in this limit as it is local and uniformly bounded.
        \item\label{RATIO_THREE} Delocalisation of~$\muHom^0$ implies that the observable
       vanishes as $k\to\infty$.
    \end{enumerate}

    \subsubsection*{Step~\ref{RATIO_ONE}}
    Fix $n$.
    We consider spin pairs $(\sigma^\black,\sigma^\white)$
    with $\sigma^\black\in\{\pm1\}^{V(\squarelattice^\black_n)}$
    and $\sigma^\white\in\{\pm1\}^{V(\squarelattice^\white_n)}$.
    Define consistency as in Section~\ref{sec:spin-rep-sixv} and
    write $\spaceSpin(\squarelattice_n)$ for the set of consistent spin pairs.
    Each pair $(\sigma^\black,\sigma^\white)\in\spaceSpin(\squarelattice_n)$ encodes an integer height modulo four.
    Define the gradient function~$h_\nabla=h_\nabla^{\sigma^\black\sigma^\white}$ on ordered pairs of adjacent faces such that, for any adjacent~$u\in V(\squarelattice^\black_n)$ and~$v\in V(\squarelattice^\white_n)$,
    \[
    	h_\nabla(v,u) = - h_\nabla(u,v) = (-1)^{\sigma^\black(u)\sigma^\white(v)}.
    \]
    For any walk $p=(p_k)_{0\leq k\leq m}$ on $\squarelattice_n^*$,
    define $\int_ph_\nabla:=\sum_{k=1}^m h_\nabla(p_{k-1},p_k)$.
    Note that
    \begin{itemize}
        \item If $p$ is a closed contractible walk on the torus, then $\int_p h_\nabla=0$;
        \item If $p$ is any closed walk on the torus, then $\int_p h_\nabla\in 4\Z$.
    \end{itemize}
    The latter holds true because the spins define the height function modulo four (see Definition~\ref{def:heights-to-spins-6v}).
    Write $H_8(\sigma^\black,\sigma^\white)$ for the indicator of the event that
    $\int_p h_\nabla\in 8\Z$, for any closed walk $p$ on~$\squarelattice_n$.
    Finally, define $\muSpin_n$ as the following probability measure on $\spaceSpin(\squarelattice_n)$:
    \begin{equation}
        \label{eq:torusspindef}
        \muSpin_n(\sigma^\black,\sigma^\white):=\tfrac1{Z_n}\cdot (\tfrac1c)^{|\omega[\sigma^\black]\cup\omega[\sigma^\white]|}H_8(\sigma^\black,\sigma^\white),
    \end{equation}
    where $Z_n$ is the partition function.
    We claim that
    \begin{equation}
        \label{eq:BKWsixvertexFINITEdomain}
        \frac{\fZ_{n,k}}{\fZ_n}=\muSpin_n[e^{i\frac\pi8\int_{p^k}h_\nabla}].
    \end{equation}
    The proof of the claim is standard in the context of the BKW coupling.
    First, observe that each oriented loop configuration
    $L^\rightarrow\in\calL^\rightarrow$ induces an edge orientation
    on~$\squarelattice_n$ (Fig.~\ref{fig:BKWcomplete}).  Clearly, $\int_{p^k}L^\rightarrow$ is a function of
    this six-vertex configuration and does not depend on a particular loop
    structure.  Summing over all oriented loop configurations inducing this
    six-vertex configuration, one deduces from~\eqref{orientedweight} that
    vertices of types $a$, $b$, and $c$ receive weights $1$, $1$, and
    $c=e^{i\lambda/2}+e^{-i\lambda/2}$ respectively.  Each six-vertex
    configuration determines the pair~$(\sigma^\black,\sigma^\white)$ up to a
    global spin flip.  The weight~\eqref{eq:torusspindef} of this pair does not
    depend on this global spin flip and is given by assigning the weights $1$,
    $1$, and $c$ to vertices of types $a$, $b$, and $c$ respectively.  Moreover,
    the path integrals of~$(\sigma^\black,\sigma^\white)$ and of the oriented
    loop representation agree, which yields~\eqref{eq:BKWsixvertexFINITEdomain}.

    \subsubsection*{Step~\ref{RATIO_TWO}}
    We aim to prove that $\muSpin_n\to\muSpin^\blackp$.
    Before proving this, observe that~\eqref{eq:torusspindef} implies that $\muSpin_n$ is invariant under resampling
    any collection of spins according the Gibbs probability kernels
    \emph{conditional on remaining in the same homotopy class}.
    Define a \emph{contractible domain} as a domain that is a strict subset of the torus
    and whose boundary cycle is contractible on the torus.
    The homotopy class of a spin configuration cannot change
    by changing the spin values on a contractible domain.
    Therefore, $\muSpin_n$ is invariant under resampling its spins on a contractible
    domain according to the \emph{unconditioned} Gibbs probability kernel on spins on that domain.

    We enhance the measure $\muSpin_n$ with the vertex percolation $\xi^\black$ as in Definition~\ref{def:b-w-perco-6v}.
    The domain Markov property (Lemma~\ref{lemma:markov-6v}) readily extends to the case when~$\domain'=\squarelattice_n$ and when~$\domain\subsetneq \squarelattice_n$ is contractible.
    We now claim that the FKG inequality for triplets~$(\sigma^\black,\xi^\blackp,-\xi^\blackm)$ (Proposition~\ref{prop:fkg-joint-six}) also holds for~$\muSpin_n$.
    We first derive $\lim_{n\to\infty}\muSpin_n=\muSpin^\blackp$ from this claim.
    Indeed, by the FKG inequality and the domain Markov property, for any contractible domain $\domain$
    and any large enough integer $n$, we have
    \[
        \muSpin_\domain^\blackm\preceq_\black\muSpin_n|_\domain \preceq_\black \muSpin_\domain^\blackp.
    \]
    Hence, any subsequential limit~$\mu$ of $(\muSpin_n)_n$ is also sandwiched between the two measures.
    Letting $\domain\nearrow\squarelattice$, we obtain
    \[
        \muSpin^\blackm     \preceq_\black \mu \preceq_\black  \muSpin^\blackp.
    \]
    In the proof of Theorem~\ref{thm:hard_six_deloc}, we show that~$\muSpin^\blackm = \muSpin^\blackp$ and this measure almost surely exhibits infinitely many alternating circuits of~$\xi^\blackp$ and~$\xi^\blackm$.
    Then, the marginal of~$\mu$ on~$(\sigma^\black,\xi^\black)$ coincides with that of~$\muSpin^\blackp$.
    By Proposition~\ref{prop:inf-vol-6v}, this implies~$\mu=\muSpin^\blackp$.
    Since the subsequential limit was arbitrary, tightness implies the desired weak convergence.

    It remains to prove that the triplet~$(\sigma^\black,\xi^\blackp,-\xi^\blackm)$ sampled from~$\muSpin_n$ satisfies the FKG inequality.
    Note that it is enough to establish the FKG lattice condition for~$\sigma^\black$ (analogue of Lemma~\ref{lemma:fkg-spins-lip}), since then the joint FKG property can be derived in the same way as in Proposition~\ref{prop:fkg-joint-lip}.
    Compared to Lemma~\ref{lemma:fkg-spins-lip}, there are two important differences:
    first, the law of~$\sigma^\white$ given~$\sigma^\black$ has additional restrictions and is not an Ising model (due to~$H_8$);
    second, the torus is not simply-connected, which becomes relevant in~\eqref{eq:fkg-proof-ising-decomp-pm}.
    We argue that the two effects cancel out, namely that~\eqref{eq:fkg-proof-measure-via-ising}
    and~\eqref{eq:fkg-proof-ising-decomp-pm} jointly turn into
    \begin{equation}
        \label{eq:fkgtorusdecomp}
        Z_n\cdot\muSpin_n(\sigma^\black)=\tfrac12\cdot(\tfrac1c)^{|\omega[\sigma^\black]|}\cdot
        \Zising(a^+(\sigma^\black))\cdot\Zising(a^-(\sigma^\black)),
    \end{equation}
    where $a^\pm(\sigma^\black)$ and~$\Zising(a)$ are defined similarly to the proof of Lemma~\ref{lemma:fkg-spins-lip}:
    \begin{align*}
    	a^\pm(\sigma^\black)_{uv}:=&\begin{cases}
            1/c &\text{if $\sigma^\black$ equals $\pm$ on both endpoints of $uv^*$,}\\
            0&\text{otherwise,}
        \end{cases}\\
        \Zising(a):=&\sum_{\sigma^\white\in\{+,-\}^{V(\squarelattice^\white_n)}}
        \prod_{uv\in E(\squarelattice_n^\white)}
		w_{uv}(a);
        \qquad
        w_{uv}(a):=\begin{cases}
            1 & \text{if $\sigma_u^\white=\sigma_v^\white$,}\\
            a_{uv}&\text{otherwise.}
        \end{cases}
    \end{align*}
    We focus on proving~\eqref{eq:fkgtorusdecomp}, as the FKG lattice condition for~$\sigma^\black$ follows from it readily.

    We fix $\sigma^\black$ and reason as in the proof of Lemma~\ref{lemma:fkg-spins-lip}, using one extra piece of information.
    Namely, for any $\sigma^\white\perp\sigma^\black$, the following two are equivalent:
    \begin{enumerate}
        \item $H_8(\sigma^\black,\sigma^\white)=1$,
        \item $\sigma^\white$ may be written as the product $\zeta^+\cdot\zeta^-$
        where $\zeta^+,\zeta^-\in\{+,-\}^{V(\squarelattice^\white_n)}$
        are chosen such that $\sigma^\black$ equals $+$ at all endpoints of edges
        in $\omega[\zeta^+]$ and $-$ at all endpoints of edges in~$\omega[\zeta^-]$
        (see also the proof of~\eqref{eq:fkg-proof-ising-decomp-pm}).
    \end{enumerate}
    By the definition of $H_8$, the first statement is equivalent to $h_\nabla$
    having a lift to a well-defined height function modulo $8$ whose height difference
    on adjacent squares is $\pm1$.
    We now prove the latter is equivalent to the second statement.
    For a given height function $h'$ modulo $8$,
    the corresponding functions $\sigma^\black$, $\zeta^+$, and~$\zeta^-$ are given by
    \begin{equation}
\label{eq:triplespin}
\begin{split}
        &\{\sigma^\black=+\}=\{h'\in\{8\Z,4+8\Z\}\};
        \\
        &\{\zeta^+=+\}=\{h'\in\{1+8\Z,3+8\Z\}\}
        ;\quad
        \{\zeta^-=+\}=\{h'\in\{1+8\Z,7+8\Z\}\}.
\end{split}
    \end{equation}
    It is straightforward to check that~$(\sigma^\black,\zeta^+\zeta^-)$ is a spin representation of~$h'$ and hence satisfies the ice rule.
    Conversely, given $\sigma^\black$, $\zeta^+$, and~$\zeta^-$,
    we define a height function~$h'$ modulo~$8$ as follows:
    \begin{itemize}
        \item At white squares, it is uniquely defined by $\zeta^+$, $\zeta^-$, and~\eqref{eq:triplespin};
        \item At any black square $u$, the height~$h'$ modulo~$8$ is defined by
        \begin{equation}\label{eq:h-mod-8-via-spins}
            h'(u):=
            \begin{cases}
                8\Z &\text{if $\sigma^\black(u)=+$ and $\zeta^-(v)=+$ at all~$v\sim u$,}\\
                2+8\Z &\text{if $\sigma^\black(u)=-$ and $\zeta^+(v)=+$ at all~$v\sim u$,}\\
                4+8\Z &\text{if $\sigma^\black(u)=+$ and $\zeta^-(v)=-$ at all~$v\sim u$,}\\
                6+8\Z &\text{if $\sigma^\black(u)=-$ and $\zeta^+(v)=-$ at all~$v\sim u$.}
                % 2+8\Z &\text{if $\sigma^\black(u)=-$ and $u$ is adjacent to some $v$ where $(\zeta^+(v),\zeta^-(v))=(+,+)$,}
            \end{cases}
        \end{equation}
    \end{itemize}
    One of these four cases must occur because $\zeta^-$ must be constant
    on the four white squares adjacent to $u$ if $\sigma^\black(u)=+$,
    and $\zeta^+$ must be constant on the same set if $\sigma^\black(u)=-$.
    It is also easy to see that at most one of
    the cases in~\eqref{eq:h-mod-8-via-spins} occurs, and therefore~$h'$ is
    well-defined.  Finally, \eqref{eq:triplespin}
    and~\eqref{eq:h-mod-8-via-spins} are consistent, in a sense that~$h'$
    differs at neighbouring faces by~$\pm 1$ modulo~$8$.

    \subsubsection*{Step~\ref{RATIO_THREE}}
    We have now established that
    \begin{equation}
        \label{eq:limitobservable}
        \lim_{n\to\infty}\frac{\fZ_{n,k}}{\fZ_n}=\muSpin^\blackp[e^{i\frac\pi8\int_{p^k}h_\nabla}].
    \end{equation}
    It suffices to demonstrate that the right-hand side tends to zero as $k\to\infty$.
    In the proof of Theorem~\ref{thm:hard_six_deloc}, we showed that the origin is surrounded by infinitely many $\xi^\black$-circuits
    and infinitely many $\xi^\white$-circuits (see Proposition~\ref{prop:ergo-6v} and the observation that
    $\muSpin^\blackp=\muSpin^\whitep$).
    We now use the same exploration procedure as in that proof: explore the outermost
    $\xi^\black$-circuit in $[-2k,2k]^2$, then the outermost~$\xi^\white$-circuit within that black circuit, and so forth.
    Now suppose that, each time we encounter a $\xi^\black$-circuit (resp.~$\xi^\white$-circuit),
    we flip a fair coin to decide if we invert the values of the spins~$\sigma^\white$ (resp.~$\sigma^\black$)
    within that circuit.
    Then, $\int_{p^k}h_\nabla$ becomes a simple random walk whose (always even) length is given by the number of circuits
    discovered before reaching the black square at $(0,0)$.
    The length of this random walk tends to infinity as $k\to\infty$
    by~\eqref{eq:largebothcircuits-6v},
    which implies that $\int_{p^k}h_\nabla+16\Z$ converges to
    the uniform distribution on the $8$ even residues modulo $16$
    as $k\to\infty$.
    Therefore, the right-hand side of~\eqref{eq:limitobservable}
    tends to zero.
\end{proof}

\subsection{The random-cluster model}
\label{subsec:continuityproof:rcm}

\begin{lemma}
    \label{lemma:contlemma2}
    Fix $q\in[1,4]$ and choose $\lambda\in[0,\frac\pi3]$ such that $\sqrt q = 2\cos \lambda$.
    Then,
    \[
        \lim_{k\to\infty}\lim_{n\to\infty}\frac{\fZ_{n,k}}{\fZ_n}=0
    \]
    implies that $\phifree_{\pselfdual(q),\pselfdual(q),q}=\phiwired_{\pselfdual(q),\pselfdual(q),q}$,
    and that neither $\eta$ nor its dual percolates.
\end{lemma}
\begin{proof}
    We drop the subscripts $\pselfdual(q)$ and $q$ for brevity.
    The proof has three steps.
    \begin{enumerate}
        \item\label{FK_ONE} The ratio~$\frac{\fZ_{n,k}}{\fZ_n}$ equals
        an observable of some random-cluster measure
        on~$\squarelattice_n$.
        \item\label{FK_TWO} As $n\to\infty$, any subsequential limit~$\phitorus$ of this measure in the weak topology is a self-dual shift-invariant Gibbs measure  of the random-cluster model with parameters $(\pselfdual(q),q)$.
        \item\label{FK_THREE} The limit in~$n$ of the observable is greater or equal than~$\phitorus((0,0)\leftrightarrow (2k,0))$ times some constant~$c_\lambda >0$.
        By our assumption, taking now the limit as~$k\to\infty$ yields zero, and therefore~$\phitorus$ exhibits no infinite cluster. Thus~$\phitorus = \phifree$ and, hence, by self-duality, we also obtain that $\phitorus=\phiwired$.
    \end{enumerate}

    \subsubsection*{Step~\ref{FK_ONE}}
    A \emph{loop configuration} on~$\squarelattice_n$ is a partition of its edges into disjoint cycles of alternating vertical and horizontal edges (Fig.~\ref{fig:BKWcomplete}).
    Write $\calL_n$ for the set of loop configurations on $\squarelattice_n$.
    We may calculate $\fZ_n$ by first summing over $L\in\calL_n$ and
    then over all \emph{oriented} loop configurations configurations corresponding to~$L$.
    Given $L$, there are two ways to orient each loop.
    Each loop~$\ell$ contributes to the weight~\eqref{orientedweight} as follows:
    \begin{itemize}
        \item If~$\ell$ is contractible, then it contributes $e^{i\lambda}$ or $e^{-i\lambda}$
        depending on its orientation, and it does not affect~$H_8$,
        \item If~$\ell$ is non-contractible then it makes as many left turns as it makes right turns, and hence contributes~$1$ to the exponential,
        while the indicator $H_8$ tests if the ($\pm1$-valued) orientations
        sum to a number in $8\Z$.
    \end{itemize}
    Taking into account the above and recalling that $\sqrt q=e^{i\lambda}+e^{-i\lambda}$, we obtain
    \begin{equation}
        \label{eq:alternativeFZn}
        \fZ_n
        =
        \sum_{L\in\calL_n}
        w'(L),
        \qquad
        \text{where}
        \qquad
        w'(L):=
        (\sqrt q)^{|\Lc|}
        2^{|\Ln|}
        p_8(|\Ln|),
    \end{equation}
    and
    where
    \begin{itemize}
        \item $\Lc\subset L$ denotes the set of contractible loops,
        \item $\Ln:=L\setminus \Lc$ denotes the set of non-contractible loops,
        \item $p_8(m)$ denotes the probability that an $m$-step simple
        random walk ends in $8\Z$.
    \end{itemize}
    For any $\ell^\rightarrow$, write
    $\int_{p^k}\ell^\rightarrow$ for the number of
    times $p^k$ is crossed from right to left by $\ell^\rightarrow$ minus the number of times
    it is crossed from left to right.
    For any $u\in\Z^2/(2n\Z)^2$,
    write $u\triangleleft\ell$, for some contractible unoriented loop $\ell$
    on $\squarelattice_n$ whenever $\ell$ surrounds $u$
    (this notion is well-defined, see~\cite[Chapter~2]{rolfsen2003knots}).
    To obtain a similar identity for $\fZ_{n,k}$,
    observe that the contribution of each loop
    to the extra exponential in the definition on the right in~\eqref{eq:crazyobservables}
    is precisely
    \[
        e^{i\frac\pi8\int_{p^k}\ell^\rightarrow}.
    \]
    Note that if $\ell^\rightarrow$ is contractible
    and oriented counterclockwise, then
    \[
        \int_{p^k}\ell^\rightarrow=\ind{(0,0)\triangleleft\ell^\rightarrow}    -\ind{(2k,0)\triangleleft\ell^\rightarrow}.
    \]
    By summing over $L\in\calL_n$ and taking into account these extra factors,
    we get
    \begin{equation}
        \label{eq:alternativeFZnk}
        \fZ_{n,k}
        =
        \sum_{L\in\calL_n}
        w'(L)
        \left[\prod_{\ell\in \Lc}
            \frac{\cos(\lambda+\frac\pi8(\ind{(0,0)\triangleleft\ell}    -\ind{(2k,0)\triangleleft\ell}))}{\cos\lambda}
        \right]
        E(\Ln),
    \end{equation}
    where $E(\Ln)$ is the expectation of
    \[
        e^{i\frac\pi8\sum_{\ell^\rightarrow\in \Ln^\rightarrow}\int_{p^k}\ell^\rightarrow}
    \]
    when the loops in $\Ln$ are oriented uniformly at random conditional on the
    sum of their orientations lying in $8\Z$.
    We will only use that~$|E(\Ln)|\leq 1$.

    Let $\phitorus_n$ be the probability measure on $\calL_n$ defined by
    \[
        \phitorus_n(L)\propto w'(L).
    \]
    By~\eqref{eq:alternativeFZnk}, we can express~$\frac{\fZ_{n,k}}{\fZ_n}$ as the following observable for~$\phitorus_n$:
    \begin{equation}
        \label{eq:eqintermediateblabla}
        \frac{\fZ_{n,k}}{\fZ_n}=\phitorus_n\left[E(\Ln)\prod\nolimits_{\ell\in \Lc}
        \frac{\cos(\lambda+\frac\pi8(\ind{(0,0)\triangleleft\ell}    -\ind{(2k,0)\triangleleft\ell}))}{\cos\lambda}\right].
    \end{equation}

    \subsubsection*{Step~\ref{FK_TWO}}
    We now identify each $L\in\calL$ with the set of
    edges $\eta\in\{0,1\}^{E(\squarelattice^\black_n)}$
    such that~$L$ consists of the interfaces between~$\eta$ and~$\eta^*$,
    see Fig.~\ref{fig:BKWcomplete}.
    Thus, $\phitorus_n$ can be viewed as a probability measure on bond percolations on~$\squarelattice_n^\black$.
    Clearly, the sequence~$(\phitorus_n)_n$ is tight in the weak topology and thus has subsequential limits; let~$\phitorus$ be one of them.

    Our next goal is to prove that~$\phitorus$ is a shift-invariant self-dual Gibbs measure for the random-cluster model at~$(\pselfdual(q),q)$.
    Indeed, this statement without the weight $p_8(|\Ln|)$ in~\eqref{eq:alternativeFZn} is proved in~\cite[Proof of Lemma~3.1]{Lis21} via the classical Burton--Keane argument.
    The argument readily extends to our measure (with the factor~$p_8(|\Ln|)$ included)
    by checking the finite energy condition.
    The finite energy condition is immediate since $|\Ln|\in 2\Z_{\geq 0}$ and $\frac14\leq p_8(a)\leq 1$ for all $a\in2\Z_{\geq 0}$.
    Thus, $\phitorus$ is indeed a Gibbs measure.
    Since each $\phitorus_n$ is self-dual and invariant to translations, so is~$\phitorus$.
    % But then $\phitorus$ must equal $\phi=\frac12\phifree+\frac12\phiwired$.

    The Burton--Keane argument applies to $\phitorus$,
    which means that both $\eta$ and $\eta^*$ exhibit at most one infinite cluster.
    In particular, there is at most one infinite interface almost surely.
    We write $\Lc$ for the set of finite interfaces (loops) in the limit.
    % The non-coexistence theorem applies to $\phifree$ and $\phiwired$,
    % which means that $\phi$-almost surely all $L=\Lc$.
    We will now show that
    \begin{equation}
        \label{eq:next}
        \lim_{n\to\infty}
        \frac{\fZ_{n,k}}{\fZ_n}
        \geq
        c_\lambda\cdot
        \phitorus\left[\prod\nolimits_{\ell\in\Lc}
        \frac{\cos(\lambda+\frac\pi8(\ind{(0,0)\triangleleft\ell}-\ind{(2k,0)\triangleleft\ell}))}{\cos\lambda}\right],
    \end{equation}
    where~$c_\lambda:= \tfrac{\cos \lambda +\tfrac\pi8}{\cos \lambda}>0$ is a constant.

    Understanding the asymptotic behaviour of the observable on the right in~\eqref{eq:eqintermediateblabla}
    is nontrivial because it is not local.
    To prove that the inferior limit of the right-hand side of~\eqref{eq:eqintermediateblabla}
    bounds the right-hand side of~\eqref{eq:next},
    we proceed as follows.
    For any~$r\geq 1$, denote the ball of radius~$r$ around~$(0,0)$ by~$B_r$.
    Let~$\Lc^r$ denote the set of contractible loops contained
    in~$B_r$ and let $A_r$ denote the event that the following holds (see Fig.~\ref{excursion}):
    \begin{itemize}
        \item At most one loop in $L\setminus \Lc^r$ intersects $p^k$;
        \item If such a loop~$\ell$ exists, then all intersections of~$\ell$ with~$p^k$ belong to the same connected component of~$\ell\cap B_r$.
    \end{itemize}
    The event~$A_r$ is local. Make the following observations.
     \begin{itemize}
        \item The integrands in~\eqref{eq:eqintermediateblabla} and~\eqref{eq:next}
        are uniformly bounded because the path $p^k$ can be intersected at most~$2k$ times.
        \item The integrands are also nonnegative because $\lambda\in[0,\frac\pi3]$.
        \item We have $\lim_{r\to\infty}\phitorus(A_r)=1$ by the Burton--Keane argument.
     \end{itemize}
     We end the proof of~\eqref{eq:next} by a series of equations (justified below):
     \begin{align}
        &\lim_{n\to\infty}
        \phitorus_n\left[E(\Ln)\prod\nolimits_{\ell\in \Lc}
        \frac{\cos(\lambda+\frac\pi8(\ind{(0,0)\triangleleft\ell}    -\ind{(2k,0)\triangleleft\ell}))}{\cos\lambda}\right]
        \\&
        \qquad=
        \lim_{r\to\infty}
        \lim_{n\to\infty}
        \phitorus_n\left[E(\Ln)\prod\nolimits_{\ell\in \Lc}
        \frac{\cos(\lambda+\frac\pi8(\ind{(0,0)\triangleleft\ell}    -\ind{(2k,0)\triangleleft\ell}))}{\cos\lambda}1_{A_r}\right]
        \\&\qquad\geq
        c_\lambda\cdot
        \lim_{r\to\infty}
        \lim_{n\to\infty}
        \phitorus_n\left[\prod\nolimits_{\ell\in \Lc^r}
        \frac{\cos(\lambda+\frac\pi8(\ind{(0,0)\triangleleft\ell}    -\ind{(2k,0)\triangleleft\ell}))}{\cos\lambda}1_{A_r}\right]
        \\&
        \qquad=c_\lambda\cdot
        \lim_{r\to\infty}
        \phitorus\left[\prod\nolimits_{\ell\in \Lc^r}
        \frac{\cos(\lambda+\frac\pi8(\ind{(0,0)\triangleleft\ell}    -\ind{(2k,0)\triangleleft\ell}))}{\cos\lambda}1_{A_r}\right]
        \\&
        \qquad =c_\lambda\cdot
        \phitorus\left[\prod\nolimits_{\ell\in \Lc}
        \frac{\cos(\lambda+\frac\pi8(\ind{(0,0)\triangleleft\ell}    -\ind{(2k,0)\triangleleft\ell}))}{\cos\lambda}\right].
     \end{align}
     Indeed, the first equality holds true because $A_r$ is local, $\phitorus(A_r)$ tends to~$1$ as $r\to\infty$,
     and the integrand is uniformly bounded.
     The inequality follows since all factors are non-negative and, on the event~$A_r$, the unique loop~$\ell \not\in \Lc^r$ intersecting~$p_k$ (if such~$\ell$ exists) satisfies $|\int_{p^k}\ell^\rightarrow|\leq 1$ and, hence, contributes at least
     \[
        c_\lambda\leq\min_{k\in\{-1,0,1\}}\frac{\cos (k\lambda+\frac\pi8)}{\cos (k\lambda)}.
     \]
     The second equality holds since the integrand is local and uniformly bounded,
     so that weak convergence of the measure implies convergence of the integral.
     The last equality follows from the dominated convergence theorem since~$\phitorus(A_r)$ tends to~$1$.

\begin{figure}
    \includegraphics{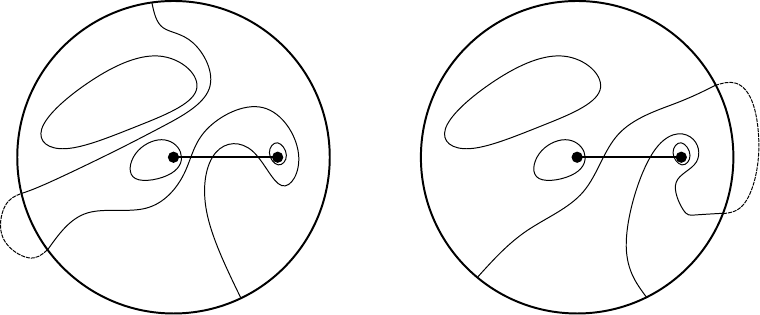}
    \caption{\textsc{Left}: A loop configuration belonging to $A_r$.
    The large loop cannot have a net intersection count other than $-1$, $0$,
    or $1$ with the line segment $p^k$ due to planarity.
    \textsc{Right}: A loop configuration not belonging to $A_r$.
    In this case, the large loop leaves the ball in between visits to $p^k$.}
    \label{excursion}
\end{figure}

    \subsubsection*{Step~\ref{FK_THREE}}
    The integrand in~\eqref{eq:next} is nonnegative,
    since $\lambda\in[0,\frac\pi3]$.
    Moreover, if $(0,0)$ and $(2k,0)$ are $\eta$-connected,
    then no loop makes a nontrivial net intersection with~$p_k$, and therefore the integrand equals $1$.
    Therefore, we get
    \[
        c_\lambda\cdot\phitorus((0,0)\leftrightarrow(2k,0))\leq\lim_{n\to\infty}
        \frac{\fZ_{n,k}}{\fZ_n}.
    \]
    In particular,
    \[
        \lim_{k\to\infty}\phitorus((0,0)\leftrightarrow(2k,0)) =0.
    \]
    This means that~$\phitorus$ does not exhibit an infinite cluster,
    and thus $\phitorus=\phifree$.
    By self-duality, we also get that $\eta^*$ does not percolate,
    and therefore~$\phitorus=\phiwired$.
\end{proof}

\appendix

\section{Logarithmic bound for the variance.}
\label{app:dicho}

As described in Section~\ref{subsec:proof_of_thm:soft_Lip_deloc}, the proof of~\cite{GlaMan21} (uniform case) can be adapted to show quantitative bounds on the delocalisation proven in Theorem~\ref{thm:loop_soft_deloc}.
This appendix contains some details regarding this derivation.
\\

	\paragraph{{\bf Excluding (ExpDec).}}
	\cite[Lemma~5.3]{GlaMan21}, assuming (ExpDec), implies that, for any~$\kappa>0$, there exists~$C_1$ such that
		\[
			\muSpin_{\kappa n,x}^\blackm(\Lambda_n \xleftrightarrow{\blackp} \Lambda_{2n}^c) < e^{-C_1n^c}, \text{ for any } n\geq 1.
		\]
	The proof is a standard application of the FKG inequality for~$\xi^\blackp$ and Corollary~\ref{cor:cbc-lip} (comparison between boundary conditions); see~\cite[Figs.~17,~18]{GlaMan21}.
	Further, \cite[Lemma~5.4]{GlaMan21}, assuming (ExpDec), shows that, for any~$\kappa>0$, there exists~$C_2$ such that
	\[
		\muSpin_{\kappa n,x}^\blackm(\Lambda_n \nxlra{\blackm} \Lambda_{2n}^c) < e^{-C_2n^c}, \text{ for any } n\geq 1.
	\]
	For contradiction, we assume the above does not hold.
	Then, we consider~$C_2$ that is small compared to~$C_1$ and get a stretched-exponential lower bound for the event that neither~$\xi^\blackp$ nor~$\xi^\blackm$ contains a crossing from~$\Lambda_n$ to~$\Lambda_{2n}^c$.
	The latter event implies, by super-duality~\eqref{eq:super_duality}, that either~$\Circ^\whitep(n,2n)$ or~$\Circ^\whitem(n,2n)$ occurs.
	Thus, by symmetry, we get a lower bound for~$\muSpin_{\Lambda_{\kappa n}}^\blackm(\Circ^\whitep(n,2n))$.
	We then explore the outermost circuit of~$\xi^\whitep$, use the domain Markov property and apply the same bound with inverted colours: we construct a circuit of~$\xi^\blackp$ under~$\blackm$ b.c. at a stretched-exponential cost, which contradicts~(ExpDec).

	Using again symmetries of the lattice and the FKG inequality, one constructs infinitely many circuits of~$\xi^\blackm$~(\cite[Lemma~5.5]{GlaMan21}): (ExpDec) implies that, for any~$\eps>0$, there exists~$n_1\in\N$ such that
	\[
		\muSpin_x^\blackm(\bigcap_{j\geq 1} \Circ^\blackm(2^jn_1,2^{j+1}n_1)] > 1-\eps.
	\]
	Since~$\muSpin_x^\blackp = \muSpin_x^\blackm$, we can construct the circuits of~$\xi^\blackp$ and~$\xi^\blackm$ simultaneously at all scales. This contradicts~(ExpDec).\\

	\paragraph{{\bf Logarithmic fluctuations via~\eqref{eq:renorm-ineq-lip}.}}
	By the above, we deduce from~\eqref{eq:renorm-ineq-lip} that~$\inf_n \alpha_n > 0$.
	It is then standard (\cite[Corollary~5.2]{GlaMan21}) that, for all~$a>1$, there exists~$c>0$ such that
	\begin{equation}\label{eq:circuits-every-scale-lip}
		\muSpin_{\Lambda_{an}}^\blackm(\Circ^\blackp(n,an)) \geq c, \text{ for all } n\geq \tfrac{3}{a-1}.
	\end{equation}
	Exploring the outermost circuit~$\Gamma_1$ of~$\xi^\blackp$ and using the domain Markov property allows to construct inside~$\Gamma_1$ a circuit~$\Gamma_2$ of~$\xi^\blackm$ that surrounds~$\Lambda_{n/a}$.
	This implies existence of a loop of~$\omega[\sigma^\black]$ between~$\Gamma_1$ and~$\Gamma_2$.
	Iterating this procedure and using that it is independent for different scales, we obtain that
	\[
		\E_{\domain,2,x}(\omega) \geq c\log n, \text{ for every domain }\domain\supseteq \Lambda_n.
	\]
	The upper bound requires an additional argument: we bound the probability to have two loops in an annulus.
	Indeed, conditioned on~$\omega=\omega[\sigma^\black]\cup\omega[\sigma^\white]$, every loop of~$\omega$ is in~$\omega[\sigma^\black]$ or in~$\omega[\sigma^\white]$ independently with probability~$1/2$.
	Thus, it is enough to upper bound the probability to have at least two loops in an annulus such that the outermost circuit~$\Gamma_1$ is in~$\omega[\sigma^\white]$ and the second-outermost circuit~$\Gamma_2$ is in~$\omega[\sigma^\black]$.
	By definition, $\Gamma_1\cap \upvert(\hexlattice) \subseteq \xi^\blackp$ or~$\Gamma_1\cap \upvert(\hexlattice) \subseteq \xi^\blackm$.
	Also, $\Gamma_2$ is the domain wall between a circuit of pluses of~$\sigma^\black$ and a circuit of minuses of~$\sigma^\black$.
	Thus, we have a circuit of pluses of~$\sigma^\black$ inside a circuit of~$\xi^\blackm$ (or the same but with reversed signs) and both circuits are in a given annulus.
	It is standard that~\eqref{eq:circuits-every-scale-lip} implies that the probability of such event is uniformly below~$1$ and we obtain
	\[
		\E_{\domain,2,x}(\omega) \leq C\log n, \text{ for every domain }\domain\subset \Lambda_n.
	\]
	Together with~\eqref{eq:var-via-loops-lip}, this gives the desired bounds on the variance of a random Lipschitz function;
	see the proofs of \cite[Theorems~1.1,1.2]{GlaMan21} for more details.\\

	\paragraph{{\bf RSW theory.}}
	Hence, it is enough to prove the renormalisation inequality~\eqref{eq:renorm-ineq-lip}.
	At the core of the proof is an RSW estimate~\cite[Proposition~4.7]{GlaMan21}: under symmetric~$\xi^\blackp/\xi^\blackm$ boundary conditions, we obtain a lower bound for the probability that a rectangle is horizontally crossed by~$\xi^\blackp$ in terms of the probability that a slightly wider rectangle is vertically crossed by~$\xi^\blackm$.
	The RSW is stated on a cylinder, since the proof relies on horizontal translations.
	For~$M,N\geq 1$, we define~$\Rect_{M,N}$ as the domain consisting of faces of~$\hexlattice$ whose centres are inside~$[-M,M-1/2]\times[0,N]$.
	Identifying the left and right boundaries of~$\Rect_{M,N}$, we obtain a cylinder that we denote by~$\Cyl_{M,N}$.
	We extend the definition of~$\muSpin_{\domain,x}$ to the cylinder in a natural way and define
	\[
		\muSpin_{\Cyl_{M,N}}^{\blackp/\blackm} :=
		\muSpin_{\Cyl_{M,N},x}(\blank \mid
		\sigma^\black|_{\text{top of } \Cyl_{M,N}} \equiv +,
		\sigma^\black|_{\text{bottom of } \Cyl_{M,N}} \equiv -).
	\]
	It is easy to see (\cite[Lemma~4.3]{GlaMan21}) that~$\muSpin_{\Cyl_{M,N}}^{\blackp/\blackm}$ conditioned on having a straight vertical top-bottom crossing of~$\xi^\blackp$ corresponds to the spin measure on~$\Rect_{M,N}$ under Dobrushin b.c.: $\sigma^\black\equiv -$ on the bottom side and~$\sigma^\black\equiv +$ on all other sides. We denote this measure by~$\muSpin_{\Rect_{M,N}}^{\blackp/\blackm}$.
	Also, the FKG inequality for triplets can be extended to~$\muSpin_{\Cyl_{M,N}}^{\blackp/\blackm}$ in a straightforward manner (\cite[Lemma~4.4]{GlaMan21}).
	Let~$\calC_h^\blackp([a,b]\times [c,d])$ (resp.~$\calC_v^\blackp([a,b]\times [c,d])$) denote the event of existence of a horizontal (resp. vertical) crossing of the rectangle~$[a,b]\times [c,d]$ by~$\xi^\blackp$.
	\cite[Proposition~4.7]{GlaMan21} states the following:
	there exists a function~$\psi:(0,1]\to (0,1]$ such that, for all~$N,M,n,k\geq 1$ with~$n+k\leq N$ and~$2n<M$,
	\begin{equation}\label{eq:RSW-lip}
		\muSpin_{\Cyl_{M,N}}^{\blackp/\blackm}(\calC_h^\blackp([-2n,2n]\times [k,k+n])\geq
		\psi(\muSpin_{\Cyl_{M,N}}^{\blackp/\blackm}(\calC_v^\blackp([-3n,3n]\times [k,k+n])).
	\end{equation}
	The main idea of the proof is reminiscent of the arguments for the random-cluster model~\cite{BefDum12,DumSidTas17}:
	$(i)$ find two vertical~$\xi^\blackp$-crossings that are not far from each other; $(ii)$ prove that they are connected to each other by a~$\xi^\blackp$-path; $(iii)$ this gives arcs of~$\xi^\blackp$ that connect points on the bottom side of the rectangle at a linear distance from each other; $(iv)$ use these arcs as building blocks to obtain a long horizontal crossing via the FKG inequality.
	\cite[Lemmata~4.7-4.11]{GlaMan21} prove~\eqref{eq:RSW-lip} in several pathological cases and reduce the problem to the case when vertical~$\xi^\blackp$-crossings are ``well-behaved'': in short, this says that the drift of the top part of any vertical $\xi^\blackp$-crossing compared to its bottom is small;  the middle part of any~$\xi^\blackp$-path does fluctuate but not too much.
	The proofs are standard and rely only on the FKG inequality for~$\xi^\blackp$, invariance to horizontal shifts and symmetries of the lattice.
	Relying on these properties and on the fact that the vertical crossings are well-behaved, one readily accomplishes the items~$(i)$, $(iii)$ and~$(iv)$ mentioned above.
	The heart of the proof is in step~$(ii)$, which corresponds to~\cite[Lemma~4.12]{GlaMan21} and describes how to connect two vertical $\xi^\blackp$-crossings.
	Let~$\Gamma_L$ and~$\Gamma_R$ be respectively the left-most and the right-most vertical $\xi^\blackp$-crossings that start and end in~$(0,\eps n)\times \{n\}$ and end in~$(0,\eps n)\times \{n+k\}$.

	{\bf Step~1:} Connect the bottom (and top) parts of~$\Gamma_L$ and~$\Gamma_R$ by a path pluses of~$\sigma^\black$.

	Indeed, since~$\Gamma_L$ and~$\Gamma_R$ are well-behaved, their bottom parts are close to each other.
	Then, we can find a domain~$\calD$ that is contained between~$\Gamma_R$ and~$\Gamma_L$ and has boundary points~$A,B,C,D$ such that:
	$AC$ is a vertical interval; $B$ and~$D$ are symmetric with respect to~$AC$; the boundary arcs~$(AB)$ and~$(CD)$ of~$\calD$  are in~$\Gamma_L$ and~$\Gamma_R$ respectively.
	By Corollary~\ref{cor:cbc-lip}, the measure in~$\calD$ stochastically dominates the measure with~$\blackp$ conditions on the arcs~$(AB),(CD)$ and~$\blackm$ conditions on the arcs~$(AD),(CB)$.
	By duality, deterministically, either~$(AB)\xleftrightarrow{\sigma^\black=+} (CD)$ or~$(AD)\xleftrightarrow{\sigma^\black=-} (CB)$.
	By symmetry, the former event has a uniformly positive probability under~$\muSpin_\calD^{\blackp/\blackm}$, and hence also in the original measure.

	{\bf Step~2:} Connect~$\Gamma_L$ and~$\Gamma_R$ by a $\xi^\blackp$-path.

	Assume the bottom and top parts of~$\Gamma_L$ and~$\Gamma_R$ are connected to each other by pluses of~$\sigma^\black$.
	Explore the exterior-most such crossings and denote them by~$\chi_1$ and~$\chi_2$.
	Using reflections and pushing away the $\xi^\blackp$-boundary conditions as on~\cite[Fig.~13]{GlaMan21}, we obtain a  planar domain~$\calD$ with boundary points~$A,B,C,D$ with the same symmetry properties as in Step~1 and such that:
	the boundary of~$\calD$ consists of arcs belonging to~$\xi^\blackp$ and paths~$\chi_1$ and~$\chi_2$;
	any crossing from~$(AB)$ to~$(CD)$ necessarily contains a crossing from~$\Gamma_L$ to~$\Gamma_R$.
	By Corollary~\ref{cor:cbc-lip}, the distribution of~$\xi^\blackp$ on~$\calD$ dominates the one under boundary conditions given by a setting~$\sigma^\black \equiv +$ at all boundary faces of~$\calD$ and setting~$\sigma^\black \equiv -$ everywhere outside.
	It is easy to see that this measure coincides with~$(\muSpin_{\calD,x}^{\blackp\,\whitep} + \muSpin_{\calD,x}^{\blackp\,\whitem})/2$.
	By the super-duality~\eqref{eq:super_duality} of~$\xi^\black$ and~$\xi^\white$ and using the same argument as for the square-crossing property (Lemma~\ref{lemma:loop_rhombus_crossing}), we show that~$\xi^\blackp$ contains a crossing from~$(AB)$ and~$(CD)$ (and hence from~$\Gamma_L$ to~$\Gamma_R$) with a uniformly positive probability.\\

	\paragraph{{\bf Pushing lemma.}}
	The last ingredient in the proof of the renormalisation inequality~\eqref{eq:renorm-ineq-lip} (and hence a dichotomy statement)
	is the pushing lemma~\cite[Corollary~4.14]{GlaMan21}: for any~$C_h\geq 3$ and~$C_v\geq 1$, there exists~$\delta>0$ such that, for any~$n\geq 1$,
	\begin{equation}\label{eq:pushing-lip}
		\muSpin_{\Rect_{C_hn, C_vn}}^{\blackp/\blackm}(\calC_h^\blackp(\Rect_{C_hn,n})) >\delta.
	\end{equation}
	The derivation of~\eqref{eq:renorm-ineq-lip} from this follows closely the original argument for the random-cluster model~\cite{DumSidTas17} and we do not provide further details; see~\cite[Section~4.4]{GlaMan21}.

	Our last goal is to derive~\eqref{eq:pushing-lip} from the RSW estimate~\eqref{eq:RSW-lip}.
	To make use of the latter, we need an input~\cite[Proposition~4.13]{GlaMan21} stating that at least some crossings exist with positive probability: for any~$C\geq 3$, there exists~$\delta>0$ such that, for all~$n\geq 1$,
	\begin{equation}
	\label{eq:input-for-rsw-lip}
		\begin{aligned}
			\text{either }
			\muSpin_{\Cyl_{Cn, 5n}}^{\blackp/\blackm}(\calC_h^\blackp([-2n,2n]\times [3n,4n])) &> \delta, \\
			\text{ or }
			\muSpin_{\Cyl_{Cn, 5n}}^{\blackp/\blackm}(\calC_v^\blackp([-3n,3n]\times [3n,4n])) &> \delta.
		\end{aligned}
	\end{equation}
	We first show how this and the RSW estimate~\eqref{eq:RSW-lip} imply the pushing lemma~\eqref{eq:pushing-lip}.
	Indeed, if either of the two alternatives in~\eqref{eq:input-for-rsw-lip} occurs, the RSW~\eqref{eq:RSW-lip} gives that
	\[
		\muSpin_{\Cyl_{Cn, 5n}}^{\blackp/\blackm}(\calC_h^\blackp([-2n,2n]\times [3n,4n]) > \delta',
	\]
	for some~$\delta'$ that depends only on~$C$.
	We can then lengthen the crossing (\cite[Corollary~4.6]{GlaMan21}) by glueing together two horizontal crossings long overlapping rectangles.
	The only subtlety here is that the glueing is done by crossing a parallelogram~$3n\times n$ in a short direction under~$\blackp$ boundary conditions via a two-step procedure (as in the proof of the RSW): first use a standard duality to construct two vertical crossings of pluses of~$\sigma$; then push away the~$\blackp$ boundary conditions and apply duality~\eqref{eq:super_duality} to get a $\xi^\blackp$-crossing.
	Iterating the procedure, the crossings in cylinders can be made as long as needed and then it remains to note that~$\muSpin_{\Rect_{C_hn, C_vn}}^{\blackp/\blackm}$ stochastically dominates~$\muSpin_{\Cyl_{C_hn, C_vn}}^{\blackp/\blackm}$.

	We now focus on~\eqref{eq:input-for-rsw-lip}.
	Let~$\Strip_j$, for~$j=1,\dots,5$ denote the part of~$\Cyl_{Cn,5n}$ between heights~$(j-1)n$ and~$jn$.
	The proof is based on a duality derived from~\eqref{eq:super_duality}: either~$\Strip_2$ is crossed from top to bottom by~$\xi^\blackp$ or~$\Strip_2$ contains a circuit of~$\xi^\white$ that goes around the cylinder.
	Assume the former event occurs.
	Then,
	\[
		\text{either~$\calC_h^\blackp([kn-2n,kn+2n]\times [n,2n])$ or~$\calC_v^\blackp([kn-3n,kn+3n]\times [n,2n])$ occurs,}
	\]
	for some integer~$k\in [0,C]$.
	By invariance to horizontal shifts, we get a lower bound on~$\muSpin_{\Cyl_{Cn, 5n}}^{\blackp/\blackm}(\calC_h^\blackp([-2n,2n]\times [n,2n])$ or on~$\muSpin_{\Cyl_{Cn, 5n}}^{\blackp/\blackm}(\calC_v^\blackp([-3n,3n]\times [n,2n])$.
	We now embed~$\Cyl_{Cn,5n}$ in an infinite cylinder and push below by~$2n$ the~$\blackp/\blackm$ boundary conditions.
	By Corollary~\ref{cor:cbc-lip}, this increases the measure in~$\Strip_2$ but the resulting measure is exactly (the shift of) the restriction of $\muSpin_{\Cyl_{Cn, 5n}}^{\blackp/\blackm}$ to~$\Strip_4$, whence~\eqref{eq:input-for-rsw-lip} follows.

	It remains to consider the case~$\Strip_2$ contains a circuit of~$\xi^\white$ that goes around the cylinder.
	By symmetry, we can assume that there is a circuit of~$\xi^\whitep$.
	We explore the lowermost such circuit and then push the top~$\blackp$ boundary conditions to form a symmetric cylindrical domain.
	Note that~$\Strip_4$ (on which our events of interest are supported) is exactly at the axis of symmetry.
	By the super-duality~\eqref{eq:super_duality}, at least one of the following events occurs:
	$\calC_h^\blackp([-2n,2n]\times [3n,4n]$, $\calC_h^\blackm([-2n,2n]\times [3n,4n]$,
	$\calC_v^\whitep([-2n,2n]\times [3n,4n]$, $\calC_v^\whitem([-2n,2n]\times [3n,4n]$.
	Then~\eqref{eq:input-for-rsw-lip} follows by symmetry and Corollary~\ref{cor:cbc-lip}.

\bibliographystyle{amsalpha}
\bibliography{biblicomplete.bib}

\end{document}